\documentclass{amsart}

\title{De Rham comparison and Poincar\'e duality for rigid varieties}

\author{Kai-Wen Lan, Ruochuan Liu, and Xinwen Zhu}
\address{University of Minnesota, 127 Vincent Hall, 206 Church Street SE, Minneapolis, MN 55455, USA}
\email{kwlan@math.umn.edu}
\address{Beijing International Center for Mathematical Research, Peking University, 5 Yi He Yuan Road, Beijing 100871, China}
\email{liuruochuan@math.pku.edu.cn}
\address{California Institute of Technology, 1200 East California Boulevard, Pasadena, CA 91125, USA}
\email{xzhu@caltech.edu}
\thanks{K.-W.\@\xspace Lan was partially supported by the National Science Foundation under agreement No.\@\xspace DMS-1352216, and by a Simons Fellowship in Mathematics.  R.\@\xspace Liu was partially supported by the National Natural Science Foundation of China under agreement Nos.\@\xspace NSFC-11571017 and NSFC-11725101, and by the Tencent Foundation.  X.\@\xspace Zhu was partially supported by the National Science Foundation under agreement Nos.\@\xspace DMS-1602092 and DMS-1902239, and by an Alfred P.\@\xspace Sloan Research Fellowship.  Any opinions, findings, and conclusions or recommendations expressed in this writing are those of the authors, and do not necessarily reflect the views of the funding organizations.}

\usepackage{amsfonts, latexsym, amssymb, mathrsfs}

\usepackage[driverfallback=hypertex, unicode]{hyperref}
\usepackage{color}

\usepackage{upref}
\usepackage[shortlabels]{enumitem}
\usepackage{verbatim}
\usepackage{xspace}

\usepackage[cmtip, all]{xy}

\begingroup\expandafter\expandafter\expandafter\endgroup
\expandafter\ifx\csname IncludeInRelease\endcsname\relax
\usepackage{fixltx2e}
\fi

\newtheorem{thm}[equation]{{Theorem}}
\newtheorem{cor}[equation]{{Corollary}}
\newtheorem{lem}[equation]{{Lemma}}
\newtheorem{prop}[equation]{{Proposition}}

\newtheorem{constr}[equation]{{Construction}}

\theoremstyle{definition}
\newtheorem{defn}[equation]{Definition}
\newtheorem{rk}[equation]{{Remark}}

\newcommand{\quash}[1]{}

\newcommand{\Ainf}{A_{\inf}}
\newcommand{\Binf}{B_{\inf}}
\newcommand{\AAinf}{\bA_{\inf}}
\newcommand{\AAinfX}[1]{\bA_{\inf, {#1}}}
\newcommand{\BBinf}{\bB_{\inf}}
\newcommand{\BBinfX}[1]{\bB_{\inf, {#1}}}

\newcommand{\BdR}{B_\dR}
\newcommand{\BdRp}{B_\dR^+}

\newcommand{\BBdRpX}[1]{\bB_{\dR, {#1}}^+}
\newcommand{\BBdR}{\bB_\dR}
\newcommand{\BBdRX}[1]{\bB_{\dR, {#1}}}

\newcommand{\OBdlpX}[1]{\cO\bB_{\dR, \log, {#1}}^+}
\newcommand{\OBdl}{\cO\bB_{\dR, \log}}
\newcommand{\OBdlX}[1]{\cO\bB_{\dR, \log, {#1}}}
\newcommand{\OCl}{\cO\bC_{\log}}
\newcommand{\OClX}[1]{\cO\bC_{\log, {#1}}}

\newcommand{\DdR}{D_\dR}
\newcommand{\DdRalg}{D_\dR^\alg}
\newcommand{\Ddl}{D_{\dR, \log}}
\newcommand{\Ddlalg}{D_{\dR, \log}^\alg}

\newcommand{\ho}{\widehat{\otimes}}

\newcommand{\GrSh}[1]{\underline{#1}}

\newcommand{\bA}{\mathbb{A}}
\newcommand{\bB}{\mathbb{B}}
\newcommand{\bC}{\mathbb{C}}
\newcommand{\bD}{\mathbb{D}}
\newcommand{\bE}{\mathbb{E}}
\newcommand{\bF}{\mathbb{F}}

\newcommand{\bL}{\mathbb{L}}

\newcommand{\bQ}{\mathbb{Q}}
\newcommand{\bR}{\mathbb{R}}

\newcommand{\bT}{\mathbb{T}}

\newcommand{\bZ}{\mathbb{Z}}

\newcommand{\cB}{\mathcal{B}}

\newcommand{\cD}{\mathcal{D}}
\newcommand{\cE}{\mathcal{E}}
\newcommand{\cF}{\mathcal{F}}

\newcommand{\cL}{\mathcal{L}}
\newcommand{\cM}{\mathcal{M}}

\newcommand{\cO}{\mathcal{O}}

\newcommand{\cR}{\mathcal{R}}

\newcommand{\cU}{\mathcal{U}}

\newcommand{\cX}{\mathcal{X}}

\newcommand{\OP}[1]{\operatorname{#1}}

\newcommand{\Ext}{\OP{Ext}}
\newcommand{\Hom}{\OP{Hom}}
\newcommand{\Spec}{\OP{Spec}}
\newcommand{\Spa}{\OP{Spa}}

\newcommand{\Em}{\hookrightarrow}                   
\newcommand{\Surj}{\twoheadrightarrow}              
\newcommand{\Mi}{\stackrel{\sim}{\to}}              
\newcommand{\Mapn}[1]{\stackrel{#1}{\to}}           
\newcommand{\Emn}[1]{\stackrel{#1}{\Em}}            

\newcommand{\can}{\Utext{can.}}                     

\newcommand{\bAi}{{\bA_f}}                          

\newcommand{\Grp}[1]{\mathrm{#1}}                   

\newcommand{\Gm}[1]{\mathbf{G}_{\Utext{m}, {#1}}}   

\newcommand{\chr}{\OP{char}}                        
\newcommand{\rank}{\OP{rk}}                         

\newcommand{\Shdom}{\mathsf{X}}                     
\newcommand{\Lquot}{\backslash}                     
\newcommand{\hd}{h}                                 
\newcommand{\hc}{\mu}                               

\newcommand{\Gal}{\OP{Gal}}                         
\newcommand{\res}{\OP{res}}                          

\newcommand{\Image}{\OP{Im}}                        

\newcommand{\alg}{\Utext{alg}}                      
\newcommand{\arith}{\Utext{arith}}                  
\newcommand{\geom}{\Utext{geom}}                    
\newcommand{\gp}{\Utext{gp}}                        
\newcommand{\red}{\Utext{red}}                      

\newcommand{\Talg}[1]{\langle{#1}\rangle}           

\newcommand{\ReFl}{E}                               
\newcommand{\Sh}{\mathrm{Sh}}                       
\newcommand{\Model}{\Sh}                            
\newcommand{\levcp}{K}                              

\newcommand{\an}{\Utext{an}}                        

\newcommand{\dualsign}{{\vee}}                      
\newcommand{\dual}[1]{{#1}^\dualsign}               
\newcommand{\Ex}{\wedge}                            

\newcommand{\Tor}{\Utext{tor}}                      
\newcommand{\Min}{\Utext{min}}                      

\newcommand{\Torcpt}[1]{{#1}^\Tor}                  
\newcommand{\Mincpt}[1]{{#1}^\Min}                  
\newcommand{\NCD}{D}                                

\newcommand{\Mod}{\mathrm{Mod}}                     
\newcommand{\Rep}{\mathrm{Rep}}                     
\newcommand{\rep}{V}                                
\newcommand{\repalt}{W}                             
\newcommand{\Shv}{\mathrm{Sh}}                      

\newcommand{\wt}{\lambda}                           
\newcommand{\wtalt}{\nu}                            
\newcommand{\rt}{\alpha}                            
\newcommand{\cort}{\dual{\alpha}}                   
\newcommand{\hsum}{\rho}                            

\newcommand{\RT}{\Phi}                              
\newcommand{\WT}{\OP{X}}                            

\newcommand{\WG}{\OP{W}}                            
\newcommand{\wl}{l}                                 

\newcommand{\B}{\Utext{B}}                          
\newcommand{\dR}{\Utext{dR}}                        
\newcommand{\Hdg}{\Utext{Hodge}}                    
\newcommand{\Hi}{\Utext{Higgs}}                     
\newcommand{\coh}{\Utext{coh}}                      
\newcommand{\cpt}{\Utext{c}}                        
\newcommand{\intcoh}{\Utext{int}}                   
\newcommand{\Sc}{\Utext{$\star$-c}}
\newcommand{\Scalt}{\Utext{$\circ$-c}}
\newcommand{\Snc}{\Utext{$\star$-nc}}
\newcommand{\Sncalt}{\Utext{$\circ$-nc}}

\newcommand{\ideal}[1]{\mathfrak{#1}}               

\newcommand{\Fil}{\Utext{Fil}}                      
\newcommand{\gr}{\OP{gr}}

\newcommand{\et}{\Utext{\'et}}                      
\newcommand{\ket}{\Utext{k\'et}}                    
\newcommand{\proet}{\Utext{pro\'et}}                
\newcommand{\proket}{\Utext{prok\'et}}              

\newcommand{\BSh}[1]{{}_{\B}\GrSh{#1}}              
\newcommand{\dRSh}[1]{{}_{\dR}\GrSh{#1}}            
\newcommand{\etSh}[1]{{}_{\et}\GrSh{#1}}            
\newcommand{\cohSh}[1]{{}_\coh\GrSh{#1}}            

\newcommand{\canext}{{\Utext{can}}}                 
\newcommand{\subext}{{\Utext{sub}}}                 

\newcommand{\AC}[1]{\overline{#1}}                  
\newcommand{\ACMap}{\iota}                          

\newcommand{\BFp}{k}                                
\newcommand{\Coef}{L}                               
\newcommand{\CoefMap}{\tau}                         

\newcommand{\RH}{\mathcal{RH}}                      
\newcommand{\RHl}{\mathcal{RH}_{\log}}              
\newcommand{\DRl}{\mathit{DR}_{\log}}               
\newcommand{\Hc}{\mathcal{H}}                       
\newcommand{\Hl}{\mathcal{H}_{\log}}                
\newcommand{\Hil}{\mathit{Higgs}_{\log}}            
\newcommand{\BGGl}{\mathit{BGG}_{\log}}             

\newcommand{\IH}{\mathit{IH}}                       

\newcommand{\Utext}[1]{\text{\rm #1}}               
\newcommand{\Refenum}[1]{\Pth{\textrm{#1}}}
\newcommand{\Refeq}[1]{\Pth{#1}}

\newcommand{\Pth}[1]{{\rm (}#1{\rm )}}              
\newcommand{\Qtn}[1]{``#1''}                        

\newcommand{\parenthesis}[1]{\Pth{#1}}              

\newcommand{\apage}{p.\@\xspace}                    

\newcommand{\aCh}{Ch.\@\xspace}                     
\newcommand{\aSec}{Sec.\@\xspace}                   
\newcommand{\aSecs}{Sec.\@\xspace}                  

\newcommand{\aDef}{Def.\@\xspace}                   
\newcommand{\aLem}{Lem.\@\xspace}                   
\newcommand{\aLems}{Lem.\@\xspace}                  
\newcommand{\aProp}{Prop.\@\xspace}                 
\newcommand{\aThm}{Thm.\@\xspace}                   
\newcommand{\aThms}{Thm.\@\xspace}                  
\newcommand{\aCor}{Cor.\@\xspace}                   
\newcommand{\aCors}{Cor.\@\xspace}                  
\newcommand{\aRem}{Rem.\@\xspace}                   
\newcommand{\aEx}{Ex.\@\xspace}                     
\newcommand{\aConj}{Conj.\@\xspace}                 

\newcommand{\resp}{resp.\@\xspace}                  
\newcommand{\ie}{i.e.\@\xspace}                     
\newcommand{\eg}{e.g.\@\xspace}                     
\newcommand{\etc}{etc\xspace}                       

\newcommand{\Refcf}{cf.\@\xspace}                   

\newcommand{\logadicdeflogstr}{2.2.2} 
\newcommand{\logadicdeflogstrseven}{7}
\newcommand{\logadicdefimm}{2.2.23} 
\newcommand{\logadicexlogadicspncd}{2.3.17} 
\newcommand{\logadicexlogadicspncdstrictclimm}{2.3.18} 
\newcommand{\logadicexlogdiffsheafncd}{3.3.20} 
\newcommand{\logadicpropAbhyankar}{4.2.1} 
\newcommand{\logadiclemAbhyankarbasic}{4.2.5} 
\newcommand{\logadiclemloggeompt}{4.4.4} 
\newcommand{\logadiclemexc}{4.5.3} 
\newcommand{\logadiclemclimmketmor}{4.5.4} 
\newcommand{\logadiclemclimmOplusp}{4.5.7} 
\newcommand{\logadiclemketmorOplusp}{4.5.8} 
\newcommand{\logadicthmpurity}{4.6.1} 
\newcommand{\logadiclemkettoetconst}{4.6.2} 
\newcommand{\logadicsecproket}{5} 
\newcommand{\logadicpropproketsiteqcqs}{5.1.5} 
\newcommand{\logadicpropproketvsket}{5.1.6} 
\newcommand{\logadicpropproketvsketadj}{5.1.7} 
\newcommand{\logadiclemlogaffperfclimm}{5.3.7} 
\newcommand{\logadicproplogaffperfbasis}{5.3.12} 
\newcommand{\logadicdefproketsheaves}{5.4.1} 
\newcommand{\logadicthmalmostvanhat}{5.4.3} 
\newcommand{\logadicsectoricchart}{6.1} 
\newcommand{\logadicthmprimcomp}{6.2.1} 
\newcommand{\logadicdefketlisse}{6.3.1} 
\newcommand{\logadiccorpuritylisse}{6.3.4} 

\newcommand{\logRHthmintromain}{1.1} 
\newcommand{\logRHexlogadicspncd}{2.1.2} 
\newcommand{\logRHsecOBdl}{2.2} 
\newcommand{\logRHdefOBdl}{2.2.10} 
\newcommand{\logRHsecOBdlexplicit}{2.3} 
\newcommand{\logRHeqzeta}{2.3.1} 
\newcommand{\logRHeqchoicet}{2.3.2} 
\newcommand{\logRHpropOBdlploc}{2.3.15} 
\newcommand{\logRHcorOBdlplocgr}{2.3.17} 
\newcommand{\logRHcorOBdRplocclimm}{2.3.20} 
\newcommand{\logRHcorlogdRcplx}{2.4.2} 
\newcommand{\logRHcorlogdRcplxone}{1}
\newcommand{\logRHcorlogdRcplxtwo}{2}
\newcommand{\logRHcorlogdRcplxthree}{3}
\newcommand{\logRHeqconnWi}{2.4.4} 
\newcommand{\logRHseclogRH}{3} 
\newcommand{\logRHeqdefcX}{3.1.5} 
\newcommand{\logRHdeflogconnetc}{3.1.7} 
\newcommand{\logRHseclogRHthm}{3.2} 
\newcommand{\logRHthmlogRHgeom}{3.2.3} 
\newcommand{\logRHthmlogRHgeomres}{2} 
\newcommand{\logRHthmlogRHarith}{3.2.7} 
\newcommand{\logRHthmlogRHarithres}{2} 
\newcommand{\logRHthmlogRHarithcomp}{3} 
\newcommand{\logRHthmunipvsnilp}{3.2.12} 
\newcommand{\logRHseccoh}{3.3} 
\newcommand{\logRHpropLOCl}{3.3.3} 
\newcommand{\logRHeqdefgammaj}{3.3.6} 
\newcommand{\logRHremcompatLmzero}{3.3.14} 
\newcommand{\logRHlemLOClcoh}{3.3.15} 
\newcommand{\logRHlemLdescent}{3.3.16} 
\newcommand{\logRHeqresgeneigen}{3.4.2} 
\newcommand{\logRHlemGammageominv}{3.4.3} 
\newcommand{\logRHremcompatLmzerores}{3.4.14} 
\newcommand{\logRHpropconnisomres}{3.4.17} 
\newcommand{\logRHlemDdltoRHlfilstrict}{3.4.18} 
\newcommand{\logRHeqlemDdltoRHlfilstrict}{3.4.19} 
\newcommand{\logRHcorDdltoRHlfilstrict}{3.4.21} 
\newcommand{\logRHcorDdlgrvecbdl}{3.4.22} 
\newcommand{\logRHlemketBdR}{3.6.1} 
\newcommand{\logRHlemlogRHarithcompdR}{3.6.3} 
\newcommand{\logRHlemlogRHarithcompHT}{3.6.4} 
\newcommand{\logRHsecDdRalg}{4.1} 
\newcommand{\logRHthmHTdegencomp}{4.1.4} 
\newcommand{\logRHsecShvar}{5} 
\newcommand{\logRHseclocsystconstr}{5.2} 
\newcommand{\logRHeqmuh}{5.2.9} 
\newcommand{\logRHproplocsystinfty}{5.2.10} 
\newcommand{\logRHrempartialflag}{5.2.11} 
\newcommand{\logRHeqlocsystetcoefbasech}{5.2.12} 
\newcommand{\logRHeqcoefproj}{5.2.14} 
\newcommand{\logRHthmlocsystcomp}{5.3.1} 
\newcommand{\logRHremlocsystHodgedegen}{5.3.5} 

\begin{document}

\begin{abstract}
    Over any smooth algebraic variety over a $p$-adic local field $k$, we construct the de Rham comparison isomorphisms for the \'etale cohomology with partial compact support of de Rham $\bZ_p$-local systems, and show that they are compatible with Poincar\'e duality and with the canonical morphisms among such cohomology.  We deduce these results from their analogues for rigid analytic varieties that are Zariski open in some proper smooth rigid analytic varieties over $k$.  In particular, we prove finiteness of \'etale cohomology with partial compact support of any $\bZ_p$-local systems, and establish the Poincar\'e duality for such cohomology after inverting $p$.
\end{abstract}

\maketitle

\tableofcontents

\numberwithin{equation}{section}

\section{Introduction}\label{sec-intro}

This paper is a sequel to \cite{Diao/Lan/Liu/Zhu:lrhrv}, in which a $p$-adic Riemann--Hilbert functor was constructed as an analogue of Deligne's Riemann--Hilbert correspondence over $\bC$ \Pth{see \cite{Deligne:1970-EDR}}.  We refer to \cite{Diao/Lan/Liu/Zhu:lrhrv} for the general introduction and backgrounds.  In this paper, we further investigate the properties of the $p$-adic Riemann--Hilbert functor.  We establish the de Rham comparison isomorphisms for the cohomology with compact support under the $p$-adic Riemann--Hilbert correspondences, and show that they are compatible with duality.  In particular, we obtain the following theorem \Pth{see Theorems \ref{thm-comp-alg-dR-cpt} and \ref{thm-comp-alg-dR-int} for more complete statements}:
\begin{thm}\label{thm-intro}
    Let $U$ be a smooth algebraic variety over a $p$-adic field $k$ \Pth{see Notation and Conventions}, and let $\bL$ be a de Rham $p$-adic \'etale local system on $U$.  Then there is a canonical comparison isomorphism
    \begin{equation}\label{eq-thm-intro}
        H^i_{\et, \cpt}(U_{\AC{k}}, \bL) \otimes_{\bQ_p} \BdR \cong H^i_{\dR, \cpt}\bigl(U, \DdRalg(\bL)\bigr) \otimes_k \BdR
    \end{equation}
    compatible with the canonical filtrations and the actions of $\Gal(\AC{k} / k)$ on both sides.  Here $\DdRalg$ is the above-mentioned $p$-adic Riemann--Hilbert functor constructed in \cite{Diao/Lan/Liu/Zhu:lrhrv}, and $H^i_{\et, \cpt}$ \Pth{\resp $H^i_{\dR, \cpt}$} denotes the usual \'etale \Pth{\resp de Rham} cohomology with compact support.

    In addition, the above comparison isomorphism \Refeq{\ref{eq-thm-intro}} is compatible with the one in \cite[\aThm \logRHthmintromain]{Diao/Lan/Liu/Zhu:lrhrv} \Pth{for varying $\bL$} in the following sense:
    \begin{enumerate}
        \item\label{enum-compat-int}  The following diagram
            \[
                \xymatrix{ {H^i_{\et, \cpt}(U_{\AC{k}}, \bL) \otimes_{\bQ_p} \BdR} \ar^-\sim[r] \ar[d] & {H^i_{\dR, \cpt}\bigl(U, \DdRalg(\bL)\bigr) \otimes_k \BdR} \ar[d] \\
                {H^i_\et(U_{\AC{k}}, \bL) \otimes_{\bQ_p} \BdR} \ar^-\sim[r] & {H^i_\dR\bigl(U, \DdRalg(\bL)\bigr) \otimes_k \BdR} }
            \]
            is commutative, where the horizontal isomorphisms are the comparison isomorphisms, and where the vertical morphisms are the canonical ones.  The vertical morphisms are strictly compatible with the filtrations.

        \item\label{enum-compat-dual}  When $U$ is of pure dimension $d$, the following diagram
            \[
                \xymatrix{ {H^i_{\et, \cpt}(U_{\AC{k}}, \bL) \otimes_{\bQ_p} \BdR} \ar_-\wr[d] \ar^-\sim[r] & {H^i_{\dR, \cpt}\bigl(U, \DdRalg(\bL)\bigr) \otimes_k \BdR} \ar^-\wr[d] \\
                {\dual{\Bigl(H^{2d - i}_\et\bigl(U_{\AC{k}}, \dual{\bL}(d)\bigr) \otimes_{\bQ_p} \BdR\Bigr)}} \ar^-\sim[r] & {\dual{\Bigl(H^{2d - i}_\dR\bigl(U, \DdRalg(\dual{\bL}(d))\bigr) \otimes_k \BdR\Bigr)}} }
            \]
            is commutative, where the horizontal isomorphisms are given by the comparison isomorphisms, where the duals are with respect to the base field $\BdR$, and where the vertical isomorphisms are given by the usual Poincar\'e duality for \'etale and de Rham cohomology.
   \end{enumerate}
\end{thm}

Although it might seem that a comparison isomorphism as in \Refeq{\ref{eq-thm-intro}} could be easily constructed using the comparison isomorphism in \cite[\aThm \logRHthmintromain]{Diao/Lan/Liu/Zhu:lrhrv} and the Poincar\'e duality for the \'etale and de Rham cohomology of algebraic varieties, in which case the compatibility \Refeq{\ref{enum-compat-dual}} would be tautological, the compatibility \Refeq{\ref{enum-compat-int}} would not be clear.  Therefore, we need a different approach.  We shall first prove such a comparison theorem for \Pth{appropriately defined} cohomology with compact support in the rigid analytic setting \Pth{see Theorems \ref{thm-L-!-coh-comp} and \ref{thm-int-coh-comp}}, using the log Riemann--Hilbert correspondence introduced in \cite{Diao/Lan/Liu/Zhu:lrhrv} and further developed in this paper, and show that the comparison isomorphisms indeed satisfy the desired compatibilities \Refeq{\ref{enum-compat-int}} and \Refeq{\ref{enum-compat-dual}}.  After that, we obtain the comparison theorem in the algebraic setting by GAGA \cite{Kopf:1974-efava} and the comparison results in \cite{Huber:1996-ERA}.

Given the general theory developed in \cite{Diao/Lan/Liu/Zhu:lrhrv}, the main new ingredient is the definition of a period sheaf that works for the cohomology with compact support.  To give a flavor of what it looks like, consider the simplest situation where $U$ is a smooth rigid analytic variety that admits a smooth compactification $X$ such that $U = X - D$ for some smooth divisor $D$.  Then we equip $X$ with the log structure defined by $D$ \Pth{as in \cite[\aEx \logadicexlogadicspncd]{Diao/Lan/Liu/Zhu:lasfr}}, and equip $D$ with the pullback of the log structure of $X$ along the closed immersion $\imath: D \to X$ \Pth{as in \cite[\aEx \logadicexlogadicspncdstrictclimm]{Diao/Lan/Liu/Zhu:lasfr}}.  We emphasize that the log structure of $D$ is nontrivial, and that it is crucial to equip $D$ with such a log structure.  For this reason, we denote $D$ with this nontrivial log structure by $D^\partial$.  Then the \Qtn{correct} period sheaf for our purpose is the sheaf
\[
    \ker\bigl(\OBdlX{X} \to \imath_{\proket, *}(\OBdlX{D^\partial})\bigr)
\]
on $X_\proket$, the pro-Kummer \'etale site of $X$, where $\OBdlX{X}$ and $\OBdlX{D^\partial}$ are the period sheaves on $X_\proket$ and $D^\partial_\proket$, respectively, as in \cite[\aDef \logRHdefOBdl]{Diao/Lan/Liu/Zhu:lrhrv}.  Note that, in general, this is \emph{not} the same as the naive $!$-pushforward to $X_\proket$ of the period sheaf on $U_\proet$.  Once the period sheaf is constructed, the remaining arguments follow similar strategies as in \cite{Diao/Lan/Liu/Zhu:lrhrv}, sometimes with generalizations.

As an application of the methods developed in the proof of Theorem \ref{thm-intro}, we obtain a version of Poincar\'e duality for the \Pth{rational} $p$-adic \'etale cohomology of smooth rigid analytic varieties \Pth{see Theorem \ref{thm-trace-et} for more complete statements}:
\begin{thm}\label{thm-intro-trace-et}
    Suppose that $U$ is a smooth rigid analytic variety over $k$ of pure dimension $d$ that is of the form $U = X - Z$, where $X$ is a proper rigid analytic variety over $k$, and where $Z$ is a closed rigid analytic subvariety of $X$.  Then there is a canonical trace morphism
    \[
        t_\et: H_{\et, \cpt}^{2d}\bigl(U_{\AC{k}}, \bQ_p(d)\bigr) \to \bQ_p,
    \]
    whose formation is compatible with the excision and Gysin isomorphisms defined by complements of smooth divisors.  In addition, for each $\bZ_p$-local system $\bL$ on $U_\et$ \Pth{which is not necessarily de Rham}, with $\bL_{\bQ_p} := \bL \otimes_{\bZ_p} \bQ_p$, we have a canonical perfect pairing
    \[
        H^i_{\et, \cpt}(U_{\AC{k}}, \bL_{\bQ_p}) \otimes_{\bQ_p} H^{2d - i}_\et\bigl(U_{\AC{k}}, \dual{\bL}_{\bQ_p}(d)\bigr) \to \bQ_p,
    \]
    which we call the \emph{Poincar\'e duality pairing}, defined by pre-composing $t_\et$ with the cup product pairing $H^i_{\et, \cpt}(U_{\AC{k}}, \bL_{\bQ_p}) \otimes_{\bQ_p} H^{2d - i}_\et\bigl(U_{\AC{k}}, \dual{\bL}_{\bQ_p}(d)\bigr) \to H_{\et, \cpt}^{2d}\bigl(U_{\AC{k}}, \bQ_p(d)\bigr)$.
\end{thm}
We refer to Definition \ref{def-H-c} for our definition of the $p$-adic \'etale cohomology with compact support for rigid analytic varieties over $k$.  We remark that the Poincar\'e duality we obtained is, essentially by construction, compatible with all the de Rham comparison isomorphisms in \cite{Scholze:2013-phtra}, \cite{Diao/Lan/Liu/Zhu:lrhrv}, and this paper.

We note that the question of Poincar\'e duality for \emph{proper} smooth rigid analytic varieties \Pth{in which case the cohomology with compact support coincides with the usual cohomology} was raised earlier by Scholze in \cite{Scholze:2013-phtra}, and Gabber has announced a proof for such a result using a different method \Pth{see \cite[Appendix to Lecture 10, footnote 4]{Scholze/Weinstein:2020-BLG}}.  Nevertheless, even in the original proper smooth setting in \cite{Scholze:2013-phtra}, our approach makes essential use of the excision and Gysin isomorphisms defined by complements of smooth divisors, and hence crucially depends on the de Rham comparison results in the nonproper setting in \cite{Diao/Lan/Liu/Zhu:lrhrv} and this paper.

We will also study the cohomology with \emph{partial compact support}, as in \cite[\aSec 4.2]{Deligne/Illusie:1987-rdcdr}, \cite[\aSec III]{Faltings:1989-ccpgr}, and \cite{Faltings:2002-aee}; and also some \emph{generalized interior cohomology}, namely, the image of a morphism between cohomology with different partial compact support conditions; and construct de Rham comparison isomorphisms for such cohomology that are also compatible with Poincar\'e duality.

\subsection*{Outline of this paper}

Let us briefly describe the organization of this paper, and highlight the main topics in each section.

In Section \ref{sec-bd}, we work with a rigid analytic variety $U$ that is the open complement in a smooth rigid analytic variety $X$ of a normal crossings divisor $D$ whose intersections of irreducible components define a stratification of $X$ with smooth \Pth{closed} strata, and use such a stratification to study the \'etale cohomology of $U$ with partial compact support.  More specifically, in Section \ref{sec-log-str-bd}, we equip the smooth strata as above with several different log structures.  In Section \ref{sec-loc-syst-bd}, we study the pullbacks to such strata of torsion Kummer \'etale local systems on $X$, and prove the primitive comparison theorem for the Kummer \'etale cohomology of $\bF_p$-local systems of this kind \Pth{see Proposition \ref{prop-L-J-dir-im-et} and Theorem \ref{thm-prim-comp-bd}}.  In Section \ref{sec-ket-coh-cpt}, we define the Kummer \'etale cohomology of torsion local systems on $U$ with partial compact support along some subdivisor $D^\Sc$ of $D$, and prove \Pth{using results in Section \ref{sec-loc-syst-bd}} the primitive comparison theorem for such cohomology \Pth{see Theorem \ref{thm-L-!-prim-comp}}.  In Section \ref{sec-proket-coh-cpt}, we define the pro-Kummer \'etale cohomology of $\widehat{\bZ}_p$-local systems on $U$ with partial compact support along $D^\Sc$, and relate it to the Kummer \'etale cohomology.  In Section \ref{sec-period-bd}, we introduce some variants supported on the boundary strata \Pth{with log structures pulled back from $X$} of the period sheaves in \cite[\aSec \logRHsecOBdl]{Diao/Lan/Liu/Zhu:lrhrv}, and establish some variants of the Poincar\'e lemma for them.

In Section \ref{sec-dR-comp-cpt}, we generalize the results in \cite[\aSec \logRHseclogRHthm]{Diao/Lan/Liu/Zhu:lrhrv} to the \'etale, de Rham, Higgs, and Hodge cohomology with partial compact support.  More specifically, in Section \ref{sec-dR-comp-cpt-main}, we introduce the de Rham, Higgs, and Hodge cohomology with partial compact support, and state the comparison theorem for such cohomology \Pth{see Theorem \ref{thm-L-!-coh-comp}}.  In Sections \ref{sec-period-A-inf-B-inf-cpt}, \ref{sec-period-B-dR-cpt}, and \ref{sec-period-OB-dR-cpt}, we introduce more variants of the period sheaves introduced in \cite[\aSec \logRHsecOBdl]{Diao/Lan/Liu/Zhu:lrhrv}, which are useful for studying the pro-Kummer \'etale cohomology with partial compact support by taking limits and by using the primitive comparison theorem established in Section \ref{sec-ket-coh-cpt}, and prove the Poincar\'e lemma for such variants of period sheaves.  In Section \ref{sec-dR-comp-cpt-proof}, we prove the desired comparison theorem, and provide some criteria for cohomology with different partial compact support conditions to be isomorphic to each other.

In Section \ref{sec-trace}, we construct some trace morphisms for the \'etale and de Rham cohomology with compact support, and show that they define Poincar\'e duality pairings for the \'etale and de Rham cohomology with partial compact support that are compatible with the comparison isomorphisms in Section \ref{sec-dR-comp-cpt}.  More specifically, in Section \ref{sec-trace-coh}, as a foundation for later constructions, we review the trace morphisms and Serre duality for the coherent cohomology of proper smooth rigid analytic varieties.  In Section \ref{sec-trace-dR}, we establish a perfect pairing between Higgs cohomology with complementary partial compact supports \Pth{see Theorem \ref{thm-Higgs-pairing}}; and we construct some trace morphisms for de Rham \Pth{\resp Hodge} cohomology with compact support using the trace morphisms for coherent cohomology, and show that they induce perfect pairings between de Rham \Pth{\resp Hodge} cohomology with complementary partial compact supports, when the coefficients of cohomology are associated with \emph{de Rham} \'etale $\bZ_p$-local systems \Pth{see Theorem \ref{thm-trace-dR-Hdg}}.  In Section \ref{sec-exc-Gysin}, we show that the \'etale and de Rham excision and Gysin isomorphisms defined by complements of smooth divisors are compatible with the de Rham comparison isomorphisms \Pth{see Propositions \ref{prop-exc-dR-comp} and \ref{prop-Gysin-dR-comp}}.  In Section \ref{sec-trace-et}, by using the compatibility results in Section \ref{sec-exc-Gysin}, we construct some trace morphisms for \'etale cohomology with compact support using the trace morphisms for de Rham cohomology constructed in Section \ref{sec-trace-dR}, and show \Pth{by comparison with the above perfect duality for Higgs cohomology} that these trace morphisms induce perfect duality pairings between \'etale cohomology with complementary partial compact supports, when the coefficients are $\bQ_p$-base extensions of $\bZ_p$-local systems \Pth{see Theorem \ref{thm-trace-et}}, which are compatible with the above perfect duality for de Rham cohomology \Pth{via comparison isomorphisms} when the coefficients are de Rham.  In Section \ref{sec-int-coh}, we introduce the notion of generalized interior cohomology, which is the image of a morphism between cohomology with different partial compact support conditions, and deduce from the results in Sections \ref{sec-dR-comp-cpt} and \ref{sec-trace-et} the de Rham comparison and the compatibility with Poincar\'e duality for such cohomology.

In Section \ref{sec-comp-alg}, we deduce the de Rham comparison and the compatibility with Poincar\'e duality for the cohomology with partial compact support and the generalized interior cohomology similarly defined over \emph{algebraic varieties}, by showing that the various constructions are compatible with the analytification functors.

In Section \ref{sec-Sh-var}, we apply the results in Section \ref{sec-comp-alg} to Shimura varieties, in the setting of \cite[\aSec \logRHsecShvar]{Diao/Lan/Liu/Zhu:lrhrv}, and obtain the de Rham comparison and the dual Bernstein--Gelfand--Gelfand \Pth{BGG} decomposition for the cohomology with partial compact support of automorphic local systems \Pth{on the \'etale side} and automorphic bundles \Pth{on the de Rham and coherent sides} on general Shimura varieties.  As a byproduct, we can compute the Hodge--Tate weights of the \'etale cohomology with partial compact support in terms of the dual BGG decomposition.  We also obtained corresponding results for the generalized interior cohomology and, when the coefficients have regular weights, for the intersection cohomology as well.

\subsection*{Acknowledgements}

We would like to thank Shizhang Li and Yoichi Mieda for some helpful discussions, thank Wies{\l}awa Nizio{\l} for sending us a preliminary version of the paper \cite{Colmez/Dospinescu/Hauseux/Niziol:2021-pecpd}, and thank the Beijing International Center for Mathematical Research, the Morningside Center of Mathematics, the California Institute of Technology, and the Academia Sinica for their hospitality.  We would also like to thank the anonymous referee for many helpful comments and suggestions.

\subsection*{Notation and conventions}

We shall follow the notation and conventions of \cite{Diao/Lan/Liu/Zhu:lrhrv}, unless otherwise specified.  In particular, we shall denote by $k$ a nonarchimedean local field \Pth{\ie, a field complete with respect to the topology induced by a nontrivial nonarchimedean multiplicative norm $|\cdot|: k \to \bR_{\geq 0}$} with residue field $\kappa$ of characteristic $p > 0$, and by $\cO_k$ its ring of integers.  Since we will be mainly working with rigid analytic varieties, we shall work with $k^+ = \cO_k$ and regard rigid analytic varieties over $k$ as adic spaces locally topologically of finite type over $\Spa(k, \cO_k)$ \Pth{as in \cite{Huber:1996-ERA}}.  All rigid analytic varieties will be separated.  Group cohomology will always mean continuous group cohomology.  For the sake of simplicity, by a \emph{$p$-adic field}, we shall mean a complete discrete valuation field of mixed characteristic $(0, p)$ with perfect residue field.

\numberwithin{equation}{subsection}

\section{Boundary stratification and cohomology with compact support}\label{sec-bd}

In this section, let $X$ be a smooth rigid analytic variety over $k$, and $D$ a normal crossings divisor \Pth{see \cite[\aEx \logRHexlogadicspncd]{Diao/Lan/Liu/Zhu:lrhrv}} with \Pth{finitely many} irreducible components $\{ D_j \}_{j \in I}$ \Pth{\ie, the images of the connected components of the normalization of $D$, as in \cite{Conrad:1999-icrs}} satisfying the condition that all the intersections
\begin{equation*}\label{eq-XJ}
    X_J := X \cap \bigl(\cap_{j \in J} \, D_j\bigr),
\end{equation*}
where $J \subset I$, are also \emph{smooth}.  \Pth{Note that $X_\emptyset = X$.}

\subsection{Log structures on smooth boundary strata}\label{sec-log-str-bd}

Let us denote by
\[
    \imath_J: X_J \to X
\]
the canonical closed immersion of adic spaces.  Let
\[
    D_J := \cup_{J \subsetneq J' \subset I} \, X_{J'}
\]
\Pth{with its canonical reduced closed subspace structure} and
\[
    U_J := X_J - D_J,
\]
as adic spaces.  \Pth{Note that $D_\emptyset = D$ and $U_\emptyset = U$.}  Then we also have a canonical open immersion of adic spaces
\[
    \jmath_J: U_J \to X_J
\]

For any $I^\Sc \subset I$, with $I^\Snc := I - I^\Sc$, let
\[
    D^\Sc := \cup_{j \in I^\Sc} \, D_j
\]
and
\[
    D^\Snc := \cup_{j \in I^\Snc} \, D_j,
\]
\Pth{with their canonical reduced closed subspace structures}, and let $U^\Sc := X - D^\Sc$ and $U^\Snc := X - D^\Snc$.  Let $\jmath_\Sc: U \to U^\Sc$, $\jmath^\Sc: U^\Sc \to X$, $\jmath_\Snc: U \to U^\Snc$, and $\jmath^\Snc: U^\Snc \to X$ denote the canonical open immersions of adic spaces.  \Pth{In Sections \ref{sec-ket-coh-cpt} and \ref{sec-proket-coh-cpt} below, we will use $\jmath_\Sc$ and $\jmath^\Sc$ to define the Kummer \'etale and pro-Kummer \'etale cohomology of $U$ with partial compact support along $D^\Sc$.}

We shall view $X$ as a log adic space by equipping it with the log structure $\alpha_X: \cM_X \to \cO_X$ defined by $D$ as in \cite[\aEx \logRHexlogadicspncd]{Diao/Lan/Liu/Zhu:lrhrv}, together with a canonical morphism of sites
\[
    \varepsilon_\et: X_\ket \to X_\et.
\]

For each $J \subset I$,  the smooth rigid analytic variety $X_J$ can be equipped with several natural log structures:
\begin{itemize}
    \item the trivial log structure $\alpha_{X_J}^{\Utext{triv}}: \cM_{X_J}^{\Utext{triv}} = \cO_{X_{J, \et}}^\times \to \cO_{X_{J, \et}}$;

    \item the log structure $\alpha_{X_J}^{\Utext{std}}: \cM_{X_J}^{\Utext{std}} \to \cO_{X_{J, \et}}$ defined by the normal crossings divisor $D_J$ as in \cite[\aEx \logRHexlogadicspncd]{Diao/Lan/Liu/Zhu:lrhrv}; and

    \item the log structure associated with the pre-log structure $\imath_J^{-1}(\cM_X) \to \cO_{X_{J, \et}}$ induced by the composition of $\alpha_X$ and $\imath_J^\#: \cO_{X_\et} \to \imath_{J, *}(\cO_{X_{J, \et}})$, which we shall denote by $\alpha_{X_J}^\partial: \cM_{X_J}^\partial \to \cO_{X_{J, \et}}$.
\end{itemize}
By abuse of notation, we shall denote these log adic spaces by $X_J^\times$, $X_J$, and $X_J^\partial$, respectively.  For the sake of clarity, let us introduce the following:
\begin{defn}\label{def-imm}
    We say that a morphism of log adic spaces is a \emph{closed immersion} \Pth{\resp an open immersion} if it is \emph{strict} as in \cite[\aDef \logadicdeflogstr(\logadicdeflogstrseven)]{Diao/Lan/Liu/Zhu:lasfr}---\ie, if the log structure on the source space is canonically isomorphic to the pullback as above \Pth{\resp the restriction} of the one on the target space---and if the underlying morphism of adic spaces is a closed immersion \Pth{\resp an open immersion}.
\end{defn}

\begin{rk}\label{rem-cl-imm}
    Definition \ref{def-imm} is more restrictive than the one in \cite[\aDef \logadicdefimm]{Diao/Lan/Liu/Zhu:lasfr}, where closed immersions that are not necessarily strict were also introduced.  However, we do not need such a generality in this paper.
\end{rk}

As explained in \cite[\aEx \logadicexlogadicspncdstrictclimm]{Diao/Lan/Liu/Zhu:lasfr}, we have the following commutative diagram of canonical morphisms between log adic spaces
\[
    \xymatrix{ {U_J^\partial} \ar[d]_-{\varepsilon_J^\partial|_{U_J^\partial}} \ar[r]^-{\jmath_J^\partial} & {X_J^\partial} \ar[d]_-{\varepsilon_J^\partial} \ar[r]^-{\imath_J^\partial} & {X} \\
    {U_J} \ar[r]^-{\jmath_J} & {X_J} }
\]
in which $\jmath_J^\partial$ and $\jmath_J$ are open immersions, $\imath_J^\partial$ is a closed immersion, and the underlying morphisms of adic spaces of $\varepsilon_J^\partial|_{U_J^\partial}$ and $\varepsilon_J^\partial$ are isomorphisms.  Moreover, $U_J$ is equipped with the trivial log structure, while $U_J^\partial$ is equipped with the log structure pulled back from $X_J^\partial$ and hence $X$.  Note that there is no natural morphism of log adic spaces from $X_J$ to $X$, and this is the main reason to introduce $X_J^\partial$.

For each $a \geq 0$, we define the log adic space
\[
    X_{(a)}^\partial := \coprod_{J \subset I^\Sc, \, |J| = a} X_J^\partial,
\]
a disjoint union, which admits a canonical finite morphism of log adic spaces
\[
    \imath_{(a)}^\partial: X_{(a)}^\partial \to X.
\]
\Pth{Note that the definition of $X_{(a)}^\partial$ only involves the irreducible components of $D^\Sc$.}

\begin{rk}\label{rem-ket-local-geom-pattern}
    In what follows, we will sometimes use Kummer \'etale localizations $X' \to X$ to reduce the proofs of various statements for torsion local systems to the analogous statements for constant ones, and the assertions to prove will often be equivalent to assertions concerning direct images and direct images with compact support from open complements of closed subspaces of the forms $X_J$, $D_J$, or $D^J$ above.  This is justified because, by \cite[\aProp \logadicpropAbhyankar{} and \aLem \logadiclemAbhyankarbasic]{Diao/Lan/Liu/Zhu:lasfr}, locally over $X'$, the underlying reduced subspaces of the preimages of the irreducible components of $D$ still form normal crossings divisors of the same pattern.
\end{rk}

We will make use of the following notation and conventions in the remainder of this paper.  Let $Y$ be a locally noetherian fs log adic space over $k$.  Let $\imath: Z \to Y$ be a closed immersion of log adic spaces, and $\jmath: W \to Y$ an open immersion of log adic spaces, over $k$.  \Pth{By Definition \ref{def-imm}, this means that the log structures on $Z$ and $W$ are the pullbacks of the one on $Y$.}  For $? = \an$, $\et$, $\ket$, $\proet$, or $\proket$ \Pth{referring to the analytic, \'etale, Kummer \'etale, pro-\'etale, or pro-Kummer \'etale topology, respectively, on these spaces}, let $(\imath_{?, *}, \imath_?^{-1})$ and $(\jmath_{?, *}, \jmath_?^{-1})$ denote the associated morphisms of topoi.  For an abelian sheaf $\cF$ on $Y_?$, we shall sometimes denote $\imath^{-1}(\cF)$ \Pth{\resp $\jmath^{-1}(\cF)$} by $\cF|_Z$ \Pth{\resp $\cF|_W$}.  Note that $\jmath_?^{-1}$ admits a left adjoint, denoted by $\jmath_{?, !}$, which is an exact functor on the category of abelian sheaves.

\begin{lem}\label{lem-exc-ex-seq}
    In the above setting, assume moreover that $W = Y - Z$.  Then, for every abelian sheaf $\cF$ on $Y_?$, where $? = \an$, $\et$, or $\ket$, we have the excision short exact sequence $0 \to \jmath_{?, !} \, \jmath_?^{-1}(\cF) \to \cF \to \imath_{?, *} \, \imath_?^{-1}(\cF) \to 0$.  Moreover, the functor $\imath_{?, *}$ \Pth{\resp $\jmath_{?, !}$} from the category of abelian sheaves on $Z$ \Pth{\resp $W$} to the category of abelian sheaves on $Y$ is exact and fully faithful.
\end{lem}
\begin{proof}
    See \cite[\aLem \logadiclemexc]{Diao/Lan/Liu/Zhu:lasfr}.
\end{proof}

\begin{lem}\label{lem-!-resol}
    Let $\cF$ be an abelian sheaf on $X_?$, where $? = \an$, $\et$, or $\ket$.  For each $a \geq 0$, let us denote $\imath_{(a), ?}^{\partial, -1}(\cF)$ by $\cF_{(a)}$, which is an abelian sheaf on $X_{(a), ?}^\partial$.  Let us choose any total order of the finite set $I^\Sc$, which induces total orders on any subset $J$ of $I^\Sc$.  Then there is an exact complex
    \[
        0 \to \jmath_{?, !}^\Sc(\cF|_{U^\Sc_?}) \to \imath_{(0), ?, *}^\partial(\cF_{(0)}) \to \imath_{(1), ?, *}^\partial(\cF_{(1)}) \to \cdots \to \imath_{(a), ?, *}^\partial(\cF_{(a)}) \to \cdots
    \]
    over $X_?$, where the morphism $\imath_{(a), ?, *}^\partial(\cF_{(a)}) \to \imath_{(a + 1), ?, *}^\partial(\cF_{(a + 1)})$, for each $a \geq 0$, is the direct sum of morphisms $\imath_{J, ?, *}^\partial(\cF|_{X_J^\partial}) \to \imath_{J', ?, *}^\partial(\cF|_{X_{J'}^\partial})$ indexed by pairs $(J, J')$ with $J \subsetneq J' \subset I^\Sc$, $|J| = a$, and $J' = J \cup \{ j_0 \}$ for some $j_0$; each being the canonical one induced by the closed immersion $X_{J'}^\partial \to X_J^\partial$ multiplied by $(-1)^{|\{ j \in J' : j < j_0 \}|}$.  \Pth{This is probably well known, but we included some details to at least set up the convention, because such complexes will appear repeatedly in our arguments.}
\end{lem}
\begin{proof}
    By Lemma \ref{lem-exc-ex-seq}, it suffices to check the exactness of the complex after pulling it back to $U^\Sc \subset X$ and to $U_J^\partial \subset X$, for each $J\subset I^\Sc$.  This pullback can be identified with $0 \to \cF|_{U^\Sc_?} \to \cF|_{U^\Sc_?} \to 0$ in the former case, where the morphism in the middle is the identity morphism; and with a complex
    \[
        0 \to 0 \to (\cF|_{U_{J, ?}^\partial})^{\binom{a}{0}} \to (\cF|_{U_{J, ?}^\partial})^{\binom{a}{1}} \to \cdots \to (\cF|_{U_{J, ?}^\partial})^{\binom{a}{a - 1}} \to (\cF|_{U_{J, ?}^\partial})^{\binom{a}{a}} \to 0 \to \cdots
    \]
    in the latter case, where $a = |J|$ and the exponents are the binomial coefficients.  In both cases, the sequences are exact by construction, as desired.
\end{proof}

\subsection{Kummer \'etale local systems on smooth boundary strata}\label{sec-loc-syst-bd}

Let $\bL$ be a torsion local system on $X_\ket$.  For each $J \subset I$, let
\[
    \bL_J := \imath_{J, \ket}^{\partial, -1}(\bL).
\]
We have the following primitive comparison theorem for $X_{J, \ket}^\partial$ and $\bL_J$:
\begin{thm}\label{thm-prim-comp-bd}
    Suppose that $k$ is algebraically closed of characteristic zero, $X_J$ is proper, and $\bL$ is an $\bF_p$-local system.  Then there is a natural almost isomorphism
    \begin{equation}\label{eq-thm-prim-comp-bd}
        H^i\bigl(X_{J, \ket}^\partial, \bL_J \bigr) \otimes_{\bF_p} (k^+ / p) \Mi H^i\bigl(X_{J, \ket}^\partial, \bL_J  \otimes_{\bF_p} (\cO_{X_{J, \ket}^\partial}^+ / p)\bigr)
    \end{equation}
    of almost finitely generated $k^+$-modules, for each $i \geq 0$.  Both $k^+$-modules are almost zero when $i > 2 \dim(X_J) + |J| = 2 \dim(X) - |J|$.
\end{thm}

The remainder of this subsection will be devoted to the proof of Theorem \ref{thm-prim-comp-bd}.  Along the way, we will establish several other facts that will be needed in the remainder of this paper.  We shall temporarily drop the assumptions that $k$ is algebraically closed, that $X_J$ is proper, and that $\bL$ is $p$-torsion.  The first step is the following proposition:
\begin{prop}\label{prop-L-J-dir-im-et}
    The sheaf $R^i\varepsilon_{J, \ket, *}^\partial( \bL_J )$ is a torsion local system on $X_{J, \ket}$, for each $i \geq 0$, and it vanishes when $i > |J|$.  Moreover, for each $m \geq 1$, the canonical morphism $R^i\varepsilon_{J, \ket, *}^\partial(\bL_J) \to R^i\varepsilon_{J, \ket, *}^\partial(\bL_J / m)$ is surjective.
\end{prop}

We need some preparations before presenting the proof of Proposition \ref{prop-L-J-dir-im-et}.

\begin{lem}\label{lem-L-J-op}
    Let $D^J = \cup_{j \in I - J} \, D_j$ as before, so that $D_J = X_J \cap D^J$ as subsets of $X$.  Let $\widetilde{\jmath}_J: X - D^J \to X$ denote the canonical open immersion of log adic spaces, whose pullback under $\imath_J^\partial: X_J^\partial \to X$ is $\jmath_J^\partial: U_J^\partial \to X_J^\partial$.  Let $\jmath^J: X - X_J^\partial \to X$ denote the complementary open immersion.  Then the adjunction morphism
    \begin{equation}\label{eq-lem-L-J-op}
        \jmath_{\ket, !}^J \, \jmath_\ket^{J, -1} \, R\widetilde{\jmath}_{J, \ket, *} \, \widetilde{\jmath}_{J, \ket}^{\; -1}(\bL) \to R\widetilde{\jmath}_{J, \ket, *} \, \widetilde{\jmath}_{J, \ket}^{\; -1} \, \jmath_{\ket, !}^J \, \jmath_\ket^{J, -1}(\bL),
    \end{equation}
    induced by $\widetilde{\jmath}_{J, \ket}^{\; -1} \, \jmath_{\ket, !}^J \, \jmath_\ket^{J, -1} \, R\widetilde{\jmath}_{J, \ket, *} \, \widetilde{\jmath}_{J, \ket}^{\; -1}(\bL) \cong \widetilde{\jmath}_{J, \ket}^{\; -1} \, \jmath_{\ket, !}^J \, \jmath_\ket^{J, -1}(\bL)$ is an isomorphism.
\end{lem}
\begin{proof}
    As explained in Remark \ref{rem-ket-local-geom-pattern}, we may work locally on $X_\ket$, and assume that $\bL = \bZ / m$ for some integer $m \geq 1$.  By the same argument as in the proof of \cite[\aThm \logadicthmpurity]{Diao/Lan/Liu/Zhu:lasfr}, it suffices to show that the analogue of \Refeq{\ref{eq-lem-L-J-op}} over $X_\et$ \Pth{with subscripts \Qtn{$\ket$} replaced with \Qtn{$\et$}} is an isomorphism.  Since $D$ is a normal crossings divisor \Pth{again, see \cite[\aEx \logRHexlogadicspncd]{Diao/Lan/Liu/Zhu:lrhrv}}, up to \'etale localization, we may reduce \Pth{by \cite[\aProp 2.1.4 and \aThms 3.8.1 and 5.7.2]{Huber:1996-ERA}} to the case of schemes, and assume that $X$ is a fiber product of two varieties $X_1$ and $X_2$ over $k$, with the morphisms $\jmath^J$ and $\widetilde{\jmath}_J$ being pullbacks of some open immersions to $X_1$ and $X_2$, respectively.  Then the desired assertion follows from the K\"unneth isomorphisms as in \cite[\aSec 4.2.7]{Beilinson/Bernstein/Deligne/Gabber:2018-FP(2)} \Pth{\Refcf{} \cite[\aLem 4.3.23 and its proof]{Lan/Stroh:2018-csisv}}.
\end{proof}

\begin{rk}\label{rem-!-*-switch}
    A similar argument shows that there is a canonical isomorphism
    \[
        \jmath_{\et, !}^\Sc \, R\jmath_{\Sc, \et, *}(\bL|_U) \Mi R\jmath_{\et, *}^\Snc \, \jmath_{\Snc, \et, !}(\bL|_U).
    \]
\end{rk}

\begin{lem}\label{lem-L-J-pure}
    The adjunction morphism
    \begin{equation}\label{eq-lem-L-J-pure}
        \bL_J \to R\jmath_{J, \ket, *}^\partial \, \jmath_{J, \ket}^{\partial, -1}(\bL_J)
    \end{equation}
    is an isomorphism.
\end{lem}
\begin{proof}
    Let us retain the setting of Lemma \ref{lem-L-J-op}.  By Lemma \ref{lem-exc-ex-seq}, it suffices to apply $\imath_{J, \ket, *}^\partial$ to the morphism \Refeq{\ref{eq-lem-L-J-pure}}, and show that the morphism
    \[
    \begin{split}
        & \imath_{J, \ket, *}^\partial \, \imath_{J, \ket}^{\partial, -1}(\bL) \to \imath_{J, \ket, *}^\partial \, R\jmath_{J, \ket, *}^\partial \, \jmath_{J, \ket}^{\partial, -1} \, \imath_{J, \ket}^{\partial, -1}(\bL) \\
        & \cong R\widetilde{\jmath}_{J, \ket, *} \, (\imath_J^\partial|_{U_J})_{\ket, *} \, \jmath_{J, \ket}^{\partial, -1} \, \imath_{J, \ket}^{\partial, -1}(\bL) \cong R\widetilde{\jmath}_{J, \ket, *} \, \widetilde{\jmath}_{J, \ket}^{\; -1} \, \imath_{J, \ket, *}^\partial \, \imath_{J, \ket}^{\partial, -1}(\bL),
    \end{split}
    \]
    which can be identified with the adjunction morphism for the sheaf $\imath_{J, \ket, *}^\partial \, \imath_{J, \ket}^{\partial, -1}(\bL)$ and the morphism $\widetilde{\jmath}_J$, is an isomorphism.  By \cite[\aThm \logadicthmpurity]{Diao/Lan/Liu/Zhu:lasfr}, the adjunction morphism $\bL \to R\widetilde{\jmath}_{J, \ket, *} \, \widetilde{\jmath}_{J, \ket}^{\; -1}(\bL)$ is an isomorphism over $X_\ket$.  Hence, we have a canonical isomorphism $\jmath_{\ket, !}^J \, \jmath_\ket^{J, -1}(\bL) \Mi \jmath_{\ket, !}^J \, \jmath_\ket^{J, -1} \, R\widetilde{\jmath}_{J, \ket, *} \, \widetilde{\jmath}_{J, \ket}^{\; -1}(\bL)$, whose composition with \Refeq{\ref{eq-lem-L-J-op}} is the adjunction morphism for the sheaf $\jmath_{\ket, !}^J \, \jmath_\ket^{J, -1}(\bL)$ and the morphism $\widetilde{\jmath}_J$.  Thus, the desired assertion follows from Lemmas \ref{lem-exc-ex-seq} and \ref{lem-L-J-op}.
\end{proof}

Let $\cM_X$ be as in Section \ref{sec-log-str-bd}.  By \cite[\aLems \logadiclemclimmketmor{} and \logadiclemkettoetconst]{Diao/Lan/Liu/Zhu:lasfr}, we have
\begin{equation}\label{eq-L-J-dir-im-et-J-log-str}
    R^i(\varepsilon_J^\partial|_{U_J^\partial})_{\et, *}(\bZ / n) \cong \bigl(\Ex^i (\overline{\cM}_X^\gp / n)\bigr)(-i)|_{U_J}
\end{equation}
over $U_{J, \et}$, for each $n \in \bZ_{\geq 1}$.  Now we are ready for the following:
\begin{proof}[Proof of Proposition \ref{prop-L-J-dir-im-et}]
    By Lemma \ref{lem-L-J-pure}, and by applying \cite[\aThm \logadicthmpurity]{Diao/Lan/Liu/Zhu:lasfr} to torsion local systems on $X_J$, we may replace $X_J^\partial$ \Pth{\resp $X$} with $U_J^\partial$ \Pth{\resp $X - D^J$}.  By \cite[\aLem \logadiclemclimmketmor]{Diao/Lan/Liu/Zhu:lrhrv} and Remark \ref{rem-ket-local-geom-pattern}, and by the same argument as in the proof of \cite[\aThm \logadicthmpurity]{Diao/Lan/Liu/Zhu:lasfr}, we may work locally on $X_\ket$, and assume that $\bL = \bZ / n$ for some $n \in \bZ_{\geq 1}$.  Then Proposition \ref{prop-L-J-dir-im-et} reduces to the isomorphism \Refeq{\ref{eq-L-J-dir-im-et-J-log-str}}, which is clearly compatible with reduction mod $m$ on both sides.
\end{proof}

\begin{cor}\label{cor-L-J-coh-fin}
    Let $\bL$ be a $\bZ / p^m$-local system on $X_\ket$.  Then we have the Leray spectral sequence
    \begin{equation}\label{eq-cor-L-J-coh-fin}
        E_2^{a, b} := H^a\bigl(X_{J, \ket}, R^b\varepsilon_{J, \ket, *}^\partial(\bL_J)\bigr) \Rightarrow H^{a + b}(X_{J, \ket}^\partial, \bL_J).
    \end{equation}
    In particular, the $\bZ / p^m$-module $H^i(X_{J, \ket}^\partial, \bL_J)$ is finitely generated, for any $i \geq 0$ and $m \geq 0$, and vanishes when $i > 2 \dim(X_J) + |J| = 2 \dim(X) - |J|$.
\end{cor}
\begin{proof}
    This follows from Proposition \ref{prop-L-J-dir-im-et} and \cite[\aThm \logadicthmprimcomp]{Diao/Lan/Liu/Zhu:lasfr}.
\end{proof}

\begin{proof}[Proof of Theorem \ref{thm-prim-comp-bd}]
    Consider the Leray spectral sequence
    \[
    \begin{split}
        E_2^{a, b} & := H^a\bigl(X_{J, \ket}, \bigl(R^b\varepsilon_{J, \ket, *}^\partial(\bL_J )\bigr) \otimes_{\bF_p} (\cO_{X_{J, \ket}}^+ / p)\bigr) \\
        & \cong H^a\bigl(X_{J, \ket}, R^b\varepsilon_{J, \ket, *}^\partial\bigl(\bL_J \otimes_{\bF_p} (\cO_{X_{J, \ket}^\partial}^+ / p)\bigr)\bigr) \\
        & \Rightarrow H^{a + b}\bigl(X_{J, \ket}^\partial, \bL_J  \otimes_{\bF_p} (\cO_{X_{J, \ket}^\partial}^+ / p)\bigr),
    \end{split}
    \]
    where the first isomorphism is based on \cite[\aLem \logadiclemketmorOplusp]{Diao/Lan/Liu/Zhu:lasfr}, which admits a morphism from the following spectral sequence, given by the base change of \Refeq{\ref{eq-cor-L-J-coh-fin}}:
    \[
        E_2^{a, b} := H^a\bigl(X_{J, \ket}, R^b\varepsilon_{J, \ket, *}^\partial(\bL_J)\bigr) \otimes_{\bF_p} (k^+ / p) \Rightarrow H^{a + b}(X_{J, \ket}^\partial, {\bL}_J) \otimes_{\bF_p} (k^+ / p).
    \]
    By Proposition \ref{prop-L-J-dir-im-et} and \cite[\aThm \logadicthmprimcomp]{Diao/Lan/Liu/Zhu:lasfr}, this morphism is given by almost isomorphisms of $k^+$-modules between the $E_2$ terms, which are almost finitely generated $k^+$-modules that are almost zero except when $a, b \geq 0$ and $a + b \leq 2 \dim(X_J) + |J| = 2 \dim(X) - |J|$ \Pth{as in Corollary \ref{cor-L-J-coh-fin}}.  Thus, the theorem follows.
\end{proof}

\subsection{Kummer \'etale cohomology with partial compact support}\label{sec-ket-coh-cpt}

In this subsection, let us fix $I^\Sc \subset I$ and define $U^\Sc$ etc as in Section \ref{sec-log-str-bd}.  Let $\bL$ be a torsion local system on $X_\ket$ as before.  As in Lemma \ref{lem-!-resol}, for each $a \geq 0$, let
\begin{equation}\label{eq-def-L-a}
   \bL_{(a)} := \imath_{(a), \ket}^{\partial, -1}(\bL).
\end{equation}

We shall deduce from Theorem \ref{thm-prim-comp-bd} its analogue for the cohomology with partial compact support.  Let us first introduce the relevant cohomology groups.
\begin{defn}\label{def-H-c-torsion}
    Assume that $k$ is algebraically closed and of characteristic zero, and that $X$ is proper over $k$.  For any torsion local system $\bL$ on $X_\ket$, we abusively define
     \[
          H_{\et, \Sc}^i(U, \bL) := H_\Sc^i(U_\et, \bL) := H^i\bigl(X_\ket, \jmath_{\ket, !}^\Sc(\bL|_{U^\Sc_\ket})\bigr).
     \]
     \Pth{We introduce both $H_{\et, \Sc}^i(U, \bL)$ and $H_\Sc^i(U_\et, \bL)$ for the sake of flexibility.}
\end{defn}

The following lemma shows that $H_{\et, \Sc}^i(U, \bL)$ can be interpreted as the cohomology of $\bL|_U$ with a partial compact support condition along $D^\Sc \subset X$, which justifies our choice of notation.
\begin{lem}\label{lem-L-!}
    We have canonical isomorphisms
    \begin{equation}\label{eq-L-!}
        \jmath_{\et, !}^\Sc \, R\jmath_{\Sc, \et, *}(\bL|_U) \Mi \jmath_{\et, !}^\Sc \, R\varepsilon_{\et, *}(\bL|_{U^\Sc_\ket}) \Mi R\varepsilon_{\et, *} \, \jmath_{\ket, !}^\Sc(\bL|_{U^\Sc_\ket})
    \end{equation}
    \Pth{\Refcf{} Remark \ref{rem-!-*-switch}}.  Therefore, if $k$ is algebraically closed and of characteristic zero and $X$ is proper, then we have
    \[
        H_{\et,\Sc}^i(U, \bL) \cong H^i\bigl(X_\et, \jmath_{\et, !}^\Sc \, R\jmath_{\Sc, \et, *}(\bL|_U)\bigr) \cong H_\cpt^i\bigl(U^\Sc_\et, R\jmath_{\Sc, \et, *}(\bL|_U)\bigr).
    \]
    In particular,
    \begin{itemize}
         \item if $I^\Sc = \emptyset$, then $H_{\et, \Sc}^i(U, \bL) \cong H^i(U_\et, \bL|_U)$;
         \item if $I^\Sc = I$, then $H_{\et, \Sc}^i(U, \bL) \cong H_\cpt^i(U_\et, \bL|_U)$,
    \end{itemize}
    where $H_\cpt^i(U_\et, \bL|_U)$ is the \'etale cohomology with compact support of the \'etale local system $\bL|_U$ on $U_\et$, as defined in \cite[\aSec 5]{Huber:1996-ERA}.
\end{lem}
\begin{proof}
    The first isomorphism in \Refeq{\ref{eq-L-!}} follows from \cite[\aThm \logadicthmpurity]{Diao/Lan/Liu/Zhu:lasfr} and its proof.  The second isomorphism, as in the proof of \cite[\aLem \logadiclemclimmketmor]{Diao/Lan/Liu/Zhu:lasfr}, follows from the definitions of the sheaves by comparing stalks at log geometric points using \cite[\aLem \logadiclemloggeompt]{Diao/Lan/Liu/Zhu:lasfr}.  The rest of the lemma follows immediately.
\end{proof}

Now we are ready to state the following primitive comparison theorem for the cohomology with partial compact support \Pth{\Refcf{} the analogous results \cite[\aThm 5.1]{Scholze:2013-phtra} and \cite[\aThm \logadicthmprimcomp]{Diao/Lan/Liu/Zhu:lasfr} for the usual cohomology}:
\begin{thm}\label{thm-L-!-prim-comp}
    Assume that $k$ is algebraically closed and of characteristic zero, that $X$ is proper over $k$, and that $\bL$ is an $\bF_p$-local system on $X_\ket$.  Then:
    \begin{enumerate}
        \item\label{thm-L-!-prim-comp-fg}  $H^i\bigl(X_\ket, \bigl(\jmath_{\ket, !}^\Sc(\bL|_{U^\Sc_\ket})\bigr) \otimes_{\bF_p} (\cO_X^+ / p)\bigr)$ is an almost finitely generated $k^+$-module for each $i \geq 0$, and is almost zero if $i > 2 \dim(X)$.

        \item\label{thm-L-!-prim-comp-isom}  There is a canonical almost isomorphism
            \begin{equation}\label{eq-thm-L-!-prim-comp}
                H_{\et, \Sc}^i(U, \bL) \otimes_{\bF_p} (k^+ / p) \Mi H^i\bigl(X_\ket, \bigl(\jmath_{\ket, !}^\Sc(\bL)\bigr) \otimes_{\bF_p} (\cO_X^+ / p)\bigr)
            \end{equation}
            of $k^+$-modules, for each $i \geq 0$.  In particular, $H_{\et, \Sc}^i(U, \bL)$ is a finite-dimensional $\bF_p$-vector space for each $i \geq 0$, and vanishes for $i > 2 \dim(X)$.
    \end{enumerate}
\end{thm}
\begin{proof}
    By Lemma \ref{lem-!-resol} and \cite[\aLem \logadiclemclimmOplusp]{Diao/Lan/Liu/Zhu:lasfr}, we have an exact complex
    \begin{equation}\label{eq-thm-L-!-prim-comp-resol}
    \begin{split}
        0 & \to \bigl(\jmath_{\ket, !}^\Sc(\bL|_{U^\Sc_\ket})\bigr) \otimes_{\bF_p} (\cO_X^+ / p) \\
        & \to \imath_{(0), \ket, *}^\partial\bigl(\bL_{(0)} \otimes_{\bF_p} (\cO_{X^\partial_{(0)}}^+ / p)\bigr) \to \imath_{(1), \ket, *}^\partial\bigl(\bL_{(1)} \otimes_{\bF_p} (\cO_{X^\partial_{(1)}}^+ / p)\bigr) \\
        & \to \cdots \to \imath_{(a), \ket, *}^\partial\bigl(\bL_{(a)} \otimes_{\bF_p} (\cO_{X^\partial_{(a)}}^+ / p)\bigr) \to \cdots
    \end{split}
    \end{equation}
    over $X_\ket$, which admits a canonical morphism from the complex
    \[
        0 \to \jmath_{\ket, !}^\Sc(\bL|_{U^\Sc_\ket}) \to \imath_{(0), \ket, *}^\partial(\bL_{(0)}) \to \imath_{(1), \ket, *}^\partial(\bL_{(1)}) \to \cdots \to \imath_{(a), \ket, *}^\partial(\bL_{(a)}) \to \cdots
    \]
    \Pth{as in Lemma \ref{lem-!-resol}}.  Therefore, we obtain a \Pth{filtration} spectral sequence
    \[
    \begin{split}
        E_1^{a, b} & := H^{a + b}\bigl(X_{(a), \ket}^\partial, \bL_{(a)} \otimes_{\bF_p} (\cO_{X^\partial_{(a)}}^+ / p)\bigr) \\
        & \cong \oplus_{J \subset I^\Sc, \, |J| = a} \; H^{a + b}\bigl(X_{J, \ket}^\partial, \bL_J \otimes_{\bF_p} (\cO_{X^\partial_J}^+ / p)\bigr) \\
        & \Rightarrow H^{a + b}\bigl(X_\ket, \bigl(\jmath_{\ket, !}^\Sc(\bL|_{U^\Sc_\ket})\bigr) \otimes_{\bF_p} (\cO_X^+ / p)\bigr),
    \end{split}
    \]
    which admits a canonical morphism from the spectral sequence
    \[
    \begin{split}
        E_1^{a, b} & := H^{a + b}(X_{(a), \ket}^\partial, \bL_{(a)}) \otimes_{\bF_p} (k^+ / p) \\
        & \cong \oplus_{J \subset I^\Sc, \, |J| = a} \; \bigl(H^{a + b}(X_{J, \ket}^\partial, \bL_J) \otimes_{\bF_p} (k^+ / p)\bigr) \\
        & \Rightarrow H^{a + b}\bigl(X_\ket, \jmath_{\ket, !}^\Sc(\bL|_{U^\Sc_\ket})\bigr) \otimes_{\bF_p} (k^+ / p).
    \end{split}
    \]
    By Theorem \ref{thm-prim-comp-bd}, this morphism is given by almost isomorphisms of $k^+$-modules between the $E_1$ terms, which are almost finitely generated $k^+$-modules that are almost zero except when $a, b \geq 0$ and $a + b \leq 2 \dim(X)$.  Thus, the morphism induces the canonical almost isomorphism \Refeq{\ref{eq-thm-L-!-prim-comp}} in \Refeq{\ref{thm-L-!-prim-comp-isom}} and justifies \Refeq{\ref{thm-L-!-prim-comp-fg}}, as desired.
\end{proof}

\subsection{Pro-Kummer \'etale cohomology with partial compact support}\label{sec-proket-coh-cpt}

Recall that a $\bZ_p$-local system $\bL$ is an inverse system $\{ \bL_n \}_{n \geq 1}$, where each $\bL_n$ is a $\bZ / p^n$-local system, satisfying $\bL_m / p^n \cong \bL_n$ for all $m \geq n \geq 1$.  \Pth{See \cite[\aDef \logadicdefketlisse]{Diao/Lan/Liu/Zhu:lasfr}.}  Since we will need to deal with inverse systems of sheaves on $X_\ket$ such as $\{ \jmath_!(\bL_n) \}_{n \geq 1}$, it is convenient to introduce the following definitions:
\begin{defn}\label{def-adic-formalism}
    A Kummer \'etale $\bZ_p$-sheaf $\cF$ on a locally noetherian fs log adic space $Y$ \Pth{over $k$} is an inverse system $\{ \cF_n \}_{n \geq 1}$ of sheaves on $Y_\ket$, where $\cF_n$ is a $\bZ / p^n$-module, for each $n \geq 1$.  Let $\Shv_{\bZ_p}(Y_\ket)$ denote the abelian category of $\bZ_p$-sheaves on $Y_\ket$, which has enough injectives by \cite[Prop. 1.1]{Jannsen:1988-cec}. If $f: Y' \to Y$ is a morphism between such log adic spaces, let
    \[
        f_\ket^{-1}: \Shv_{\bZ_p}(Y_\ket) \rightleftarrows \Shv_{\bZ_p}(Y'_\ket): f_{\ket, *}
    \]
    be the pair of adjoint functors, namely, the inverse and direct image functors of $\bZ_p$-sheaves, given by applying the usual $f_\ket^{-1}$ and $f_{\ket, *}$ \Pth{for torsion sheaves} to each component of the inverse system.  If $f = \jmath: W \to Y$ is an open immersion, let
    \[
        \jmath_{\ket, !}: \Shv_{\bZ_p}(W_\ket) \to \Shv_{\bZ_p}(Y_\ket)
    \]
    be the left adjoint of $\jmath_\ket^{-1}$, again given by applying the usual $\jmath_{\ket, !}$ \Pth{for torsion sheaves} to each component of the inverse system.

    When $k$ is algebraically closed of characteristic zero, we define the $i$-th cohomology $H^i(Y_\ket, \cdot)$ as the $i$-th right derived functor of the functor
    \[
        \Shv_{\bZ_p}(Y_\ket) \to \Mod_{\bZ_p}: \; \{ \cF_n \}_{n \geq 1} \mapsto \Gamma(Y_\ket, \varprojlim_n \cF_n) \cong \varprojlim_n \Gamma(Y_\ket, \cF_n).
    \]
\end{defn}

\begin{defn}\label{def-H-c}
    Let $X$ be as before.  Assume that $k$ is algebraically closed of characteristic zero and that $X$ is proper over $k$.  For each $\bZ_p$-local system $\bL$ on $X_\ket$, we define
    \[
        H_{\et, \Sc}^i(U, \bL) := H_\Sc^i(U_\et, \bL) := H^i\bigl(X_\ket, \jmath_{\ket, !}(\bL|_{U^\Sc_\ket})\bigr).
    \]
    \Pth{Again, we introduce both $H_{\et, \Sc}^i(U, \bL)$ and $H_\Sc^i(U_\et, \bL)$ for the sake of flexibility.}
\end{defn}

\begin{lem}\label{lem-def-H-c-fin-Z-p}
    In the setting of Definition \ref{def-H-c}, there is a canonical isomorphism $H_{\et, \Sc}^i(U, \bL) \cong \varprojlim_n H^i_{\et, \Sc}(U, \bL_n)$ as finite $\bZ_p$-modules.
\end{lem}
\begin{proof}
    By adapting the arguments in the proof of \cite[\aLem 2.3(i)]{Huber:1998-ctlac} to the setting here for the cohomology of the proper $X$ over $k$, we can write $H_{\et, \Sc}^i(U, \bL)$ as
    \[
        H^i\bigl(R\Gamma\bigl(X_\ket, R\varprojlim_n \, \jmath_{\ket, !}(\bL_n|_{U^\Sc_\ket})\bigr)\bigr) \cong H^i\bigl(R\varprojlim_n R\Gamma\bigl(X_\ket, \jmath_{\ket, !}(\bL_n|_{U^\Sc_\ket})\bigr)\bigr).
    \]
    Under our assumptions that $k$ is algebraically closed of characteristic zero and that $X$ is proper over $k$, since each $H^i_{\et, \Sc}(U, \bL_n)$ is finite by Theorem \ref{thm-L-!-prim-comp}, the right-hand side is equal to $\varprojlim_n H^i\bigl(R\Gamma\bigl(X_\ket, \jmath_{\ket, !}(\bL_n|_{U^\Sc_\ket})\bigr)\bigr) \cong \varprojlim_n H^i_{\et, \Sc}(U, \bL_n)$, which is a finite $\bZ_p$-module by standard arguments.
\end{proof}

\begin{rk}\label{rem-def-H-c-conv}
    When $I^\Sc = I$, we have $H_{\et, \Sc}^i(U, \bL) \cong \varprojlim_n H_\cpt^i(U_\et, \bL_n|_U)$, by Lemma \ref{lem-L-!}.  Note that this is different from the $H_\cpt^i(U_\et, \bL|_U)$ as defined in \cite{Huber:1998-ctlac}.  \Pth{Recall that, for a $\bZ_p$-sheaf $\cF = \{ \cF_n \}_{n \geq 1}$ on $U_\et$, which is partially proper over $k$, the cohomology with compact support $H_\cpt^i(U_\et, \cF)$ is defined in \cite{Huber:1998-ctlac} as the $i$-th derived functor of the functor $\cF = \{ \cF_n \}_{n \geq 1} \mapsto \Gamma_\cpt(U_\et, \varprojlim_n \cF_n)$, where $\Gamma_c$ is the functor of sections with proper support, as in \cite[\aDef 5.2.1]{Huber:1996-ERA}.}  In particular, as explained in \cite[\aEx A.1]{Colmez/Dospinescu/Hauseux/Niziol:2021-pecpd}, even when $U$ is the affine line over the $p$-adic complex number, Huber's $H_\cpt^2\bigl(U_\et, \bZ_p(1)\bigr)$ can fail to be a finite $\bZ_p$-module, and hence is not suitable for our study of de Rham comparison and Poincar\'e duality.  Nevertheless, despite this discrepancy, we shall abusively \emph{define} \Pth{or rather \emph{denote}}
    \[
        H_\cpt^i(U_\et, \bL|_U) := \varprojlim_n H_\cpt^i(U_\et, \bL_n|_U).
    \]
    Again by Lemma \ref{lem-L-!}, when $I^\Sc = \emptyset$, which is another extremal case, we have
    \[
        H_{\et, \Sc}^i(U, \bL) \cong \varprojlim_n H^i(U_\et, \bL_n|_U) \cong H^i(U_\et, \bL|_U).
    \]
\end{rk}

\begin{rk}\label{rem-def-H-c-alt}
    We shall denote the objects defined by any subset $I^\Scalt \subset I^\Sc \subset I$ with subscripts \Qtn{$\Scalt$}.  Then the objects with subscripts \Qtn{$\Sc$} admit compatible canonical morphisms to those with subscripts \Qtn{$\Scalt$}.  We shall also denote with subscripts \Qtn{$\Snc$} the objects defined with the complementary subset $I^\Snc \subset I$ replacing $I^\Sc \subset I$.  \Pth{This is consistent of the previous definitions of $\jmath_\Snc$ and $\jmath^\Snc$, although we will not explicitly use them.}
\end{rk}

For each locally noetherian fs log adic space $Y$, the pro-Kummer \'etale site $Y_\proket$ was introduced in \cite[\aSec \logadicsecproket]{Diao/Lan/Liu/Zhu:lasfr}.  Let $\upsilon_Y: Y_\proket \to Y_\ket$ denote the natural projection of sites.  \Pth{We shall omit the subscript \Qtn{$Y$} when the context is clear.}

\begin{lem}\label{lem-from-ket-to-proket}
    For each morphism $f: Z \to Y$ of locally noetherian fs log adic spaces, $\upsilon_Z^{-1} \, f^{-1}_\ket \cong f^{-1}_\proket \, \upsilon_Y^{-1}$.  If $f$ is quasi-compact, then $\upsilon_Y^{-1} \, f_{\ket, *} \cong f_{\proket, *} \, \upsilon_Z^{-1}$.
\end{lem}
\begin{proof}
    The first statement is clear.  As for the second, we may assume that $Y$ is affinoid.  Let $U = \varprojlim_i U_i$ be any qcqs object of $Y_\proket$.  Then $f^{-1}(U) = \varprojlim_i f^{-1}(U_i)$ is a qcqs object in $Z_\proket$.  By \cite[\aProp \logadicpropproketvsket]{Diao/Lan/Liu/Zhu:lasfr}, we have $\bigl(\upsilon_Y^{-1} \, f_{\ket, *}(\cF)\bigr)(U) \cong \varinjlim_i \cF\bigl(f^{-1}(U_i)\bigr) \cong \bigl(\upsilon_Z^{-1}(\cF)\bigr)\bigl(f^{-1}(U)\bigr) \cong \bigl(f_{\proket, *} \, \upsilon_Z^{-1}(\cF)\bigr)(U)$, for each abelian sheaf $\cF$ on $Z_\ket$, as desired.
\end{proof}

\begin{rk}\label{rem-compat-!-bc}
    The above basic results in this subsection are compatible with base changes from $k$ to other nonarchimedean local fields.
\end{rk}

Finally, let $\widehat{\bZ}_p := \varprojlim_n (\bZ / p^n)$, and let $\Shv_{\widehat{\bZ}_p}(Y_\proket)$ denote the category of $\widehat{\bZ}_p$-sheaves on $Y_\proket$, understood in the naive sense.  Then there is a natural functor
\begin{equation}\label{eq-ket-to-proket}
    \upsilon_Y^{-1}: \Shv_{\bZ_p}(Y_\ket) \to \Shv_{\widehat{\bZ}_p}(Y_\proket): \cF = \{ \cF_n \}_{n \geq 1} \mapsto \widehat{\cF} := \varprojlim_n \bigl(\upsilon_Y^{-1}(\cF_n)\bigr).
\end{equation}

\subsection{Period sheaves on the boundary strata}\label{sec-period-bd}

Let us begin with some notational preparation, which will also be used in some later subsections.  Consider the perfectoid field $K := \widehat{\AC{k}}$, the $p$-adic completion of some fixed algebraic closure $\AC{k}$ of $k$, with $K^+ = \cO_K$.  Let $(K^\flat, K^{\flat+})$ denote the tilt of $(K, K^+)$, as usual.  As in \cite[\aSec \logRHsecOBdlexplicit]{Diao/Lan/Liu/Zhu:lrhrv}, let $\xi \in \Ainf = W(K^{\flat+})$ be given by \cite[\aLem 6.3]{Scholze:2013-phtra}, which generates the kernel of the surjective canonical homomorphism $\theta: \Ainf \to K^+$.  Let $\varpi \in K^{\flat+}$ be such that $\varpi^\sharp = p$.  Then $p^m \Ainf / p^{m + 1} \Ainf \cong \Ainf / p \cong K^{\flat+}$ and $\varpi^n K^{\flat+} / \varpi^{n + 1} K^{\flat+} \cong K^{\flat+} / \varpi \cong K^+ / p$, for all $m , n \geq 0$.

\begin{rk}\label{rem-almost-Ainf}
    We shall consider the almost setting over $\Ainf$ with respect to the ideal generated by $\{ [\varpi^{1 / p^N}] \}_{N \geq 1}$, as in \cite[paragraph preceding \aThm 6.5]{Scholze:2013-phtra}.  As explained in the proof of \cite[\aThm 6.5]{Scholze:2013-phtra}, multiplication by $[\varpi] \in \Binf = W(K^{\flat+})[\frac{1}{p}]$ is invertible \Pth{and so almost isomorphisms becomes isomorphisms} after reduction modulo powers of $\xi$, because $[\varpi]$ is mapped to $\varpi^\sharp = p$ in $K$.
\end{rk}

\begin{defn}\label{def-AAinf-etc-bd}
    For each $J \subset I$, by applying $\imath_{J, \proket, *}^\partial$ to the sheaves on $X_{J, \proket}^\partial$ defined in \cite[\aDef \logadicdefproketsheaves]{Diao/Lan/Liu/Zhu:lasfr} and \cite[\aSec \logRHsecOBdl]{Diao/Lan/Liu/Zhu:lasfr}, we obtain the sheaves $\widehat{\cO}_{X_J^\partial}^\partial$, $\widehat{\cO}_{X_J^\partial}^{+, \partial}$, $\widehat{\cO}_{X_J^\partial}^{\flat, \partial}$, $\widehat{\cO}_{X_J^\partial}^{\flat+, \partial}$, $\AAinfX{X_J^\partial}^\partial$, $\BBinfX{X_J^\partial}^\partial$, $\BBdRX{X_J^\partial}^{+, \partial}$, $\BBdRX{X_J^\partial}^\partial$, $\OBdlX{X_J^\partial}^{+, \partial}$, $\OBdlX{X_J^\partial}^\partial$, their filtered pieces, and $\OClX{X_J^\partial}^\partial := \gr^0 \OBdlX{X_J^\partial}^\partial$, together with the homomorphisms $\theta^\partial: \AAinfX{X_J^\partial}^\partial \to \widehat{\cO}_{X_J^\partial}^{+, \partial}$ and $\theta^\partial: \BBinfX{X_J^\partial}^\partial \to \widehat{\cO}_{X_J^\partial}^\partial$, on $X_\proket$, denoted with additional superscripts \Qtn{$\partial$}.  For each $a \geq 0$, we define similar sheaves $\widehat{\cO}_{X_{(a)}^\partial}^\partial$, $\widehat{\cO}_{X_{(a)}^\partial}^{+, \partial}$, $\widehat{\cO}_{X_{(a)}^\partial}^{\flat, \partial}$, $\widehat{\cO}_{X_{(a)}^\partial}^{\flat+, \partial}$, $\AAinfX{X_{(a)}^\partial}^\partial$, $\BBinfX{X_{(a)}^\partial}^\partial$, $\BBdRX{X_{(a)}^\partial}^{+, \partial}$, $\BBdRX{X_{(a)}^\partial}^\partial$, $\OBdlX{X_{(a)}^\partial}^{+, \partial}$, $\OBdlX{X_{(a)}^\partial}^\partial$, their filtered pieces, and $\OClX{X_{(a)}^\partial}^\partial$ on $X_\proket$ by direct sums.
\end{defn}

\begin{lem}\label{lem-O-flat+-cl}
    For each $J \subset I$, and for each log affinoid perfectoid object $U = \varprojlim_{i \in I} U_i$ in $X_\proket$ with associated perfectoid space $\widehat{U}$, the pullback of $U$ to $X_{J, \proket}^\partial$ defined by $V = \varprojlim_{i \in I} (U_i \times_X X_J^\partial)$ is a log perfectoid affinoid object in $X_{J, \proket}^\partial$, with an associated perfectoid space $\widehat{V}$ and a closed immersion $\widehat{V} \to \widehat{U}$ of adic spaces compatible with $\imath_J^\partial: X_J^\partial \to X$.  \Pth{However, the closed immersion $\widehat{V} \to \widehat{U}$ is generally \emph{not} the pullback of $\imath_J^\partial$ under $\widehat{U} \to U$.}  Suppose that $\widehat{V} = \Spa(\overline{R}, \overline{R}^+)$ for some perfectoid $(\overline{R}, \overline{R}^+)$ with tilt $(\overline{R}^\flat, \overline{R}^{\flat+})$.  Then we have the following:
    \begin{enumerate}
        \item $\bigl(\widehat{\cO}_{X_{J, \proket}^\partial}^\partial(U), \widehat{\cO}_{X_{J, \proket}^\partial}^{+, \partial}(U)\bigr) \cong \bigl(\widehat{\cO}_{X_{J, \proket}^\partial}(V), \widehat{\cO}_{X_{J, \proket}^\partial}^+(V)\bigr) \cong \bigl(\overline{R}, \overline{R}^+\bigr)$.

        \item $\bigl(\widehat{\cO}_{X_{J, \proket}^\partial}^{\flat, \partial}(U), \widehat{\cO}_{X_{J, \proket}^\partial}^{\flat+, \partial}(U)\bigr) \cong \bigl(\widehat{\cO}_{X_{J, \proket}^\partial}^\flat(V), \widehat{\cO}_{X_{J, \proket}^\partial}^{\flat+}(V)\bigr) \cong \bigl(\overline{R}^\flat, \overline{R}^{\flat+}\bigr)$.
    \end{enumerate}
\end{lem}
\begin{proof}
    These follow from \cite[\aLem \logadiclemlogaffperfclimm{} and \aThm \logadicthmalmostvanhat]{Diao/Lan/Liu/Zhu:lasfr}.
\end{proof}

\begin{cor}\label{cor-O-flat+-cl}
    For each $J \subset I$, let $\cF$ be one of the following sheaves on $X_{J, \proket}^\partial$: $\widehat{\cO}_{X_J^\partial}$, $\widehat{\cO}_{X_J^\partial}^+$, $\widehat{\cO}_{X_J^\partial}^\flat$, $\widehat{\cO}_{X_J^\partial}^{\flat+}$, $\AAinfX{X_J^\partial}$, $\BBinfX{X_J^\partial}$, $\BBdRpX{X_J^\partial}$, and $\BBdRX{X_J^\partial}$.  Then the canonical morphisms $\imath_{J, \proket}^{\partial, -1}(\cF^\partial) = \imath_{J, \proket}^{\partial, -1} \, \imath_{J, \proket, *}^\partial(\cF) \to \cF$ and $\cF^\partial \to \imath_{J, \proket, *}^\partial \, \imath_{J, \proket}^{\partial, -1}(\cF^\partial)$ defined by adjunction are isomorphisms.  If $U$ and $V$ are in Lemma \ref{lem-O-flat+-cl}, then $\cF^\partial(U) \cong \cF(V)$.  Moreover, we have the following:
    \begin{enumerate}
        \item $\AAinfX{X_J^\partial}^\partial \cong W(\widehat{\cO}_{X_{J, \proket}^\partial}^{\flat+, \partial})$ and $\BBinfX{Z}^\partial \cong \AAinfX{Z}^\partial[\frac{1}{p}]$.

        \item The kernels of $\theta^\partial: \AAinfX{X_J^\partial}^\partial \to \widehat{\cO}_{X_J^\partial}^{+, \partial}$ and $\theta^\partial: \BBinfX{X_J^\partial}^\partial \to \widehat{\cO}_{X_J^\partial}^\partial$ are locally principal over $X_\proket$, and are generated by the above $\xi$ over $X_{K, \proket}$.

        \item $\BBdRX{X_J^\partial}^{+, \partial} \cong \varprojlim (\BBinfX{X_J^\partial}^\partial / \xi^r)$ and $\BBdRX{X_J^\partial}^\partial \cong \BBdRX{X_J^\partial}^{+, \partial}[\frac{1}{\xi}]$, where $\xi$ is any local generator of $\ker \theta^\partial$, which can be the above $\xi$ over $X_{K, \proket}$.
    \end{enumerate}
\end{cor}
\begin{proof}
    The assertions for $\widehat{\cO}_{X_J^\partial}$, $\widehat{\cO}_{X_J^\partial}^+$, $\widehat{\cO}_{X_J^\partial}^\flat$, and $\widehat{\cO}_{X_J^\partial}^{\flat+}$ follow from Lemma \ref{lem-O-flat+-cl} and \cite[\aLem \logadiclemlogaffperfclimm{} and \aThm \logadicthmalmostvanhat]{Diao/Lan/Liu/Zhu:lasfr}.  Since $\cF^\partial = \imath_{\proket, *}(\cF)$, and since $\imath_{\proket, *}$ is compatible with limits and colimits \Pth{by \cite[\aProp \logadicpropproketsiteqcqs]{Diao/Lan/Liu/Zhu:lasfr}}, the remaining assertions also follow.
\end{proof}

\begin{lem}\label{lem-property-AAinf-bd}
    Over ${X_\proket}_{/X_K}$, we have the following, for each $J \subset I$:
    \begin{enumerate}
        \item\label{lem-property-AAinf-bd-mod-pi}  $\AAinfX{X_J^\partial}^\partial / (p, [\varpi]) \cong \widehat{\cO}_{X_{J, \proket}^\partial}^{+, \partial} / p \cong \imath_{\proket,*}(\cO_{X_{J, \proket}^\partial}^+ / p)$.

        \item\label{lem-property-AAinf-bd-coh}  For all log affinoid perfectoid object $U$ in ${X_\proket}_{/X_K}$, and all $m, n \geq 1$, the $\Ainf$-module $H^j\bigl(U_\proket, \AAinfX{X_J^\partial}^\partial / (p^m, [\varpi^n])\bigr)$ is almost zero, when $j > 0$; and is almost isomorphic to $\AAinfX{X_J^\partial}^\partial(U) / (p^m, [\varpi^n])$, when $j = 0$.

        \item\label{lem-property-AAinf-bd-lim}  $\AAinfX{X_J^\partial}^\partial \cong \varprojlim_{m, n} \bigl(\AAinfX{X_J^\partial}^\partial / (p^m, [\varpi^n])\bigr)$; and $R^j\varprojlim_{m, n} \bigl(\AAinfX{X_J^\partial}^\partial / (p^m, [\varpi^n])\bigr)$ is almost zero, for all $j > 0$.
    \end{enumerate}
\end{lem}
\begin{proof}
    The assertion \Refeq{\ref{lem-property-AAinf-bd-mod-pi}} follows from Lemmas \ref{lem-from-ket-to-proket} and \ref{lem-O-flat+-cl}, and Definition \ref{def-AAinf-etc-bd}.  Since $H^j(U_\proket, \widehat{\cO}_{X_{J, \proket}^\partial}^{+, \partial} / p^m) \cong H^j(V_\proket, \widehat{\cO}_{X_{J, \proket}^\partial}^+ / p^m)$, where $V$ is the pullback of $U$ as in Lemma \ref{lem-O-flat+-cl}, the assertion \Refeq{\ref{lem-property-AAinf-bd-coh}} follows \Pth{by induction} from \cite[\aProp 7.13]{Scholze:2012-ps} and \cite[\aThm \logadicthmalmostvanhat]{Diao/Lan/Liu/Zhu:lasfr}.  Finally, the assertion \Refeq{\ref{lem-property-AAinf-bd-lim}} follows from \cite[\aLem 3.18]{Scholze:2013-phtra}, \cite[\aProp \logadicproplogaffperfbasis]{Diao/Lan/Liu/Zhu:lasfr}, and the previous two assertions.
\end{proof}

Essentially by definition, and by Lemmas \ref{lem-from-ket-to-proket} and \ref{lem-O-flat+-cl}, we have the following two lemmas:
\begin{lem}\label{lem-AAinf-bd-mor}
    For $J \subset J' \subset I$, we have canonical morphisms $\widehat{\cO}_{X_J^\partial}^\partial \to \widehat{\cO}_{X_{J'}^\partial}^\partial$, $\widehat{\cO}_{X_J^\partial}^{+, \partial} \to \widehat{\cO}_{X_{J'}^\partial}^{+, \partial}$, $\widehat{\cO}_{X_J^\partial}^{\flat, \partial} \to \widehat{\cO}_{X_{J'}^\partial}^{\flat, \partial}$, $\widehat{\cO}_{X_J^\partial}^{\flat+, \partial} \to \widehat{\cO}_{X_{J'}^\partial}^{\flat+, \partial}$, $\AAinfX{X_J^\partial}^\partial \to \AAinfX{X_{J'}^\partial}^\partial$, $\BBinfX{X_J^\partial}^\partial \to \BBinfX{X_{J'}^\partial}^\partial$, $\BBdRX{X_J^\partial}^{+, \partial} \to \BBdRX{X_{J'}^\partial}^{+, \partial}$, and $\BBdRX{X_J^\partial}^\partial \to \BBdRX{X_{J'}^\partial}^\partial$ over $X_\proket$.  For $a \geq a' \geq 0$, we have similar morphisms for the analogous sheaves for $X_{(a), \proket}^\partial$ and $X_{(a'), \proket}^\partial$.
\end{lem}

\begin{lem}\label{lem-BdR-bd-gr}
    For each $J \subset I$, both $\BBdRX{X_J^\partial}^{+, \partial}$ and $\BBdRX{X_J^\partial}^\partial$ admit filtrations induced by powers of $\ker(\theta^\partial: \BBinfX{X_J^\partial}^\partial \to \widehat{\cO}_{X_J^\partial}^\partial)$, which are also induced by those of $\BBdRpX{X}$ and $\BBdRX{X}$.  Over $X_{K, \proket}$, the filtrations are given by multiplication by powers of $\xi$, and induce canonical isomorphisms $\gr^r \BBdRX{X_J^\partial}^{+, \partial} \cong \widehat{\cO}_{X_J^\partial}^\partial(r)$, for $r \geq 0$; and $\gr^r \BBdRX{X_J^\partial}^\partial \cong \widehat{\cO}_{X_J^\partial}^\partial(r)$, for all $r \in \bZ$, where $(r)$ denotes Tate twists as usual.  For each $a \geq 0$, we have similar facts for $\BBdRX{X_{(a)}^\partial}^{+, \partial}$ and $\BBdRX{X_{(a)}^\partial}^\partial$.
\end{lem}

\begin{lem}\label{lem-OBdRp-bd-loc}
    Let us temporarily assume that $X$ is affinoid and admits a toric chart $X \to \bD_k^n := \Spa(k\Talg{T_1, \ldots, T_n}, k^+\Talg{T_1, \ldots, T_n})$ as in \cite[\aSec \logRHsecOBdlexplicit]{Diao/Lan/Liu/Zhu:lrhrv}, with $D$ defined by $\{ T_1 \cdots T_n = 0 \}$.  Let $\widetilde{X} \to X$ be the log affinoid perfectoid object of $X_\proket$ defined as in \cite[\aSec \logRHsecOBdlexplicit]{Diao/Lan/Liu/Zhu:lrhrv}, with associated perfectoid space $\widehat{\widetilde{X}}$.  Let $\xi \in \Ainf$ be as above.  Suppose that $X_J^\partial$ is defined by $\{ T_1 = \cdots = T_a = 0 \}$.  Then, for each $U \in {X_\proket}_{/\widetilde{X}}$, we have a canonical surjective morphism
    \begin{equation}\label{eq-lem-OBdRp-bd-loc-BdR}
        \BBdRpX{X}(U) \Surj \BBdRX{X_J^\partial}^{+, \partial}(U)
    \end{equation}
    inducing, for each $r \geq 1$, an isomorphism
    \begin{equation}\label{eq-lem-OBdRp-bd-loc-BdR-mod-xi-r}
        \bigl(\BBdRpX{X}(U) \big/ \xi^r\bigr) \big/ ([T_1^{s\flat}], \ldots, [T_a^{s\flat}])^\wedge_{s \in \bQ_{> 0}} \Mi \bigl(\BBdRX{X_J^\partial}^{+, \partial}(U) \big/ \xi^r\bigr),
    \end{equation}
    where $([T_1^{s\flat}], \ldots, [T_a^{s\flat}])^\wedge_{s \in \bQ_{> 0}}$ denotes the $p$-adic completion of the ideal generated by $\{ [T_1]^{s\flat}, \ldots, [T_a^{s\flat}] \}_{s \in \bQ_{> 0}}$; and we have a canonical $\BBdRX{X_J^\partial}^{+, \partial}|_{\widetilde{X}}$-linear isomorphism
    \begin{equation}\label{eq-lem-OBdRp-bd-loc-OBdR-bd}
        \OBdlX{X_J^\partial}^{+, \partial}|_{\widetilde{X}} \cong \BBdRX{X_J^\partial}^{+, \partial}|_{\widetilde{X}}[[y_1, \ldots, y_n]].
    \end{equation}
    compatible with the canonical $\BBdRpX{X}|_{\widetilde{X}}$-linear isomorphism
    \begin{equation}\label{eq-lem-OBdRp-bd-loc-OBdR}
        \OBdlpX{X}|_{\widetilde{X}} \cong \BBdRpX{X}|_{\widetilde{X}}[[y_1, \ldots, y_n]],
    \end{equation}
    in \cite[\aProp \logRHpropOBdlploc]{Diao/Lan/Liu/Zhu:lrhrv}.
\end{lem}
\begin{proof}
    Combine Corollary \ref{cor-O-flat+-cl} and \cite[\aCor \logRHcorOBdRplocclimm]{Diao/Lan/Liu/Zhu:lrhrv}.
\end{proof}

\begin{cor}\label{cor-OBdRp-cplx-bd}
    For each $a \geq 0$, we have an exact complex
    \begin{equation}\label{eq-cor-OBdRp-cplx-bd}
    \begin{split}
        0 \to \BBdRX{X_{(a)}^\partial}^{+, \partial} \to \OBdlX{X_{(a)}^\partial}^{+, \partial} & \Mapn{\nabla} \OBdlX{X_{(a)}^\partial}^{+, \partial} \otimes_{\cO_{X_\proket}} \Omega^{\log, 1}_X \\
        & \Mapn{\nabla} \OBdlX{X_{(a)}^\partial}^{+, \partial} \otimes_{\cO_{X_\proket}} \Omega^{\log, 2}_X \to \cdots
    \end{split}
    \end{equation}
    over $X_\proket$.  The statement also holds with $\BBdRX{X_{(a)}^\partial}^{+, \partial}$ and $\OBdlX{X_{(a)}^\partial}^{+, \partial}$ replaced with $\BBdRX{X_{(a)}^\partial}^\partial$ and $\OBdlX{X_{(a)}^\partial}^\partial$, respectively.
\end{cor}
\begin{proof}
    Combine Lemma \ref{lem-OBdRp-bd-loc}, \cite[\aCor \logRHcorlogdRcplx]{Diao/Lan/Liu/Zhu:lrhrv}, and \cite[\aEx \logadicexlogdiffsheafncd]{Diao/Lan/Liu/Zhu:lasfr}.
\end{proof}

\begin{cor}\label{cor-OBdRp-cplx-bd-strict}
    Both the canonical morphisms $\OBdlpX{X} \to \OBdlX{X_J^\partial}^{+, \partial}$ and $\OBdlX{X} \to \OBdlX{X_J^\partial}^\partial$ are strictly compatible with the filtrations on both sides.
\end{cor}
\begin{proof}
    The assertion for $\OBdlpX{X} \to \OBdlX{X_J^\partial}^{+, \partial}$, which is \'etale local in nature, follows from Lemma \ref{lem-OBdRp-bd-loc}.  Then the assertion for $\OBdlX{X} \to \OBdlX{X_J^\partial}^\partial$ follows from the definition of both sides as completions \Pth{see \cite[\aDef \logRHdefOBdl]{Diao/Lan/Liu/Zhu:lrhrv}}.
\end{proof}

\section{Comparison theorems for cohomology with compact support}\label{sec-dR-comp-cpt}

\subsection{Statements of main results}\label{sec-dR-comp-cpt-main}

In this section, we shall retain the setting of Section \ref{sec-bd}, but assume that $k$ is a $p$-adic field, and that $X$ is \emph{proper} over $k$.  As in Section \ref{sec-period-bd}, let $K = \widehat{\AC{k}}$ and $K^+ = \cO_K$, and let $\xi \in \Ainf = W(K^{\flat+})$ be given by \cite[\aLem 6.3]{Scholze:2013-phtra}.  Let $\bL$ be a $\bZ_p$-local system on $X_\ket$, as in \cite[\aDef \logadicdefketlisse]{Diao/Lan/Liu/Zhu:lasfr}.  Fix $I^\Sc \subset I$ as before.  As usual, we shall denote by $(-D^\Sc)$ the tensor product with \Pth{pullbacks of} the invertible ideal defining the divisor $D^\Sc \subset X$.

We will freely use the notation and constructions in \cite[\aSec \logRHseclogRH]{Diao/Lan/Liu/Zhu:lrhrv}.  In particular, we have the ringed spaces $\cX^+ = (X_\an, \cO_X \ho_k \BdRp)$ and $\cX = (X_\an, \cO_X \ho_k \BdR)$, as in \cite[(\logRHeqdefcX)]{Diao/Lan/Liu/Zhu:lrhrv}.  Moreover, we have the notions of \emph{log connections and their log de Rham complexes} on $\cX$ and on $X$, and of \emph{log Higgs bundles and their log Higgs complexes} on $X_K$, as in \cite[\aDef \logRHdeflogconnetc]{Diao/Lan/Liu/Zhu:lrhrv}.  Note that, given a log connection or a log Higgs bundle, its tensor product with \Pth{the pullback of} the invertible ideal defining $D^\Sc \subset X$ is still a log connection or a log Higgs bundle.

\begin{defn}\label{def-dR-Hi-Hdg-coh-cpt}
    For a log connection $\cE$ on $\cX$, we define
    \begin{equation}\label{eq-def-dR-coh-cpt-geom}
        H_{\dR, \Sc}^i(\cU, \cE) := H^i\bigl(\cX, \DRl\bigl(\cE(-D^\Sc)\bigr)\bigr).
    \end{equation}
    Similarly, for a log connection $E$ on $X$, we define
    \begin{equation}\label{eq-def-dR-coh-cpt}
        H_{\dR, \Sc}^i(U_\an, E) := H^i\bigl(X_\an, \DRl\bigl(E(-D^\Sc)\bigr)\bigr),
    \end{equation}
    For a log Higgs bundle $E$ on $X_K$, we define
    \begin{equation}\label{eq-def-Hi-coh-cpt}
        H_{\Hi, \Sc}^i(U_{K, \an}, E) := H^i\bigl(X_{K, \an}, \Hil\bigl(E(-D^\Sc)\bigr)\bigr),
    \end{equation}
    Finally, for a log connection $E$ on $X$ equipped with a decreasing filtration by coherent subsheaves $\Fil^\bullet E$ satisfying the \Pth{usual} Griffiths transversality condition, we define
    \begin{equation}\label{eq-def-Hdg-coh-cpt}
        H_{\Hdg, \Sc}^{a, i - a}(U_\an, E) := H^i\bigl(X_\an, \gr^a \DRl\bigl(E(-D^\Sc)\bigr)\bigr).
    \end{equation}
    Then there is also the Hodge--de Rham spectral sequence
    \begin{equation}\label{eq-def-Hdg-dR-coh-cpt-spec-seq}
        E_1^{a, i - a} := H_{\Hdg, \Sc}^{a, i - a}(U_\an, E) \Rightarrow H_{\dR, \Sc}^i(U_\an, E),
    \end{equation}
    When the eigenvalues of the residues of $\cE$ and $E$ \Pth{along the irreducible components of $D$} are in $\bQ \cap [0, 1)$, in which case $\cE$ and $E$ are the \emph{canonical extensions} of $\cE|_\cU$ and $E|_U$, respectively \Pth{see the discussions in \cite[\aCh 1, \aSec 4]{Andre/Baldassarri:2001-DDA} or \cite[\aSec 11]{Andre/Baldassarri/Cailotto:2020-DDA(2)}}, we shall also write $H_{\dR, \Sc}^i(\cU, \cE|_\cU)$, $H_{\dR, \Sc}^i(U_\an, E|_U)$, and $H_{\Hdg, \Sc}^i(U_\an, E|_U)$, when the meaning of such notation is clear in the context.  In particular, for each $\bZ_p$-local system $\bL$ on $X_\ket$, we shall write $H_{\dR, \Sc}^i\bigl(\cU, \RH(\bL)\bigr)$, $H_{\dR, \Sc}^i\bigl(U_\an, \DdR(\bL)\bigr)$, and $H_{\Hdg, \Sc}^i\bigl(U_\an, \DdR(\bL)\bigr)$, which is justified by \cite[\aThms \logRHthmlogRHgeom(\logRHthmlogRHgeomres) and \logRHthmlogRHarith(\logRHthmlogRHarithres)]{Diao/Lan/Liu/Zhu:lrhrv}.  We shall also abusively write $H_{\Hi, \Sc}^i\bigl(\cU, \Hc(\bL)\bigr)$ instead of $H_{\Hi, \Sc}^i\bigl(\cU, \Hl(\bL)\bigr)$.  \Pth{For simplicity, we shall write $\RH(\bL)$ etc instead of $\RH(\bL|_U)$ etc in such notation.}
\end{defn}

\begin{rk}\label{rem-def-dR-Hi-Hdg-coh-cpt-abuse}
    The definitions above are rather serious abuses of notation, because, a priori, they do depend on $\cE$ and $E$ over the \emph{whole} $X_\an$.   Nevertheless, we will mainly apply them to $\cE = \RHl(\bL)$, $E = \Hl(\bL)$, and $E = \Ddl(\bL)$, for $\bZ_p$-local systems $\bL$ on $X_\ket$.  Since the eigenvalues of the residues of $\RHl(\bL)$ and $\Ddl(\bL)$ are in $\bQ \cap [0, 1)$, the definition of their de Rham cohomology \Pth{with support conditions} is compatible with their analogues in the complex analytic setting using canonical extensions, as in \cite[II, 6]{Deligne:1970-EDR} and \cite[\aSec 2.11 and \aCor 2.12]{Esnault/Viehweg:1992-LVT-B}.
\end{rk}

\begin{rk}\label{rem-def-dR-Hi-Hdg-coh-cpt-conv}
    If $I^\Sc = I$ and hence $D^\Sc = D$ in the above, we shall abusively denote $H_{\dR, \Sc}^i(U_\an, E)$ by $H_{\dR, \cpt}^i(U_\an, E)$.  If $I^\Sc = \emptyset$ and hence $D^\Sc = \emptyset$, we shall abusively denote $H_{\dR, \Sc}^i(U_\an, E)$ by $H_\dR^i(U_\an, E)$, even though $H_\dR^i(U_\an, E)$ is defined using $E$ over the whole compactification $X$.  Nevertheless, for simplicity, we shall still write $H_{\dR, \cpt}^i(U_\an, E|_U)$ and $H_\dR^i(U_\an, E|_U)$, as in the last part of Definition \ref{def-dR-Hi-Hdg-coh-cpt}, when the meaning is clear in the context.  This abusive choice of notation is consistent with our previous choice for the \'etale cohomology \Pth{see Remark \ref{rem-def-H-c-conv}}.  We shall use similar notation for the other cohomology in Definition \ref{def-dR-Hi-Hdg-coh-cpt}.
\end{rk}

\begin{rk}[{\Refcf{} Remark \ref{rem-def-H-c-alt}}]\label{rem-def-dR-Hi-Hdg-coh-cpt-alt}
    We shall denote the objects defined by any subset $I^\Scalt \subset I^\Sc \subset I$ with subscripts \Qtn{$\Scalt$}.  Then the objects with subscripts \Qtn{$\Sc$} admit compatible canonical morphisms to those with subscripts \Qtn{$\Scalt$}.  Also, we shall denote with subscripts \Qtn{$\Snc$} the objects defined with the complementary subset $I^\Snc \subset I$ replacing $I^\Sc \subset I$.
\end{rk}

The main result of this section is the following:
\begin{thm}\label{thm-L-!-coh-comp}
    For each $i \geq 0$, we have a canonical $\Gal(\AC{k} / k)$-equivariant isomorphism
    \begin{equation}\label{eq-thm-L-!-coh-comp-dR-RH}
        H_{\et, \Sc}^i(U_K, \bL) \otimes_{\bZ_p} \BdR \cong H_{\dR, \Sc}^i\bigl(\cU, \RH(\bL)\bigr),
    \end{equation}
    compatible with the filtrations on both sides, and also \Pth{by taking $\gr^0$} a canonical $\Gal(\AC{k} / k)$-equivariant isomorphism
    \begin{equation}\label{eq-thm-L-!-coh-comp-Hi}
        H_{\et, \Sc}^i(U_K, \bL) \otimes_{\bZ_p} K \cong H_{\Hi, \Sc}^i\bigl(U_{K, \an}, \Hc(\bL)\bigr).
    \end{equation}

    Suppose that $\bL|_U$ is a \emph{de Rham} $\bZ_p$-local system on $U_\et$.  Then we also have a canonical $\Gal(\AC{k} / k)$-equivariant isomorphism
    \begin{equation}\label{eq-thm-L-!-coh-comp-dR}
        H_{\et, \Sc}^i(U_K, \bL) \otimes_{\bZ_p} \BdR \cong H_{\dR, \Sc}^i\bigl(U_\an, \DdR(\bL)\bigr) \otimes_k \BdR,
    \end{equation}
    compatible with the filtrations on both sides, and also \Pth{by taking $\gr^0$} a canonical $\Gal(\AC{k} / k)$-equivariant isomorphism
    \begin{equation}\label{eq-thm-L-!-coh-comp-Hdg}
        H_{\et, \Sc}^i(U_K, \bL) \otimes_{\bZ_p} K \cong \oplus_{a + b = i} \, \Bigl(H_{\Hdg, \Sc}^{a, b}\bigl(U_\an, \DdR(\bL)\bigr) \otimes_k K(-a)\Bigr).
    \end{equation}
    Moreover, the Hodge--de Rham spectral sequence
    \begin{equation}\label{eq-thm-L-!-coh-comp-spec-seq}
        E_1^{a, b} := H_{\Hdg, \Sc}^{a, b}\bigl(U_\an, \DdR(\bL)\bigr) \Rightarrow H_{\dR, \Sc}^{a + b}\bigl(U_\an, \DdR(\bL)\bigr)
    \end{equation}
    degenerates on the $E_1$ page.
\end{thm}

The proof of Theorem \ref{thm-L-!-coh-comp} will be carried out in the following subsections.  We shall freely use the notation introduced in Section \ref{sec-period-bd}.  For simplicity, the pullbacks of various sheaves from $X_\proket$ to $X_{K, \proket}$ will be denoted by the same symbols.

\subsection{Period sheaves {$\AAinf^\Sc$} and {$\BBinf^\Sc$}}\label{sec-period-A-inf-B-inf-cpt}

\begin{defn}\label{def-AAinf-cpt}
    Let
    \[
        \AAinfX{X}^\Sc := \ker(\AAinfX{X_{(0)}^\partial}^\partial \to \AAinfX{X_{(1)}^\partial}^\partial)
    \]
    \Pth{see Lemma \ref{lem-AAinf-bd-mor}}, and
    \[
        \BBinfX{X}^\Sc := \AAinfX{X}^\Sc[\tfrac{1}{p}] \cong \AAinfX{X}^\Sc \otimes_{\widehat{\bZ}_p} \widehat{\bQ}_p.
    \]
    We shall omit the subscripts \Qtn{$X$} when the context is clear.
\end{defn}

\begin{rk}\label{rem-def-AAinf-cpt}
    By definition, we have $\AAinfX{X_{(0)}^\partial}^\partial \cong \AAinfX{X}$, $\BBinfX{X_{(0)}^\partial}^\partial \cong \BBinfX{X} \cong \AAinfX{X}[\frac{1}{p}] \cong \AAinfX{X} \otimes_{\widehat{\bZ}_p} \widehat{\bQ}_p$, and $\BBinfX{X}^\Sc \cong \ker(\BBinfX{X_{(0)}^\partial}^\partial \to \BBinfX{X_{(1)}^\partial}^\partial)$.  Moreover, we could have defined $\AAinfX{X}^\Sc$ as a derived limit as in \Refeq{\ref{eq-cplx-AAinf-deg-zero}} below \Pth{with $\widehat{\bL} = \widehat{\bZ}_p$ there}, without using the boundary stratification.
\end{rk}

\begin{lem}\label{lem-fil-AAinf-cpt}
    Both $\AAinfX{X}^\Sc$ and $\BBinfX{X}^\Sc$ are equipped with filtrations induced by those of $\AAinfX{X}$ and $\BBinfX{X}$, respectively.  Over $X_{K, \proket}$, they agree with the filtrations defined more directly by multiplication by powers of $\xi$, where $\xi$ is as in Section \ref{sec-period-bd}, and we have compatible canonical isomorphisms $\AAinfX{X}^\Sc \otimes_{\Ainf} (\Ainf / \xi^r) \Mi \AAinfX{X}^\Sc / \xi^r$ and $\BBinfX{X}^\Sc \otimes_{\Binf} (\Binf / \xi^r) \Mi \BBinfX{X}^\Sc / \xi^r$, for each $r \geq 1$.
\end{lem}
\begin{proof}
    By Definition \ref{def-AAinf-cpt} and Remark \ref{rem-def-AAinf-cpt}, $\AAinfX{X}^\Sc$ and $\BBinfX{X}^\Sc$ are subsheaves of $\AAinfX{X}$ and $\BBinfX{X}$, respectively, and the first assertion follows.  Over $X_{K, \proket}$, by \cite[\aLem 6.3]{Scholze:2013-phtra}, the filtrations on $\AAinfX{X}$ and $\BBinfX{X}$ are defined by multiplication by powers of $\xi$, and the same is true for all the sheaves $\AAinfX{X_J^\partial}^\partial$ and $\BBinfX{X_J^\partial}^\partial$ \Pth{see Definition \ref{def-AAinf-etc-bd} and Corollary \ref{cor-O-flat+-cl}}.  Hence, $\xi^r$ acts with zero kernels on $\AAinfX{X}^\Sc$ and $\BBinfX{X}^\Sc$, for each $r \geq 1$, and the second assertion also follows.
\end{proof}

The goal of this subsection is to prove the following:
\begin{prop}\label{prop-coh-AAinf-cpt}
    We have a canonical $\Gal(\AC{k} / k)$-equivariant almost isomorphism \Pth{of $\Ainf$-modules, as in Remark \ref{rem-almost-Ainf}}
    \begin{equation}\label{eq-prop-coh-AAinf-cpt-comp-Ainf}
        H^i_{\et, \Sc}(U_K, \bL) \otimes_{\bZ_p} \Ainf \cong H^i\bigl(X_{K, \proket}, \widehat{\bL} \otimes_{\widehat{\bZ}_p} \AAinfX{X}^\Sc\bigr),
    \end{equation}
    which induces by inverting $p$ a canonical $\Gal(\AC{k} / k)$-equivariant almost isomorphism
    \begin{equation}\label{eq-prop-coh-AAinf-cpt-comp-Binf}
        H^i_{\et, \Sc}(U_K, \bL) \otimes_{\bZ_p} \Binf \cong H^i\bigl(X_{K, \proket}, \widehat{\bL} \otimes_{\widehat{\bZ}_p} \BBinfX{X}^\Sc\bigr).
    \end{equation}
    Moreover, the isomorphism \Refeq{\ref{eq-prop-coh-AAinf-cpt-comp-Binf}} is compatible with the filtrations defined by multiplication by powers of $\xi$ \Pth{\Refcf{} Lemma \ref{lem-fil-AAinf-cpt}}; and, for all $r \geq 1$, we have compatible canonical $\Gal(\AC{k} / k)$-equivariant isomorphisms \Pth{see Remark \ref{rem-almost-Ainf}}
    \begin{equation}\label{eq-prop-coh-AAinf-cpt-comp-Binf-gr}
        H^i_{\et, \Sc}(U_K, \bL) \otimes_{\bZ_p} (\Binf / \xi^r) \cong H^i\bigl(X_{K, \proket}, \widehat{\bL} \otimes_{\widehat{\bZ}_p} (\BBinfX{X}^\Sc / \xi^r)\bigr).
    \end{equation}
\end{prop}

Let $\varpi \in K^{\flat+}$ be such that $\varpi^\sharp = p$.  We begin with the following consequence of the primitive comparison isomorphism \Pth{see Theorem \ref{thm-L-!-prim-comp}}:
\begin{lem}\label{lem-coh-Ainf}
    For each $i \geq 0$ and all $m, n \geq 1$, we have canonical $\Gal(\AC{k} / k)$-equivariant almost isomorphisms
    \begin{equation}\label{eq-lem-coh-Ainf}
    \begin{split}
        & H^i_{\et, \Sc}(U_K, \bL) \otimes_{\bZ_p} \Ainf \Mi R\varprojlim_{m, n}\Bigl(H^i_{\et, \Sc}(U_K, \bL_m) \otimes_{\bZ_p} \bigl(\Ainf / (p^m, [\varpi^n])\bigr) \Bigr) \\
        & \Mi H^i\Bigl(X_{K, \proket}, R\varprojlim_{m, n} \Bigl(\bigl(\upsilon_X^{-1} \, \jmath_{\ket, !}^\Sc(\bL_m|_{U^\Sc_\ket})\bigr) \otimes_{\widehat{\bZ}_p} \bigl(\AAinfX{X} / (p^m, [\varpi^n])\bigr)\Bigr)\Bigr).
    \end{split}
    \end{equation}
\end{lem}
\begin{proof}
    By Lemma \ref{lem-def-H-c-fin-Z-p}, $H^i_{\et, \Sc}(U_K, \bL) \cong \varprojlim_m H^i_{\et, \Sc}(U_K, \bL_m)$ is a finitely generated $\bZ_p$-module, and $H^i_{\et, \Sc}(U_K, \bL) / p^m \Mi H^i_{\et, \Sc}(U_K, \bL_m)$ for all sufficiently large $m$.  Therefore, since $\Ainf \cong \varprojlim_m (\Ainf / p^m) \cong \varprojlim_{m, n} \bigl(\Ainf / (p^m, [\varpi^n])\bigr)$, we obtain
    \[
    \begin{split}
        & H^i_{\et, \Sc}(U_K, \bL) \otimes_{\bZ_p} \Ainf \Mi R\varprojlim_m \bigl(H^i_{\et, \Sc}(U_K, \bL) \otimes_{\bZ_p} (\Ainf / p^m)\bigr) \\
        & \Mi R\varprojlim_m \bigl(H^i_{\et, \Sc}(U_K, \bL_m) \otimes_{\bZ_p} (\Ainf / p^m)\bigr) \\
        & \Mi R\varprojlim_{m, n} \Bigl(H^i_{\et, \Sc}(U_K, \bL_m) \otimes_{\bZ_p} \bigl(\Ainf/ (p^m, [\varpi^n])\bigr)\Bigr)
    \end{split}
    \]
    \Pth{with vanishing higher limits}, whose composition is the first almost isomorphism in \Refeq{\ref{eq-lem-coh-Ainf}}.  By using the almost isomorphism \Refeq{\ref{eq-thm-L-!-prim-comp}} in Theorem \ref{thm-L-!-prim-comp}, by Lemma
    \ref{lem-property-AAinf-bd}\Refenum{\ref{lem-property-AAinf-bd-mod-pi}}, and by the same inductive argument as in the proof of \cite[\aThm 8.4]{Scholze:2013-phtra}, we obtain the second almost isomorphism in \Refeq{\ref{eq-lem-coh-Ainf}}.  By their very constructions, both the almost isomorphisms in \Refeq{\ref{eq-lem-coh-Ainf}} are canonical and independent of the choices, and hence $\Gal(\AC{k} / k)$-equivariant, as desired.
\end{proof}

\begin{lem}\label{lem-Scholze-cplx}
    Let $\{ \cF_i \}_{i \in \bZ_{\geq 1}}$ be an inverse system of abelian sheaves on a site $T$, and let $\{ 0 \to \cF_i \to \cF_{i, 0} \to \cF_{i, 1} \to \cdots \cF_{i, a} \to \cdots \}_{i \in \bZ_{\geq 1}}$ be an inverse system of exact complexes.  Assume that there exists a basis $\cB$ of the site $T$ such that, for each $U \in \cB$, the following conditions hold:
    \begin{enumerate}
        \item\label{lem-Scholze-cplx-coh-van}  $H^b(U, \cF_{i, a}) = 0$, for all $a \geq 0$, $b > 0$, and $i \geq 1$.

        \item\label{lem-Scholze-cplx-ex}  The complex $0 \to \cF_i(U) \to \cF_{i, 0}(U) \to \cF_{i, 1}(U) \to \cdots \to \cF_{i, a}(U) \to \cdots$ is exact, for each $i \geq 0$.

        \item\label{lem-Scholze-cplx-lim-van}  $\cF_{i + 1, a}(U) \to \cF_{i, a}(U)$ is surjective, for all $a \geq 0$.
    \end{enumerate}
    Then, for $? = \emptyset$ or any $a \geq 0$, we have $R^j \varprojlim_i \cF_{i, ?} = 0$ and $H^j(U, \varprojlim_i \cF_{i, ?}) = 0$, for $j > 0$; and $\bigl(\varprojlim_i \cF_{i, ?}\bigr)(U) \cong \varprojlim_i\bigl(\cF_{i, ?}(U)\bigr)$.  Moreover, the complex $0 \to \varprojlim_i \cF_i \to \varprojlim_i \cF_{i, 0} \to \varprojlim_i \cF_{i, 1} \to \cdots$ is also exact.
\end{lem}
\begin{proof}
    Since $\cF_{i, \bullet}$ is a resolution of $\cF_i$, we have a \Pth{filtration} spectral sequence $E_1^{a, b} := H^b\bigl(U, \cF_{i, a}) \Rightarrow H^{a + b}(U, \cF_i)$, which is concentrated on the terms $E_1^{a, 0}$, by assumption \Refenum{\ref{lem-Scholze-cplx-coh-van}}.  Then the spectral sequence degenerates on the $E_2$ page, and
    \begin{equation}\label{eq-lem-Scholze-cplx-coh-van}
        H^j(U, \cF_i) = 0,
    \end{equation}
    for all $i \geq 1$ and $j > 0$.  Similarly, by assumption \Refenum{\ref{lem-Scholze-cplx-ex}}, we have a spectral sequence $E_1^{a, b} := R^b\varprojlim_i\bigl(\cF_{i, a}(U)\bigr) \Rightarrow R^{a + b}\varprojlim_i\bigl(\cF_i(U)\bigr)$, which is concentrated on the terms $E_1^{a, 0}$, because $R^b\varprojlim_i\bigl(\cF_{i, a}(U)\bigr) = 0$ for all $b > 0$, by assumption \Refenum{\ref{lem-Scholze-cplx-lim-van}}.  Then the spectral sequence degenerates on the $E_2$ page, and
    \begin{equation}\label{eq-lem-Scholze-cplx-lim-van}
        R^j\varprojlim_i\bigl(\cF_i(U)\bigr) = 0
    \end{equation}
    for all $j > 0$.  Hence, by \Refeq{\ref{eq-lem-Scholze-cplx-coh-van}} and \Refeq{\ref{eq-lem-Scholze-cplx-lim-van}}, by assumptions \Refeq{\ref{lem-Scholze-cplx-coh-van}} and \Refeq{\ref{lem-Scholze-cplx-lim-van}}, and by \cite[\aLem 3.18]{Scholze:2013-phtra}, the first assertion of the lemma follows.  Consequently, by \cite[\aProp 1.12.4]{Kashiwara/Shapira:1990-SM}, we have an exact complex $0 \to (\varprojlim_i \cF_i)(U) \to (\varprojlim_i \cF_{i, 0})(U)\to (\varprojlim_i \cF_{i,  1})(U) \to \cdots$.  Since $U$ is an arbitrary object in the basis $\cB$ of $T$, the second assertion of the lemma also follows, as desired.
\end{proof}

\begin{lem}\label{lem-cplx-AAinf}
    For each $m \geq 1$ and each $n \geq 1$, we have a canonical $\Gal(\AC{k} / k)$-equivariant exact complex
    \begin{equation}\label{eq-lem-cplx-AAinf-tor}
    \begin{split}
        0 & \to \bigl(\upsilon_X^{-1} \, \jmath_{\ket, !}^\Sc(\bL_m|_{U^\Sc_\ket})\bigr) \otimes_{\widehat{\bZ}_p} \bigl(\AAinfX{X} / (p^m, [\varpi^n])\bigr) \\
        & \to \widehat{\bL} \otimes_{\widehat{\bZ}_p} \bigl(\AAinfX{X_{(0)}^\partial}^\partial / (p^m, [\varpi^n])\bigr) \to \widehat{\bL} \otimes_{\widehat{\bZ}_p} \bigl(\AAinfX{X_{(1)}^\partial}^\partial / (p^m, [\varpi^n])\bigr) \\
        & \to \cdots \to \widehat{\bL} \otimes_{\widehat{\bZ}_p} \bigl(\AAinfX{X_{(a)}^\partial}^\partial / (p^m, [\varpi^n])\bigr) \to \cdots
    \end{split}
    \end{equation}
    over $X_{K, \proket}$.  Consequently, we have a canonical $\Gal(\AC{k} / k)$-equivariant almost quasi-isomorphism between
    \[
        R\varprojlim_{m, n} \Bigl(\bigl(\upsilon_X^{-1} \, \jmath_{\ket, !}^\Sc(\bL_m|_{U^\Sc_\ket})\bigr) \otimes_{\widehat{\bZ}_p} \bigl(\AAinfX{X} / (p^m, [\varpi^n])\bigr)\Bigr)
    \]
    \Pth{with almost vanishing higher limits} and
    \[
        \widehat{\bL} \otimes_{\widehat{\bZ}_p} \AAinfX{X_{(0)}^\partial}^\partial \to \widehat{\bL} \otimes_{\widehat{\bZ}_p} \AAinfX{X_{(1)}^\partial}^\partial \to \cdots \to \widehat{\bL} \otimes_{\widehat{\bZ}_p} \AAinfX{X_{(a)}^\partial}^\partial \to \cdots
    \]
    \Pth{which is almost exact except in degree $0$} over $X_{K, \proket}$.  Since $\widehat{\bL}$ is a local system, by Definition \ref{def-AAinf-cpt}, we obtain a canonical $\Gal(\AC{k} / k)$-equivariant almost isomorphism
    \begin{equation}\label{eq-cplx-AAinf-deg-zero}
        \widehat{\bL} \otimes_{\widehat{\bZ}_p} \AAinfX{X}^\Sc \Mi R\varprojlim_{m, n} \Bigl(\bigl(\upsilon_X^{-1} \, \jmath_{\ket, !}^\Sc(\bL_m|_{U^\Sc_\ket})\bigr) \otimes_{\widehat{\bZ}_p} \bigl(\AAinfX{X} / (p^m, [\varpi^n])\bigr)\Bigr).
    \end{equation}
\end{lem}
\begin{proof}
    The first assertion follows inductively from the exactness of \Refeq{\ref{eq-thm-L-!-prim-comp-resol}}, by using Lemma \ref{lem-property-AAinf-bd}\Refenum{\ref{lem-property-AAinf-bd-mod-pi}} and the canonical morphisms induced by the short exact sequence $0 \to p \bL_m \to \bL_m \to \bL_m / p \to 0$.  Let $U$ be any log perfectoid affinoid object in ${X_\proket}_{/X_K}$, which we may assume to trivialize $\widehat{\bL}$, because such objects form a basis, by \cite[\aProp \logadicproplogaffperfbasis]{Diao/Lan/Liu/Zhu:lasfr}.  By induction on $m$ and $n$, by \cite[\aLem \logadiclemclimmOplusp]{Diao/Lan/Liu/Zhu:lasfr} and Lemmas \ref{lem-!-resol} and \ref{lem-property-AAinf-bd}, and by downward induction on $a$ using the finiteness of $|I^\Sc|$, we see that \Refeq{\ref{eq-lem-cplx-AAinf-tor}} is almost exact when evaluated on $U$, and so that the second assertion follows from the almost version of Lemma \ref{lem-Scholze-cplx}, as desired.
\end{proof}

Thus, we are ready for the following:
\begin{proof}[Proof of Proposition \ref{prop-coh-AAinf-cpt}]
    By combining Lemma \ref{lem-coh-Ainf} and \Refeq{\ref{eq-cplx-AAinf-deg-zero}}, we obtain the two almost isomorphisms \Refeq{\ref{eq-prop-coh-AAinf-cpt-comp-Ainf}} and \Refeq{\ref{eq-prop-coh-AAinf-cpt-comp-Binf}}, which are naturally compatible with the multiplication by powers of $\xi$ on both sides.  For each $r \geq 1$, since
    \[
        H^i_{\et, \Sc}(U_K, \bL) \otimes_{\bZ_p} (\Binf / \xi^r) \cong \bigl(H^i_{\et, \Sc}(U_K, \bL) \otimes_{\bZ_p} \bQ_p\bigr) \otimes_{\bQ_p} (\Binf / \xi^r)
    \]
    and since $H^i_{\et, \Sc}(U_K, \bL) \otimes_{\bZ_p} \bQ_p$ is a finite-dimensional $\bQ_p$-vector space, by using the canonical almost isomorphism \Refeq{\ref{eq-prop-coh-AAinf-cpt-comp-Binf}} just established, we see that $\xi^r$ acts with almost zero kernel on $H^i(X_{K, \proket}, \widehat{\bL} \otimes_{\widehat{\bZ}_p} \BBinfX{X}^\Sc)$.  Therefore, all the connecting morphisms in the long exact sequence associated with the short exact sequence $0 \to \widehat{\bL} \otimes_{\widehat{\bZ}_p} \BBinfX{X}^\Sc \Mapn{\xi^r} \widehat{\bL} \otimes_{\widehat{\bZ}_p} \BBinfX{X}^\Sc \to \widehat{\bL} \otimes_{\widehat{\bZ}_p} (\BBinfX{X}^\Sc / \xi^r) \to 0$ over $X_{K, \proket}$ \Pth{see Lemma \ref{lem-fil-AAinf-cpt}} are almost zero, and we obtain a canonical isomorphism
    \begin{equation}\label{eq-prop-coh-AAinf-cpt-Binf-mod-xi-r}
        H^i\bigl(X_{K, \proket}, \widehat{\bL} \otimes_{\widehat{\bZ}_p} \BBinfX{X}^\Sc\bigr) / \xi^r \Mi H^i\bigl(X_{K, \proket}, \widehat{\bL} \otimes_{\widehat{\bZ}_p} (\BBinfX{X}^\Sc / \xi^r)\bigr).
    \end{equation}
    \Pth{Again, see Remark \ref{rem-almost-Ainf}.}  It follows that \Refeq{\ref{eq-prop-coh-AAinf-cpt-comp-Binf}} is compatible with the filtrations, and its combination with \Refeq{\ref{eq-prop-coh-AAinf-cpt-Binf-mod-xi-r}} induces the desired isomorphism \Refeq{\ref{eq-prop-coh-AAinf-cpt-comp-Binf-gr}}.
\end{proof}

\subsection{Period sheaves {$\BBdR^{\Sc, +}$} and {$\BBdR^\Sc$}}\label{sec-period-B-dR-cpt}

\begin{defn}\label{def-O-hat-cpt}
    For $? = \emptyset$, $+$, $\flat$, or $\flat+$, let $\widehat{\cO}_X^{\Sc, ?} := \ker(\widehat{\cO}_{X_{(0)}^\partial}^{?, \partial} \to \widehat{\cO}_{X_{(1)}^\partial}^{?, \partial})$.
\end{defn}

\begin{defn}\label{def-BBdRp-cpt}
    Let
    \[
        \BBdRX{X}^{\Sc, +} := \ker(\BBdRX{X_{(0)}^\partial}^{+, \partial} \to \BBdRX{X_{(1)}^\partial}^{+, \partial})
    \]
    and
    \[
        \BBdRX{X}^\Sc := \ker(\BBdRX{X_{(0)}^\partial}^\partial \to \BBdRX{X_{(1)}^\partial}^\partial)
    \]
    \Pth{see Lemma \ref{lem-AAinf-bd-mor}}.  We shall omit the subscripts \Qtn{$X$} when the context is clear.
\end{defn}

\begin{rk}\label{rem-def-BBdRp-cpt}
    By definition, we have $\BBdRX{X_{(0)}^\partial}^{+, \partial} \cong \BBdRpX{X}$; and we have $\BBdRX{X}^\Sc \cong \BBdRX{X}^{\Sc, +}[\tfrac{1}{\xi}] \cong \BBdRX{X}^{\Sc, +} \otimes_{\BdRp} \BdR$ over $X_{K, \proket}$.  Moreover, we could have defined $\BBdRX{X}^{\Sc, +}$ as a derived limit as in \Refeq{\ref{eq-lem-BBdRp-cpt}} below \Pth{with $\widehat{\bL} = \widehat{\bZ}_p$ there}, without reference to the boundary stratification.
\end{rk}

The goal of this subsection is to prove the following generalization of \cite[\aLem \logRHlemketBdR]{Diao/Lan/Liu/Zhu:lrhrv}:

\begin{prop}\label{prop-L-!-coh-comp-proket}
    For each $i \geq 0$, we have a canonical $\Gal(\AC{k} / k)$-equivariant isomorphism
    \begin{equation}\label{eq-prop-BBdRp-cpt-comp-dR}
        H^i_{\et, \Sc}(U_K, \bL) \otimes_{\bZ_p} \BdRp \cong H^i(X_{K, \proket}, \widehat{\bL} \otimes_{\widehat{\bZ}_p} \BBdRX{X}^{\Sc, +}),
    \end{equation}
    compatible with filtrations on both sides, and also \Pth{by taking $\gr^0$} a canonical $\Gal(\AC{k} / k)$-equivariant isomorphism
    \begin{equation}\label{eq-prop-BBdRp-cpt-comp-Higgs}
        H^i_{\et, \Sc}(U_K, \bL) \otimes_{\bZ_p} K \cong H^i(X_{K, \proket}, \widehat{\bL} \otimes_{\widehat{\bZ}_p} \widehat{\cO}_{X_{K, \proket}}^\Sc).
    \end{equation}
\end{prop}

\begin{lem}\label{lem-BBdRp-cpt}
    Over $X_{K, \proket}$, we have a canonical $\Gal(\AC{k} / k)$-equivariant isomorphism
    \begin{equation}\label{eq-lem-BBdRp-cpt}
        \widehat{\bL} \otimes_{\widehat{\bZ}_p} \BBdRX{X}^{\Sc, +} \cong R\varprojlim_r \bigl(\widehat{\bL} \otimes_{\widehat{\bZ}_p} (\BBinfX{X}^\Sc / \xi^r)\bigr),
    \end{equation}
    \Pth{with vanishing higher limits} and a canonical $\Gal(\AC{k} / k)$-equivariant exact complex
    \begin{equation}\label{eq-lem-BBdRp-cpt-cplx}
    \begin{split}
        0 & \to \widehat{\bL} \otimes_{\widehat{\bZ}_p} \BBdRX{X}^{\Sc, +} \to \widehat{\bL} \otimes_{\widehat{\bZ}_p} \BBdRX{X_{(0)}^\partial}^{+, \partial} \\
        & \to \widehat{\bL} \otimes_{\widehat{\bZ}_p} \BBdRX{X_{(1)}^\partial}^{+, \partial} \to \cdots \to \widehat{\bL} \otimes_{\widehat{\bZ}_p} \BBdRX{X_{(a)}^\partial}^{+, \partial} \to \cdots,
    \end{split}
    \end{equation}
    which is strictly compatible with the filtrations defined by multiplication by powers of $\xi$, and induces, for each $r \in \bZ$, a canonical $\Gal(\AC{k} / k)$-equivariant isomorphism
    \begin{equation}\label{eq-lem-BBdRp-cpt-gr}
        \gr^r(\widehat{\bL} \otimes_{\widehat{\bZ}_p} \BBdRX{X}^\Sc) \cong \widehat{\bL} \otimes_{\widehat{\bZ}_p} \gr^r(\BBdRX{X}^\Sc) \cong \widehat{\bL} \otimes_{\widehat{\bZ}_p} \widehat{\cO}_{X_{K, \proket}}^\Sc(r)
    \end{equation}
    and a canonical $\Gal(\AC{k} / k)$-equivariant exact complex
    \begin{equation}\label{eq-lem-BBdRp-cpt-cplx-gr}
    \begin{split}
        0 & \to \widehat{\bL} \otimes_{\widehat{\bZ}_p} \widehat{\cO}_{X_{K, \proket}}^\Sc \to \widehat{\bL} \otimes_{\widehat{\bZ}_p} \widehat{\cO}_{X_{(0), K, \proket}^\partial} \\
        & \to \widehat{\bL} \otimes_{\widehat{\bZ}_p} \widehat{\cO}_{X_{(1), K, \proket}^\partial} \to \cdots \to \widehat{\bL} \otimes_{\widehat{\bZ}_p} \widehat{\cO}_{X_{(a), K, \proket}^\partial} \to \cdots.
    \end{split}
    \end{equation}
\end{lem}
\begin{proof}
    Since $\widehat{\bL}$ is a local system, by forming the tensor product of the short exact sequence $0 \to \Binf \Mapn{\xi^r} \Binf \to \Binf / \xi^r \to 0$ with the complex $\widehat{\bL} \otimes_{\widehat{\bZ}_p} \BBinfX{X_{(\bullet)}^\partial}^\partial$ \Pth{which is almost exact except in degree $0$, by Lemma \ref{lem-cplx-AAinf}}, we obtain a short exact sequence of complexes $0 \to \widehat{\bL} \otimes_{\widehat{\bZ}_p} \BBinfX{X_{(\bullet)}^\partial}^\partial \Mapn{\xi^r} \widehat{\bL} \otimes_{\widehat{\bZ}_p} \BBinfX{X_{(\bullet)}^\partial}^\partial \to \widehat{\bL} \otimes_{\widehat{\bZ}_p} (\BBinfX{X_{(\bullet)}^\partial}^\partial / \xi^r) \to 0$, inducing an almost long exact sequence with only three nonzero terms in the beginning $0 \to \widehat{\bL} \otimes_{\widehat{\bZ}_p} \BBinfX{X}^\Sc \to \widehat{\bL} \otimes_{\widehat{\bZ}_p} \BBinfX{X}^\Sc \to \widehat{\bL} \otimes_{\widehat{\bZ}_p} (\BBinfX{X}^\Sc / \xi^r) \to 0 \to \cdots$, showing that we have a canonical isomorphism $\bigl(\widehat{\bL} \otimes_{\widehat{\bZ}_p} \BBinfX{X}^\Sc\bigr) / \xi^r \Mi \widehat{\bL} \otimes_{\widehat{\bZ}_p} (\BBinfX{X}^\Sc / \xi^r)$ \Pth{\Refcf{} Lemma \ref{lem-fil-AAinf-cpt}} and a canonical $\Gal(\AC{k} / k)$-equivariant exact complex
    \begin{equation}\label{eq-lem-BBdRp-cpt-cplx-BBinf-xi-r}
    \begin{split}
        0 & \to \widehat{\bL} \otimes_{\widehat{\bZ}_p} (\BBinfX{X}^\Sc / \xi^r) \to \widehat{\bL} \otimes_{\widehat{\bZ}_p} (\BBinfX{X_{(0)}^\partial}^\partial / \xi^r) \\
        & \to \widehat{\bL} \otimes_{\widehat{\bZ}_p} (\BBinfX{X_{(1)}^\partial}^\partial / \xi^r) \to \cdots \to \widehat{\bL} \otimes_{\widehat{\bZ}_p} (\BBinfX{X_{(a)}^\partial}^\partial / \xi^r) \to \cdots.
    \end{split}
    \end{equation}
    When $r = 1$, this gives the exact complex \Refeq{\ref{eq-lem-BBdRp-cpt-cplx-gr}}, because $\BBinfX{X_{(a)}^\partial}^\partial / \xi \cong \widehat{\cO}_{X_{(a)}^\partial}^\partial$.  \Pth{Alternatively, we can obtain the exact complex \Refeq{\ref{eq-lem-BBdRp-cpt-cplx-gr}} more directly from the exact complex \Refeq{\ref{eq-thm-L-!-prim-comp-resol}}.}  More generally, let $U$ be any log perfectoid affinoid object in ${X_\proket}_{/X_K}$, which we may assume to trivialize $\widehat{\bL}$, because such objects form a basis, by \cite[\aProp \logadicproplogaffperfbasis]{Diao/Lan/Liu/Zhu:lasfr}.  By induction on $r$, by the exactness of \Refeq{\ref{eq-lem-BBdRp-cpt-cplx-BBinf-xi-r}}, by \cite[\aThm \logadicthmalmostvanhat]{Diao/Lan/Liu/Zhu:lasfr}, and by downward induction on $a$ using the finiteness of $|I^\Sc|$, we see that \Refeq{\ref{eq-lem-BBdRp-cpt-cplx-BBinf-xi-r}} is exact when evaluated on $U$, and so that Lemma \ref{lem-Scholze-cplx} applies, from which we obtain that $R^j\varprojlim_r \bigl(\widehat{\bL} \otimes_{\widehat{\bZ}_p} (\BBinfX{X}^\Sc / \xi^r)\bigr) = 0$, for all $j > 0$, and that the canonical $\Gal(\AC{k} / k)$-equivariant complex
    \begin{equation}\label{eq-lem-BBdRp-cpt-cplx-pre}
    \begin{split}
        0 & \to \varprojlim_r \bigl(\widehat{\bL} \otimes_{\widehat{\bZ}_p} (\BBinfX{X}^{\Sc, +} / \xi^r)\bigr) \to \widehat{\bL} \otimes_{\widehat{\bZ}_p} \BBdRX{X_{(0)}^{+, \partial}} \\
        & \to \widehat{\bL} \otimes_{\widehat{\bZ}_p} \BBdRX{X_{(1)}^{+, \partial}} \to \cdots \to \widehat{\bL} \otimes_{\widehat{\bZ}_p} \BBdRX{X_{(a)}^{+, \partial}} \to \cdots
    \end{split}
    \end{equation}
    is exact.  Since $\widehat{\bL}$ is a local system, by Definition \ref{def-BBdRp-cpt}, we obtain an exact sequence
    \begin{equation}\label{eq-lem-BBdRp-cpt-cplx-init}
        0 \to \widehat{\bL} \otimes_{\widehat{\bZ}_p} \BBdRX{X}^{\Sc, +} \to \widehat{\bL} \otimes_{\widehat{\bZ}_p} \BBdRX{X_{(0)}^\partial}^{+, \partial} \to \widehat{\bL} \otimes_{\widehat{\bZ}_p} \BBdRX{X_{(1)}^\partial}^{+, \partial}
    \end{equation}
    as in the first few terms of \Refeq{\ref{eq-lem-BBdRp-cpt-cplx}}.  Hence, we obtain both \Refeq{\ref{eq-lem-BBdRp-cpt}} and \Refeq{\ref{eq-lem-BBdRp-cpt-cplx}} by comparing \Refeq{\ref{eq-lem-BBdRp-cpt-cplx-pre}} and \Refeq{\ref{eq-lem-BBdRp-cpt-cplx-init}}, which are strictly compatible with filtrations because \Refeq{\ref{eq-lem-BBdRp-cpt-cplx-pre}} and \Refeq{\ref{eq-lem-BBdRp-cpt-cplx-init}} are, by their very constructions above.  Since
    \[
        \gr^r(\widehat{\bL} \otimes_{\widehat{\bZ}_p} \BBdRX{X_{(a)}^\partial}^{+, \partial}) \cong \widehat{\bL} \otimes_{\widehat{\bZ}_p} \widehat{\cO}_{X_{(a), K, \proket}^\partial}^\partial(a)
    \]
    by Lemma \ref{lem-BdR-bd-gr}, for all $a \geq 0$; and since
    \[
        \widehat{\bL} \otimes_{\widehat{\bZ}_p} \widehat{\cO}_{X_{K, \proket}}^\Sc \cong \ker(\widehat{\bL} \otimes_{\widehat{\bZ}_p} \widehat{\cO}_{X_{(0), K, \proket}^\partial}^\partial \to \widehat{\bL} \otimes_{\widehat{\bZ}_p} \widehat{\cO}_{X_{(1), K, \proket}^\partial}^\partial),
    \]
    by Definition \ref{def-O-hat-cpt}, we also obtain \Refeq{\ref{eq-lem-BBdRp-cpt-gr}} and \Refeq{\ref{eq-lem-BBdRp-cpt-cplx-gr}}, as desired.
\end{proof}

\begin{proof}[Proof of Proposition \ref{prop-L-!-coh-comp-proket}]
    Since $H^i_{\et, \Sc}(U_K, \bL) \otimes_{\bZ_p} \bQ_p$ is a finite $\bQ_p$-module \Pth{see Lemma \ref{lem-def-H-c-fin-Z-p}}, and since $\BdRp \cong \varprojlim_r (\Binf / \xi^r)$, by Proposition \ref{prop-coh-AAinf-cpt}, we obtain
    \[
    \begin{split}
        H^i_{\et, \Sc}(U_K, \bL) \otimes_{\bZ_p} \BdRp & \cong \bigl((H^i_{\et, \Sc}(U_K, \bL) \otimes_{\bZ_p} \bQ_p) \otimes_{\bQ_p} \Binf\bigr) \otimes_{\Binf} \BdRp \\
        & \cong R\varprojlim_r \bigl(H^i(X_{K, \proket}, \widehat{\bL} \otimes_{\widehat{\bZ}_p} \BBinfX{X}^\Sc) \otimes_{\Binf} (\Binf / \xi^r)\bigr) \\
        & \cong R\varprojlim_r H^i\bigl(X_{K, \proket}, \widehat{\bL} \otimes_{\widehat{\bZ}_p} (\BBinfX{X}^\Sc / \xi^r)\bigr)
    \end{split}
    \]
    \Pth{with vanishing higher limits}, which are compatible with the filtrations defined by multiplication by powers of $\xi$.  Thus, the proposition follows from Lemma \ref{lem-BBdRp-cpt} and the standard isomorphism $R\varprojlim_r R\Gamma\bigl(X_{K, \proket}, \widehat{\bL} \otimes_{\widehat{\bZ}_p} (\BBinfX{X}^\Sc / \xi^r)\bigr) \cong R\Gamma\bigl(X_{K, \proket}, R\varprojlim_r \bigl(\widehat{\bL} \otimes_{\widehat{\bZ}_p} (\BBinfX{X}^\Sc / \xi^r)\bigr)\bigr)$, as desired.
\end{proof}

\subsection{Period sheaves {$\OBdl^{\Sc, +}$} and {$\OBdl^\Sc$}, and Poincar\'e lemma}\label{sec-period-OB-dR-cpt}

\begin{defn}\label{def-OBdRp-cpt}
    Let
    \[
        \OBdlX{X}^{\Sc, +} := \ker(\OBdlX{X_{(0)}^\partial}^{+, \partial} \to \OBdlX{X_{(1)}^\partial}^{+, \partial});
    \]
    \[
        \OBdlX{X}^\Sc := \ker(\OBdlX{X_{(0)}^\partial}^\partial \to \OBdlX{X_{(1)}^\partial}^\partial);
    \]
    \[
        \Fil^r \OBdlX{X}^{\Sc, +} := \ker(\Fil^r \OBdlX{X_{(0)}^\partial}^{+, \partial} \to \Fil^r \OBdlX{X_{(1)}^\partial}^{+, \partial}),
    \]
    for $r \geq 0$;
    \[
        \Fil^r \OBdlX{X}^\Sc := \ker(\Fil^r \OBdlX{X_{(0)}^\partial}^\partial \to \Fil^r \OBdlX{X_{(1)}^\partial}^\partial);
    \]
    for $r \in \bZ$; and
    \[
        \OClX{X}^\Sc := \gr^0\bigl(\OBdlX{X}^\Sc\big) \cong \ker(\OClX{X_{(0)}^\partial}^\partial \to \OClX{X_{(1)}^\partial}^\partial).
    \]
    \Pth{See Lemma \ref{lem-AAinf-bd-mor}.  The isomorphism above is justified by Corollary \ref{cor-OBdRp-cplx-bd-strict}.}
\end{defn}

\begin{cor}\label{cor-OBdRp-cplx-bd-cplx}
    The morphisms in Lemma \ref{lem-AAinf-bd-mor} induce an exact complex
    \begin{equation}\label{eq-cor-OBdRp-cplx-bd-cplx-sh-OBdl}
    \begin{split}
        0 & \to \OBdlX{X}^{\Sc, +} \to \OBdlX{X_{(0)}^\partial}^{+, \partial} \\
        & \to \OBdlX{X_{(1)}^\partial}^{+, \partial} \to \cdots \to \OBdlX{X_{(a)}^\partial}^{+, \partial} \to \cdots
    \end{split}
    \end{equation}
    strictly compatible with the filtrations.  Moreover, by forming the tensor product of \Refeq{\ref{eq-cor-OBdRp-cplx-bd-cplx-sh-OBdl}} with the finite locally free $\cO_X$-module $\Omega^{\log, \bullet}_X$, we obtain an exact complex of log de Rham complexes \Pth{\Refcf{} \cite[\aCor \logRHcorlogdRcplx]{Diao/Lan/Liu/Zhu:lrhrv}}
    \begin{equation}\label{eq-cor-OBdRp-cplx-bd-cplx}
    \begin{split}
        0 & \to \OBdlX{X}^{\Sc, +} \otimes_{\cO_X} \Omega^{\log, \bullet}_X \to \OBdlX{X_{(0)}^\partial}^{+, \partial} \otimes_{\cO_X} \Omega^{\log, \bullet}_X \\
        & \to \OBdlX{X_{(1)}^\partial}^{+, \partial} \otimes_{\cO_X} \Omega^{\log, \bullet}_X \to \cdots \to \OBdlX{X_{(a)}^\partial}^{+, \partial} \otimes_{\cO_X} \Omega^{\log, \bullet}_X \to \cdots
    \end{split}
    \end{equation}
    strictly compatible with the filtrations.  The above statements hold with $\OBdlX{X_{(a)}^\partial}^{+, \partial}$ replaced with $\OBdlX{X_{(a)}^\partial}^\partial$, for all $a \geq 0$.  Consequently, we have an exact complex
    \begin{equation}\label{eq-cor-OBdRp-cplx-bd-cplx-sh-OCl}
        0 \to \OClX{X}^\Sc \to \OClX{X_{(0)}^\partial}^\partial \to \OClX{X_{(1)}^\partial}^\partial \to \cdots \to \OClX{X_{(a)}^\partial}^\partial \to \cdots.
    \end{equation}
\end{cor}
\begin{proof}
    By Lemma \ref{lem-OBdRp-bd-loc} and \cite[\aProp \logRHpropOBdlploc]{Diao/Lan/Liu/Zhu:lrhrv}, over each log affinoid perfectoid object $\widetilde{X}$ as in Lemma \ref{lem-OBdRp-bd-loc}, we have compatible isomorphisms as in \Refeq{\ref{eq-lem-OBdRp-bd-loc-OBdR-bd}} which $\BBdRX{X}|_{\widetilde{X}}[[y_1, \ldots, y_n]]$-equivariantly identify the pullback of \Refeq{\ref{eq-cor-OBdRp-cplx-bd-cplx-sh-OBdl}} with the complex
    \begin{equation}\label{eq-cor-OBdRp-cplx-bd-cplx-sh-OBdl-pow}
    \begin{split}
        0 & \to \OBdlX{X}^{\Sc, +}|_{\widetilde{X}} \to \BBdRX{X_{(0)}^\partial}^{+, \partial}|_{\widetilde{X}}[[y_1, \ldots, y_n]] \\
        & \to \BBdRX{X_{(1)}^\partial}^{+, \partial}|_{\widetilde{X}}[[y_1, \ldots, y_n]] \to \ldots \to \BBdRX{X_{(a)}^\partial}^{+, \partial}|_{\widetilde{X}}[[y_1, \ldots, y_n]] \to \ldots,
    \end{split}
    \end{equation}
    where the filtration on each $\BBdRX{X_{(a)}^\partial}^{+, \partial}|_{\widetilde{X}}[[y_1, \ldots, y_n]]$ is given by powers of the ideal generated by $(\xi, y_1, \ldots, y_n)$.  Consequently, by Definition \ref{def-OBdRp-cpt}, the complex \Refeq{\ref{eq-cor-OBdRp-cplx-bd-cplx-sh-OBdl-pow}} is strictly compatible with filtrations, and is exact because the complex \Refeq{\ref{eq-lem-BBdRp-cpt-cplx}} \Pth{with $\widehat{\bL} = \widehat{\bZ}_p$} is.  This verifies the assertions for $\OBdlX{X_{(a)}^\partial}^{+, \partial}$.  By similarly using \cite[\aCor \logRHcorOBdlplocgr]{Diao/Lan/Liu/Zhu:lrhrv}, the assertions for $\OBdlX{X_{(a)}^\partial}^\partial$ and $\OClX{X}^\Sc$ also follow.
\end{proof}

We have the following variant of the \emph{Poincar\'e lemma}:
\begin{prop}\label{prop-Poin-lem}
    We have the following convenient facts over $X_{K, \proket}$:
    \begin{enumerate}
        \item\label{prop-Poin-lem-1}  The exact complex in \cite[\aCor \logRHcorlogdRcplx(\logRHcorlogdRcplxone)]{Diao/Lan/Liu/Zhu:lrhrv} induces an exact complex \begin{equation}\label{eq-prop-Poin-lem-1}
            \begin{split}
                0 & \to \widehat{\bL} \otimes_{\widehat{\bZ}_p} \BBdRX{X}^{\Sc, +} \to \widehat{\bL} \otimes_{\widehat{\bZ}_p} \OBdlX{X}^{\Sc, +} \\
                & \Mapn{\nabla} (\widehat{\bL} \otimes_{\widehat{\bZ}_p} \OBdlX{X}^{\Sc, +}) \otimes_{\cO_X} \Omega^{\log, 1}_X \\
                & \Mapn{\nabla} (\widehat{\bL} \otimes_{\widehat{\bZ}_p} \OBdlX{X}^{\Sc, +}) \otimes_{\cO_X} \Omega^{\log, 2}_X \to \cdots.
            \end{split}
            \end{equation}

        \item\label{prop-Poin-lem-2}  The above statement holds with \cite[\aCor \logRHcorlogdRcplx(\logRHcorlogdRcplxone)]{Diao/Lan/Liu/Zhu:lrhrv} replaced with \cite[\aCor \logRHcorlogdRcplx(\logRHcorlogdRcplxtwo)]{Diao/Lan/Liu/Zhu:lrhrv}, and with $\BBdRX{X}^{\Sc, +}$ and $\OBdlX{X}^{\Sc, +}$ replaced with $\BBdRX{X}^\Sc$ and $\OBdlX{X}^\Sc$, respectively.

        \item\label{prop-Poin-lem-3}  As in \cite[\aCor \logRHcorlogdRcplx(\logRHcorlogdRcplxthree)]{Diao/Lan/Liu/Zhu:lrhrv}, for each $r \in \bZ$, the subcomplex
            \[
            \begin{split}
                0 & \to \Fil^r(\widehat{\bL} \otimes_{\widehat{\bZ}_p} \BBdRX{X}^\Sc) \to \Fil^r(\widehat{\bL} \otimes_{\widehat{\bZ}_p} \OBdlX{X}^\Sc) \\
                & \Mapn{\nabla} \Fil^{r - 1}(\widehat{\bL} \otimes_{\widehat{\bZ}_p} \OBdlX{X}^\Sc) \otimes_{\cO_X} \Omega^{\log, 1}_X \\
                & \Mapn{\nabla} \Fil^{r - 2}(\widehat{\bL} \otimes_{\widehat{\bZ}_p} \OBdlX{X}^\Sc) \otimes_{\cO_X} \Omega^{\log, 2}_X \to \cdots
            \end{split}
            \]
            of the complex for $\BBdRX{X}^\Sc$ and $\OBdlX{X}^\Sc$ is also exact.

        \item\label{prop-Poin-lem-4}  For each $r \in \bZ$, the quotient complex
            \[
            \begin{split}
                0 & \to \gr^r(\widehat{\bL} \otimes_{\widehat{\bZ}_p} \BBdRX{X}^\Sc) \to \gr^r(\widehat{\bL} \otimes_{\widehat{\bZ}_p} \OBdlX{X}^\Sc) \\
                & \Mapn{\nabla} \gr^{r - 1}(\widehat{\bL} \otimes_{\widehat{\bZ}_p} \OBdlX{X}^\Sc) \otimes_{\cO_X} \Omega^{\log, 1}_X \\
                & \Mapn{\nabla} \gr^{r - 2}(\widehat{\bL} \otimes_{\widehat{\bZ}_p} \OBdlX{X}^\Sc) \otimes_{\cO_X} \Omega^{\log, 2}_X \to \cdots
            \end{split}
            \]
            of the previous complex is exact, and can be $\Gal(\AC{k} / k)$-equivariantly identified with the complex
            \[
            \begin{split}
                0 & \to \widehat{\bL} \otimes_{\widehat{\bZ}_p} \widehat{\cO}_X^\Sc(r) \to \bigr(\widehat{\bL} \otimes_{\widehat{\bZ}_p} \OCl^\Sc(r)\bigr) \\
                & \Mapn{\nabla} \bigl(\widehat{\bL} \otimes_{\widehat{\bZ}_p} \OClX{X}^\Sc(r - 1)\bigr) \otimes_{\cO_X} \Omega^{\log, 1}_X \\
                & \Mapn{\nabla} \bigl(\widehat{\bL} \otimes_{\widehat{\bZ}_p} \OClX{X}^\Sc(r - 2)\bigr) \otimes_{\cO_X} \Omega^{\log, 2}_X \to \cdots.
            \end{split}
            \]
    \end{enumerate}
\end{prop}
\begin{proof}
    Let $\cR^\bullet$ denote the complex \Refeq{\ref{eq-prop-Poin-lem-1}}, which we would like to show to be exact.  Since $\widehat{\bL}$ is a local system, by forming its tensor product with the exact complex \Refeq{\ref{eq-cor-OBdRp-cplx-bd}} in Corollary \ref{cor-OBdRp-cplx-bd}, we obtain an exact complex
    \[
    \begin{split}
        0 & \to \widehat{\bL} \otimes_{\widehat{\bZ}_p} \BBdRX{X_{(a)}^\partial}^{+, \partial} \to \widehat{\bL} \otimes_{\widehat{\bZ}_p} \OBdlX{X_{(a)}^\partial}^{+, \partial} \\
        & \Mapn{\nabla} \widehat{\bL} \otimes_{\widehat{\bZ}_p} \OBdlX{X_{(a)}^\partial}^{+, \partial} \otimes_{\cO_X} \Omega^{\log, 1}_X \Mapn{\nabla} \widehat{\bL} \otimes_{\widehat{\bZ}_p} \OBdlX{X_{(a)}^\partial}^{+, \partial} \otimes_{\cO_X} \Omega^{\log, 2}_X \to \cdots,
    \end{split}
    \]
    which we denote by $\cR_{(a)}^\bullet$, for each $a \geq 0$; and we obtain a canonical exact complex of complexes
    \begin{equation}\label{prop-Poin-lem-cplx-cplx}
        0 \to \cR^\bullet \to \cR_{(0)}^\bullet \to \cR_{(1)}^\bullet \to \cdots \to \cR_{(a)}^\bullet \to \cdots,
    \end{equation}
    by Lemma \ref{lem-BBdRp-cpt} and Corollary \ref{cor-OBdRp-cplx-bd-cplx}.  Since \Refeq{\ref{prop-Poin-lem-cplx-cplx}} contains only finitely many nonzero terms, we can break it into finitely many short exact sequences of complexes by taking kernels and cokernels, and argue by taking the associated long exact sequences of cohomology and by downward induction that the complex $\cR^\bullet$ is exact when all the other complexes $\cR_{(a)}^\bullet$ are.  This shows that the complex \Refeq{\ref{eq-prop-Poin-lem-1}} in \Refenum{\ref{prop-Poin-lem-1}} is exact, as desired.  The remaining assertions then follow from this, from the strict compatibility with filtrations in Corollary \ref{cor-OBdRp-cplx-bd-cplx}, and from the corresponding assertions in \cite[\aCor \logRHcorlogdRcplx]{Diao/Lan/Liu/Zhu:lrhrv}.
\end{proof}

\subsection{Comparison of cohomology}\label{sec-dR-comp-cpt-proof}

For simplicity, we shall omit the subscripts \Qtn{$X$} from the period sheaves.  As in \cite[\aSec \logRHseclogRH]{Diao/Lan/Liu/Zhu:lrhrv}, let $\mu: X_\proket \to X_\an$ and $\mu': {X_\proket}_{/X_K} \to X_\an$ denote the canonical morphisms of sites.  Recall that $\RHl(\bL) = R\mu'_*(\widehat{\bL} \otimes_{\widehat{\bZ}_p} \OBdl)$, $\Hl(\bL) = \gr^0\bigl(\RHl(\bL)\bigr) \cong R\mu'_*(\widehat{\bL} \otimes_{\widehat{\bZ}_p} \OCl)$, and $\Ddl(\bL) = \mu_*(\widehat{\bL} \otimes_{\widehat{\bZ}_p} \OBdl)$.
\begin{defn}\label{def-unip-twist}
    Let
    \[
        \RHl^\Sc(\bL) := \ker\Bigl( \RHl(\bL) \to \oplus_{j \in I^\Sc} \, \bigl(\RHl(\bL)|_{\cD_j}^0\bigr)\Bigr)
    \]
    and
    \[
        \Ddl^\Sc(\bL) := \ker\Bigl( \Ddl(\bL) \to \oplus_{j \in I^\Sc} \, \bigl(\Ddl(\bL)|_{D_j}^0\bigr)\Bigr),
    \]
    which are equipped with the induced log connections and filtrations, where \Qtn{$|_{\cD_j}$} and \Qtn{$|_{D_j}$} denote pullbacks \Pth{as coherent sheaves} to $\cD_j$ and $D_j$, respectively, and where the superscripts \Qtn{$0$} denote the maximal quotient sheaves on which the residue endomorphisms act nilpotently \Pth{\Refcf{} \cite[(\logRHeqresgeneigen)]{Diao/Lan/Liu/Zhu:lrhrv}}, with induced quotient filtrations.  For simplicity, by pushforward, we shall abusively consider such sheaves as coherent sheaves on the ambient spaces $\cX$ and $X$.  Accordingly, let
    \[
        \Hl^\Sc(\bL) := \gr^0\bigl(\RHl^\Sc(\bL)\bigr),
    \]
    which is equipped with a canonically induced log Higgs field.
\end{defn}

\begin{rk}\label{rem-unip-twist}
    While the eigenvalues of the residues of $\RHl(\bL)$ along \Pth{the irreducible components of} $D$ are all in $[0, 1)$, the eigenvalues of the residues of $\RHl^\Sc(\bL)$ along $D^\Sc$ and $D^\Snc$ are in $(0, 1]$ and $[0, 1)$, respectively; and the analogous statement is true for $\Ddl(\bL)$ and $\Ddl^\Sc(\bL)$.  By definition, we always have the canonical inclusion $\RHl(\bL)(-D^\Sc) \Em \RHl^\Sc(\bL)$ \Pth{\resp $\Hl(\bL)(-D^\Sc) \Em \Hl^\Sc(\bL)$, \resp $\Ddl(\bL)(-D^\Sc) \Em \Ddl^\Sc(\bL)$}, which is an isomorphism when the residues of $\RHl(\bL)$ \Pth{\resp $\RHl(\bL)$, \resp $\Ddl(\bL)$} along irreducible components of $D^\Sc$ are all nilpotent.  \Pth{By \cite[\aThm \logRHthmunipvsnilp]{Diao/Lan/Liu/Zhu:lrhrv}, such a nilpotence holds when $\bL_{\bQ_p}$ has \emph{unipotent} geometric monodromy along all irreducible components of $D^\Sc$.}
\end{rk}

\begin{lem}\label{lem-unip-twist-qis}
    The canonical morphisms of log de Rham complexes
    \begin{equation}\label{eq-lem-unip-twist-qis-RHl}
        \DRl\bigl(\RHl(\bL)(-D^\Sc)\bigr) \to \DRl\bigl(\RHl^\Sc(\bL)\bigr)
    \end{equation}
    and
    \begin{equation}\label{eq-lem-unip-twist-qis-Ddl}
        \DRl\bigl(\Ddl(\bL)(-D^\Sc)\bigr) \to \DRl\bigl(\Ddl^\Sc(\bL)\bigr),
    \end{equation}
    which are strictly compatible with the filtrations by construction, are quasi-isomorphisms.  Hence, the log Higgs complex
    \begin{equation}\label{eq-lem-unip-twist-qis-Hil}
        \Hil\bigl(\Hl(\bL)(-D^\Sc)\bigr) \to \Hil\bigl(\Hl^\Sc(\bL)\bigr)
    \end{equation}
    is also a quasi-isomorphism.
\end{lem}
\begin{proof}
    By definition of $\RHl^\Sc(\bL)$, the residues induce automorphisms of the pullback of $\bigl(\RHl^\Sc(\bL)\bigr) / \bigl(\RHl(\bL)(-D^\Sc)\bigr)$ to $D_j$, for all $j \in I^\Sc$.  Hence, \Refeq{\ref{eq-lem-unip-twist-qis-RHl}} is a quasi-isomorphism, by the same argument as in the proof of \cite[\aLem 2.10]{Esnault/Viehweg:1992-LVT-B}; and so is \Refeq{\ref{eq-lem-unip-twist-qis-Hil}} by taking $\gr^0$.  Similarly, \Refeq{\ref{eq-lem-unip-twist-qis-Ddl}} is also a quasi-isomorphism.
\end{proof}

\begin{prop}\label{prop-RHl-Hl-Ddl-cpt}
    The canonical morphisms $R\mu'_*(\widehat{\bL} \otimes_{\widehat{\bZ}_p} \OBdl^\Sc) \to \RHl(\bL)$, $R\mu'_*(\widehat{\bL} \otimes_{\widehat{\bZ}_p} \OCl^\Sc) \to \Hl(\bL)$, and $\mu_*(\widehat{\bL} \otimes_{\widehat{\bZ}_p} \OBdl^\Sc) \to \Ddl(\bL)$ factor through canonical isomorphisms $R\mu'_*(\widehat{\bL} \otimes_{\widehat{\bZ}_p} \OBdl^\Sc) \Mi \RHl^\Sc(\bL)$, $R\mu'_*(\widehat{\bL} \otimes_{\widehat{\bZ}_p} \OCl^\Sc) \Mi \Hl^\Sc(\bL)$, and $\mu_*(\widehat{\bL} \otimes_{\widehat{\bZ}_p} \OBdl^\Sc) \Mi \Ddl^\Sc(\bL)$, respectively.
\end{prop}
\begin{proof}
    It suffices to establish the assertion for $\Hl^\Sc(\bL)$, after which the assertions for $\RHl^\Sc(\bL)$ and $\Ddl^\Sc(\bL)$ follow.  Since the assertions are \'etale local in nature, we may suppose as in \cite[\aSec \logRHseccoh]{Diao/Lan/Liu/Zhu:lrhrv} that $X = \Spa(R, R^+)$ is an affinoid log adic space over $k$, equipped with a strictly \'etale morphism
    \[
        X \to \bD^n = \Spa(k\Talg{T_1, \ldots, T_n}, k^+\Talg{T_1, \ldots, T_n})
    \]
    \Pth{with $P = \bZ_{\geq 0}^n$ and $Q = 0$ there and} with $D^\Sc \Em X$ given by the preimage of $\{ T_1 \cdots T_r = 0 \} \Em \bD^n$, so that we have a log perfectoid affinoid covering $\widetilde{X} \to X$ as defined there such that $\widetilde{X}_K \to X_K$ is a Galois pro-Kummer \'etale covering with Galois group $\Gamma_\geom \cong (\widehat{\bZ}(1))^n$.  For each $m \geq 1$, let us write $X_{K, m} = \Spa(R_{K, m}, R_{K, m}^+) := X_K \times_{\bD^n_K} \bD^n_{K, m}$, and denote by $(\widehat{R}_{K, \infty}, \widehat{R}^+_{K, \infty})$ the $p$-adic completion of $\varinjlim_m (R_{K, m}, R^+_{K, m})$, so that $\widehat{\cO}(\widetilde{X}_K) = \widehat{R}_{K, \infty}$.  For each subset $J$ of $\{ 1, \ldots, r \}$, let $R_{J, K, m}$ denote the quotient of $R_{K, m}$ by the ideal generated by $\{ T_j \}_{j \in J}$, and let $R^+_{J, K, m}$ denote the integral closure in $R_{J, K, m}$ of the image of $R^+_{K, m}$.  Note that the nilpotent elements in $R^+_{J, K, m}$ are necessarily $p$-divisible.  Therefore, if we denote by $(\widehat{R}_{J, K, \infty}, \widehat{R}^+_{J, K, \infty})$ the $p$-adic completion of $\varinjlim_m (R_{J, K, m}, R^+_{J, K, m})$, then we have a canonical isomorphism $\widehat{R}_{K, \infty} / (T_j^s)^\wedge_{j \in J, \, s \in \bQ_{> 0}} \Mi \widehat{R}_{J, K, \infty}$, and we have $\widehat{\cO}_{X_J^\partial}^\partial(\widetilde{X}_K) = \widehat{R}_{J, K, \infty}$, as in Lemma \ref{lem-O-flat+-cl}, where $X_J^\partial \subset D^\Sc$ is defined by $\{ T_j = 0 \}_{j \in J}$ \Pth{with its log structure pulled back from $X$}.  When $m = 1$, we shall drop the subscripts \Qtn{$m$} in the above notation.

    Let $\cL := \widehat{\bL} \otimes_{\widehat{\bZ}_p} \widehat{\cO}$ and $\cL^\Sc := \widehat{\bL} \otimes_{\widehat{\bZ}_p} \widehat{\cO}^\Sc$, so that $(\widehat{\bL} \otimes_{\widehat{\bZ}_p} \OCl)|_{\widetilde{X}_K} \cong \cL|_{\widetilde{X}_K}[W_1, \ldots, W_n]$ and $(\widehat{\bL} \otimes_{\widehat{\bZ}_p} \OCl^\Sc)|_{\widetilde{X}_K} \cong \cL^\Sc|_{\widetilde{X}_K}[W_1, \ldots, W_n]$, by Lemma \ref{lem-OBdRp-bd-loc} and \cite[\aCor \logRHcorOBdlplocgr]{Diao/Lan/Liu/Zhu:lrhrv}.  Let $L_\infty := \cL(\widetilde{X}_K)$ and $L^\Sc_\infty := \cL^\Sc(\widetilde{X}_K)$.  For each $J \subset \{ 1, \ldots, r \}$, let $\cL_J := \widehat{\bL} \otimes_{\widehat{\bZ}_p} \widehat{\cO}_{X_J^\partial}^\partial$ and $L_{J, \infty} := \cL_J(\widetilde{X}_K)$.  Then $L_\infty$ is a finite projective $\widehat{R}_{K, \infty}$-module, and $L_{J, \infty} \cong L_\infty \otimes_{\widehat{R}_K} \widehat{R}_{J, K}$, for all $j$.  By evaluating the exact complexes \Refeq{\ref{eq-lem-BBdRp-cpt-cplx-gr}} and \Refeq{\ref{eq-cor-OBdRp-cplx-bd-cplx-sh-OCl}} on $\widetilde{X}$, and by \cite[\aThm \logadicthmalmostvanhat]{Diao/Lan/Liu/Zhu:lasfr}, we obtain an exact complex
    \begin{equation}\label{eq-prop-RHl-Hl-Ddl-cpt-infty}
    \begin{split}
        0 & \to L^\Sc_\infty[W_1, \ldots, W_n] \to L_\infty[W_1, \ldots, W_n] \to \oplus_{|J| = 1} \, L_{J, \infty}[W_1, \ldots, W_n] \\
        & \to \oplus_{|J| = 2} \, L_{J, \infty}[W_1, \ldots, W_n] \to \cdots \to \oplus_{|J| = r} \, L_{J, \infty}[W_1, \ldots, W_n] \to 0
    \end{split}
    \end{equation}
    respecting the variables $W_1, \ldots, W_n$.  By Corollary \ref{cor-O-flat+-cl}, Lemma \ref{lem-OBdRp-bd-loc}, and \cite[\aProp \logRHpropLOCl{} and \aLem \logRHlemLOClcoh]{Diao/Lan/Liu/Zhu:lrhrv},
    \[
        H^i({X_\proket}_{/X_K}, \widehat{\bL} \otimes_{\widehat{\bZ}_p} \OClX{X_J^\partial}^\partial) \cong H^i(\Gamma_\geom, L_{J, \infty}[W_1, \ldots, W_n])
    \]
    is zero, when $i > 0$; and is canonically isomorphic to a finite projective $R_K / (T_j)_{j \in J}$-module $L(X_{J, K}^\partial)$, when $i = 0$, whose formation is compatible with pullbacks under rational localizations and finite \'etale morphisms, by \cite[\aLem \logRHlemLdescent]{Diao/Lan/Liu/Zhu:lrhrv}.  Concretely, in the notation of \cite[\aSec \logRHseccoh]{Diao/Lan/Liu/Zhu:lrhrv}, there exists some model $L_{m_0}(X_K)$ of $L_\infty$ over $R_{K, m_0}$, for some $m_0 \geq 1$, such that $L_{m_0}(X_{J, K}^\partial) := L_{m_0}(X_K) / (T_j^{\frac{1}{m_0}})_{j \in J}$ is a good model of $L_{J, \infty}$, for each $J$, and the $R_K / (T_j)_{j \in J}$-submodule $L(X_{J, K}^\partial)$ is the maximal $K$-subspace of $L_{m_0}(X_{J, K}^\partial)$ on which $\Gamma_\geom$ acts unipotently.  Let
    \begin{equation}\label{eq-def-L-X-K-cpt}
        L(X_K)^\Sc := \ker\bigl(L(X_K) \to \oplus_{|J| = 1} \, L(X_{J, K}^\partial)\bigr).
    \end{equation}

    Since each $L_{m_0}(X_K)$ is finite projective and hence flat over $R_{K, m_0}$, by usual arguments \Pth{\Refcf{} the proof of \cite[\aLem 2.3]{Harris/Lan/Taylor/Thorne:2016-rccsv}}, we have an exact complex
    \[
    \begin{split}
        & 0 \to (T_1^{\frac{1}{m_0}} \cdots T_r^{\frac{1}{m_0}}) \, L_{m_0}(X_K) \to L_{m_0}(X_K) \to \oplus_{|J| = 1} \, \bigl(L_{m_0}(X_K) / (T_j^{\frac{1}{m_0}})_{j \in J}\bigr) \\
        & \to \oplus_{|J| = 2} \, \bigl(L_{m_0}(X_K) / (T_j^{\frac{1}{m_0}})_{j \in J}\bigr) \to \cdots \to \oplus_{|J| = r} \, \bigl(L_{m_0}(X_K) / (T_j^{\frac{1}{m_0}})_{j \in J}\bigr) \to 0,
    \end{split}
    \]
    where $J$ in the above direct sums runs over subsets of $\{ 1, \ldots, r \}$.  By taking the maximal $K$-subspaces on which $\Gamma_\geom$ acts unipotently \Pth{\Refcf{} \cite[\aRem \logRHremcompatLmzero]{Diao/Lan/Liu/Zhu:lrhrv}}, we obtain an exact complex
    \begin{equation}\label{eq-prop-RHl-Hl-Ddl-cpt}
    \begin{split}
        & 0 \to L(X_K)^\Sc \to L(X_K) \to \oplus_{|J| = 1} \, L(X_{J, K}^\partial) \\
        & \to \oplus_{|J| = 2} \, L(X_{J, K}^\partial) \to \cdots \to \oplus_{|J| = r} \, L(X_{J, K}^\partial) \to 0
    \end{split}
    \end{equation}

    Now, by the exactness of \Refeq{\ref{eq-prop-RHl-Hl-Ddl-cpt-infty}}, we have a spectral sequence
    \[
        E_1^{a, b} := H^b(\Gamma_\geom, \oplus_{|J| = a} \, L_{J, \infty}[W_1, \ldots, W_n]) \Rightarrow H^{a + b}(\Gamma_\geom, L^\Sc_\infty[W_1, \ldots, W_n]).
    \]
    By the above discussions, the $E_1$ page is concentrated on the terms $E_1^{a, 0}$.  Hence, the spectral sequence degenerates on the $E_2$ page, and by the exactness of \Refeq{\ref{eq-prop-RHl-Hl-Ddl-cpt}},
    \[
        H^i({X_\proket}_{/X_K}, \widehat{\bL} \otimes_{\widehat{\bZ}_p} \OCl^\Sc) \cong H^i(\Gamma_\geom, L^\Sc_\infty[W_1, \ldots, W_n])
    \]
    is zero, when $i > 0$; and is canonically isomorphic to $L(X_K)^\Sc$, when $i = 0$, whose formation is compatible with pullbacks under rational localizations or finite \'etale morphisms.  Thus, by comparing Definition \ref{def-unip-twist} and \Refeq{\ref{eq-def-L-X-K-cpt}} using \cite[\aRem \logRHremcompatLmzerores]{Diao/Lan/Liu/Zhu:lrhrv}, we obtain $R\mu'_*(\widehat{\bL} \otimes_{\widehat{\bZ}_p} \OCl^\Sc) \cong \Hl^\Sc(\bL)$, as desired.
\end{proof}

\begin{lem}\label{lem-L-!-DdR-to-RH}
    The canonical morphism
    \begin{equation}\label{eq-lem-L-!-DdR-to-RH}
        \Ddl^\Sc(\bL) \ho_k \BdR \to \RHl^\Sc(\bL)
    \end{equation}
    induced by \cite[(\logRHeqlemDdltoRHlfilstrict)]{Diao/Lan/Liu/Zhu:lrhrv} is injective and strictly compatible with the filtrations on both sides.  That is, the induced morphism
    \begin{equation}\label{eq-lem-L-!-DdR-to-RH-gr}
        \gr^r\bigl(\Ddl^\Sc(\bL) \ho_k \BdR\bigr) \to \gr^r\bigl(\RHl^\Sc(\bL)\bigr)
    \end{equation}
    is injective, for each $r$.  If $\bL|_U$ is a \emph{de Rham} $\bZ_p$-local system on $U_\et$, then both \Refeq{\ref{eq-lem-L-!-DdR-to-RH}} and \Refeq{\ref{eq-lem-L-!-DdR-to-RH-gr}} are isomorphisms, and $\gr \Ddl^\Sc(\bL)$ is a vector bundle of rank $\rank_{\bQ_p}(\bL)$.
\end{lem}
\begin{proof}
    These follow from \cite[\aLem \logRHlemDdltoRHlfilstrict, and \aCors \logRHcorDdltoRHlfilstrict{} and \logRHcorDdlgrvecbdl]{Diao/Lan/Liu/Zhu:lrhrv}, and from Proposition \ref{prop-RHl-Hl-Ddl-cpt} and its proof.
\end{proof}

\begin{lem}\label{lem-L-!-coh-comp-arith}
    Suppose that $\bL|_U$ is a \emph{de Rham} $\bZ_p$-local system on $U_\et$.  For each $i \geq 0$, we have a canonical $\Gal(\AC{k} / k)$-equivariant isomorphism
    \begin{equation}\label{eq-lem-L-!-coh-comp-arith-dR}
        H_{\dR, \Sc}^i\bigl(\cU, \RH(\bL)\bigr) \cong H_{\dR, \Sc}^i\bigl(U_\an, \DdR(\bL)\bigr) \otimes_k \BdR,
    \end{equation}
    which induces \Pth{by taking $\gr^0$} a canonical $\Gal(\AC{k} / k)$-equivariant isomorphism
    \begin{equation}\label{eq-lem-L-!-coh-comp-arith-Hdg}
        H_{\Hi, \Sc}^i\bigl(U_{K, \an}, \Hc(\bL)\bigr) \cong \oplus_{a + b = i} \Bigl(H_{\Hdg, \Sc}^{a, b}\bigl(U_\an, \DdR(\bL)\bigr) \otimes_k K(-a)\Bigr).
    \end{equation}
\end{lem}
\begin{proof}
    These follow from Proposition \ref{prop-RHl-Hl-Ddl-cpt}, Lemmas \ref{lem-unip-twist-qis} and \ref{lem-L-!-DdR-to-RH}, and the same arguments as in the proofs of \cite[\aLems \logRHlemlogRHarithcompdR{} and \logRHlemlogRHarithcompHT]{Diao/Lan/Liu/Zhu:lrhrv}.
\end{proof}

We are ready to complete the following:
\begin{proof}[Proof of Theorem \ref{thm-L-!-coh-comp}]
    By applying $R\mu'_*$ to the exact sequences in Lemma \ref{prop-Poin-lem}, and by the projection formula, Lemma \ref{lem-unip-twist-qis}, and Proposition \ref{prop-RHl-Hl-Ddl-cpt}, we can replace the targets of the isomorphisms in Proposition \ref{prop-L-!-coh-comp-proket} with $H_{\dR, \Sc}^i\bigl(\cU, \RH(\bL)\bigr)$ and $H_{\Hi, \Sc}^i\bigl(U_{K, \an}, \Hc(\bL)\bigr)$, respectively, and obtain the canonical isomorphisms \Refeq{\ref{eq-thm-L-!-coh-comp-dR-RH}} and \Refeq{\ref{eq-thm-L-!-coh-comp-Hi}}.  Consequently, by Lemma \ref{lem-L-!-coh-comp-arith}, we also obtain the canonical isomorphisms \Refeq{\ref{eq-thm-L-!-coh-comp-dR}} and \Refeq{\ref{eq-thm-L-!-coh-comp-Hdg}}.  Finally, these isomorphisms imply that
    \[
        \dim_k\bigl(H_{\dR, \Sc}^i\bigl(U_\an, \DdR(\bL)\bigr)\bigr) = \sum_{a + b = i} \dim_k\bigl(H_{\Hdg, \Sc}^{a, b}\bigl(U_\an, \DdR(\bL)\bigr)\bigr),
    \]
    and hence the spectral sequence \Refeq{\ref{eq-thm-L-!-coh-comp-spec-seq}} degenerates on the $E_1$ page, as desired.
\end{proof}

In the remainder of this subsection, let us provide some criteria for cohomology with different partial compact support conditions to be isomorphic to each other.

\begin{lem}\label{lem-dR-qis-geom}
    Let $I^+_\geom$ denote the subset of $I$ consisting of $j \in I$ such that the eigenvalues of the residue of $\RHl(\bL)$ along $D_j$ are all in $\bQ \cap (0, 1)$ \Pth{\ie, nonzero}.  Let $E = \sum_{j \in I} \, c_j D_j$, where $c_j \in \bZ$, be a divisor satisfying the following condition:
    \begin{enumerate}
        \item If $j \in I^+_\geom$, then there is no condition on $c_j$.

        \item If $j \in I^\Sc - I^+_\geom$, then $c_j \leq 0$.

        \item If $j \in I^\Snc - I^+_\geom$, then $c_j \geq 0$.
    \end{enumerate}
    Let us write $E = E^+ - E^-$, where $E^+ := \sum_{j \in I^+_\geom, \, c_j \geq 0} \, c_j D_j + \sum_{j \in I^\Snc - I^+_\geom} \, c_j D_j$ and $E^- := - \sum_{j \in I^+_\geom, \, c_j < 0} \, c_j D_j - \sum_{j \in I^\Sc - I^+_\geom} c_j D_j$ are both effective divisors.  Then the canonical morphisms
    \begin{equation}\label{eq-lem-dR-qis-geom-plus}
        \DRl\bigl(\bigl(\RHl(\bL)\bigr)(-D^\Sc)\bigr) \to \DRl\bigl(\bigl(\RHl(\bL)\bigr)(-D^\Sc + E^+)\bigr)
    \end{equation}
    and
    \begin{equation}\label{eq-lem-dR-qis-geom-minus}
        \DRl\bigl(\bigl(\RHl(\bL)\bigr)(-D^\Sc + E)\bigr) \to \DRl\bigl(\bigl(\RHl(\bL)\bigr)(-D^\Sc + E^+)\bigr)
    \end{equation}
    are quasi-isomorphisms, which induce a canonical isomorphism
    \[
        H_{\dR, \Sc}^i\bigl(\cU, \RH(\bL)\bigr) \cong H^i\bigl(\cX, \DRl\bigl((\RHl(\bL))(-D^\Sc + E)\bigr)\bigr).
    \]
\end{lem}
\begin{proof}
    By the same argument as in the proof of \cite[\aLem 2.7]{Esnault/Viehweg:1992-LVT-B}, for any $c \in \bZ$, the eigenvalues of the residue of $\bigl(\RHl(\bL)\bigr)(c D_j)$ along $D_j$ are the corresponding eigenvalues of $\RHl(\bL)$ minus $c$.  Since the eigenvalues of the residues of $\RHl(\bL)$ are all in $\bQ \cap [0, 1)$, the canonical morphism \Refeq{\ref{eq-lem-dR-qis-geom-plus}} \Pth{\resp{} \Refeq{\ref{eq-lem-dR-qis-geom-minus}}} is a quasi-isomorphism, by the same argument as in the proof of \cite[Properties 2.9 a)]{Esnault/Viehweg:1992-LVT-B} \Pth{\resp \cite[Properties 2.9 b)]{Esnault/Viehweg:1992-LVT-B}}, because none of the eigenvalues of the residues of $\bigl(\RHl(\bL)\bigr)(c_j D_j)$ are in $\bZ_{\geq 1}$ \Pth{\resp $\bZ_{\leq 0}$}, by the choice of $E^+$ \Pth{\resp $E^-$}.
\end{proof}

\begin{cor}\label{cor-dR-qis-geom}
    Let $I^+_\geom$ be as in Lemma \ref{lem-dR-qis-geom}.  Suppose
    \begin{equation}\label{eq-cor-dR-qis-geom-cond}
        I^\Sc - I^+_\geom \subset I^\Scalt \subset I^\Sc \cup I^+_\geom \subset I.
    \end{equation}
    Then, for each $i \geq 0$, we have a canonical $\Gal(\AC{k} / k)$-equivariant isomorphism
    \begin{equation}\label{eq-cor-dR-qis-geom}
        H_{\dR, \Sc}^i\bigl(\cU, \RH(\bL)\bigr) \cong H_{\dR, \Scalt}^i\bigl(\cU, \RH(\bL)\bigr),
    \end{equation}
    which induces \Pth{by taking $\gr^0$} a canonical $\Gal(\AC{k} / k)$-equivariant isomorphism
    \begin{equation}\label{eq-cor-dR-qis-geom-Hdg}
        H_{\Hi, \Sc}^i\bigl(U_{K, \an}, \Hc(\bL)\bigr) \cong H_{\Hi, \Scalt}^i\bigl(U_{K, \an}, \Hc(\bL)\bigr).
    \end{equation}
    Since $K$ is a field extension of $\bQ_p$, by Theorem \ref{thm-L-!-coh-comp}, for $\bL_{\bQ_p} := \bL \otimes_{\bZ_p} \bQ_p$, we also obtain a canonical $\Gal(\AC{k} / k)$-equivariant isomorphism
    \begin{equation}\label{eq-cor-dR-qis-geom-et}
        H_{\et, \Sc}^i(U_K, \bL_{\bQ_p}) \cong H_{\et, \Scalt}^i(U_K, \bL_{\bQ_p}).
    \end{equation}
\end{cor}
\begin{proof}
    Since $I^\Sc - I^+_\geom = I^\Scalt - I^+_\geom$ and $I^\Sc \cup I^+_\geom = I^\Scalt \cup I^+_\geom$, we may assume that $I^\Scalt = I^\Sc - I^+_\geom \subset I^\Sc$, in which case there are compatible canonical morphisms from the cohomology with compact support condition defined by $\Sc$ to that defined by $\Scalt$, and apply Lemma \ref{lem-dR-qis-geom} and Theorem \ref{thm-L-!-coh-comp}.
\end{proof}

\begin{lem}\label{lem-dR-qis-arith}
    Let $I^+_\arith$ denote the subset of $I$ consisting of $j \in I$ such that the eigenvalues of the residue of $\Ddl(\bL)$ along $D_j$ are all in $\bQ \cap (0, 1)$ \Pth{\ie, nonzero}.  Let $E = \sum_{j \in I} \, c_j D_j$, where $c_j \in \bZ$, be a divisor satisfying the following conditions:
    \begin{enumerate}
        \item If $j \in I^+_\arith$, then there is no condition on $c_j$.

        \item If $j \in I^\Sc - I^+_\arith$, then $c_j \leq 0$.

        \item If $j \in I^\Snc - I^+_\arith$, then $c_j \geq 0$.
    \end{enumerate}
    Let us write $E = E^+ - E^-$, where $E^+ := \sum_{j \in I^+_\arith, \, c_j \geq 0} \, c_j D_j + \sum_{j \in I^\Snc - I^+_\arith} \, c_j D_j$ and $E^- := - \sum_{j \in I^+_\arith, \, c_j < 0} \, c_j D_j - \sum_{j \in I^\Sc - I^+_\arith} c_j D_j$ are both effective divisors.  Then the canonical morphisms
    \begin{equation}\label{eq-lem-dR-qis-arith-plus}
        \DRl\bigl(\bigl(\Ddl(\bL)\bigr)(-D^\Sc)\bigr) \to \DRl\bigl(\bigl(\Ddl(\bL)\bigr)(-D^\Sc + E^+)\bigr)
    \end{equation}
    and
    \begin{equation}\label{eq-lem-dR-qis-arith-minus}
        \DRl\bigl(\bigl(\Ddl(\bL)\bigr)(-D^\Sc + E)\bigr) \to \DRl\bigl(\bigl(\Ddl(\bL)\bigr)(-D^\Sc + E^+)\bigr)
    \end{equation}
    are quasi-isomorphisms, which induce a canonical isomorphism
    \[
        H_{\dR, \Sc}^i\bigl(U_\an, \DdR(\bL)\bigr) \cong H^i\bigl(X_\an, \DRl\bigl((\Ddl(\bL))(-D^\Sc + E)\bigr)\bigr).
    \]
\end{lem}
\begin{proof}
    As in the proof of Lemma \ref{lem-dR-qis-geom}, these follow from the same arguments as in the proofs of \cite[\aLem 2.7 and Properties 2.9]{Esnault/Viehweg:1992-LVT-B}.
\end{proof}

\begin{cor}\label{cor-dR-qis-arith}
    Let $I^+_\arith$ be as in Lemma \ref{lem-dR-qis-arith}.  Suppose
    \begin{equation}\label{eq-cor-dR-qis-arith-cond}
        I^\Sc - I^+_\arith \subset I^\Scalt \subset I^\Sc \cup I^+_\arith \subset I.
    \end{equation}
    Then, for each $i \geq 0$, we have a canonical isomorphism
    \begin{equation}\label{eq-cor-dR-qis-arith}
        H_{\dR, \Sc}^i\bigl(U_\an, \DdR(\bL)\bigr) \cong H_{\dR, \Scalt}^i\bigl(U_\an, \DdR(\bL)\bigr),
    \end{equation}
    which induces, for each $a \in \bZ$, \Pth{by taking $\gr^a$} a canonical isomorphism
    \begin{equation}\label{eq-cor-dR-qis-arith-Hdg}
        H_{\Hdg, \Sc}^{a, i - a}\bigl(U_\an, \DdR(\bL)\bigr) \cong H_{\Hdg, \Scalt}^{a, i - a}\bigl(U_\an, \DdR(\bL)\bigr).
    \end{equation}
\end{cor}
\begin{proof}
    Since $I^\Sc - I^+_\arith = I^\Scalt - I^+_\arith$ and $I^\Sc \cup I^+_\arith = I^\Scalt \cup I^+_\arith$, we may assume that $I^\Scalt = I^\Sc - I^+_\arith$, and apply Lemma \ref{lem-dR-qis-arith} and Theorem \ref{thm-L-!-coh-comp}.
\end{proof}

\section{Trace morphisms and Poincar\'e duality}\label{sec-trace}

In this section, we shall retain the setting of Section \ref{sec-dR-comp-cpt-main} \Pth{except in the review in Section \ref{sec-trace-coh}}, but assume moreover that $X$ is \emph{of pure dimension $d$}.  The goal of this section is to construct the trace morphisms for de Rham and \'etale cohomology with compact support in degree $2d$, and show that they are compatible with trace morphisms of lower dimensions via Gysin morphisms, and with each other under the de Rham comparison isomorphism in Theorem \ref{thm-L-!-coh-comp}.

Throughout the section, we shall denote by $\bL$ a $\bZ_p$-local system on $X_\ket$.  Recall that, by \cite[\aCor \logadiccorpuritylisse]{Diao/Lan/Liu/Zhu:lasfr}, any $\bZ_p$-local system on $U_\et$ uniquely extends over $X_\ket$ by pushforward.  Thus, our results for $\bL$ on $X_\ket$ are applicable to all $\bZ_p$-local systems on $U_\et$, despite the notation.

\subsection{Serre duality for coherent cohomology}\label{sec-trace-coh}

In this subsection, we review the trace morphism and Serre duality for the coherent cohomology of proper smooth rigid analytic varieties, and record some of their basic properties.

Let $Y$ be a rigid analytic variety \Pth{regarded as an adic space} over $k$.  For any closed subset $Z \subset Y$, let $H^i_Z(Y_\an, \,\cdot\,)$ denote the usual sheaf cohomology with support in $Z$; \ie, the $i$-th derived functor of
\[
    \cF \mapsto \Gamma_Z(Y_\an, \cF) := \ker\bigl(\Gamma(Y_\an, \cF) \to \Gamma((Y - Z)_\an, \cF)\bigr).
\]
Then there is a canonical morphism $H^i_Z(Y_\an, \,\cdot\,) \to H^i(Y_\an, \,\cdot\,)$, for each $i \geq 0$.  Let $y$ be a classical point \Pth{defined by a finite extension of $k$} and $\cF$ a coherent sheaf on $Y$.  Let $(A, \ideal{m}, M)$ be the $\ideal{m}_y$-adic completion of $(\cO_{Y, y}, \ideal{m}_y, \cF_y)$, where $\ideal{m}_y$ is the maximal ideal of $\cO_{Y, y}$.  Then $A$ is a noetherian complete local ring with residue field $k(y) = A / \ideal{m}$ a finite extension of $k$, and $M$ is a finitely generated $A$-module.  Let $H^i_{\ideal{m}}(M)$ denote the usual algebraic local cohomology \Pth{see, \eg, \cite[\aCh 1]{Brodmann/Sharp:1998-LC}}.

\begin{lem}\label{lem-alg-to-an-loc-coh}
    For each $i \geq 0$, there is a canonical morphism
    \begin{equation}\label{eq-lem-alg-to-an-loc-coh}
        H^i_{\ideal{m}}(M) \to H^i_{\{ y \}}(Y_\an, \cF).
    \end{equation}
\end{lem}
\begin{proof}
    We may replace $Y$ with an open affinoid subspace $U = \Spa(R, R^\circ)$ containing $y$, because $H^i_{\{ y \}}(Y_\an, \cF) \cong H^i_{\{ y \}}(U_\an, \cF|_U)$.  Let $N := \Gamma(Y_\an, \cF)$, which is a finite $R$-module because $\cF$ is coherent.  Let $V := \Spec(R)$, which is a noetherian scheme, and let $\pi: (Y, \cO_Y) \to (V, \cO_V)$ denote the morphism of ringed spaces.  Note that $\pi^{-1}\bigl(\pi(y)\bigr) = \{ y \}$, and that $\cO_Y$ is flat over $\pi^{-1}(\cO_V)$.  Moreover, for the coherent $\cO_V$-module $\widetilde{N}$ associated with the above $N$, we have $\cF \cong \cO_Y \otimes_{\pi^{-1}(\cO_V)} \pi^{-1}(\widetilde{N})$.  By \cite[\aThms 1.3.8 and 4.3.2]{Brodmann/Sharp:1998-LC}, the local cohomology for a finitely generated module over a noetherian local ring is torsion, and hence its formation is compatible with the formation of completions \Pth{with respect to powers of the maximal ideal, as usual}.  Hence, by \cite[\aThm 2.8 and its proof]{Hartshorne:1967-LC}, we have $H^i_{\ideal{m}}(M) \cong H_{\{ \pi(y) \}}^i(V, \widetilde{N})$, and the desired morphism \Refeq{\ref{eq-lem-alg-to-an-loc-coh}} is induced by the composition of canonical morphisms $H_{\{ \pi(y) \}}^i(V, \widetilde{N}) \to H_{\{ y \}}^i\bigl(Y_\an, \pi^{-1}(\widetilde N)\bigr) \to H_{\{ y \}}^i(Y_\an, \cF)$.
\end{proof}

By composing \Refeq{\ref{eq-lem-alg-to-an-loc-coh}} with the canonical morphism $H^i_{\{ y \}}(Y_\an, \cF) \to H^i(Y_\an, \cF)$, we obtain a canonical morphism
\begin{equation}\label{eq-alg-loc-coh-to-an-coh}
    H^i_{\ideal{m}}(M) \to H^i(Y, \cF).
\end{equation}

Now let $Y$ be smooth of pure dimension $d$, with $\cF = \Omega_Y^d$ the sheaf of top-degree differentials on $Y$.  Then the $A$-module $M$ can be identified with the top-degree continuous K\"ahler differentials $\Omega_{A / k}^d$ of $A$ \Pth{with its $\ideal{m}$-adic topology} over $k$.
\begin{constr}\label{constr-Poin-res}
    Let us construct a canonical residue morphism
    \begin{equation}\label{eq-constr-Poin-res}
        \res_y: H^d_{\ideal{m}}(\Omega_{A / k}^d) \to k.
    \end{equation}
    By choosing local coordinates of $Y$ near $k(y)$, we have compatible isomorphisms $A \cong k(y)[[T_1, \ldots, T_d]]$ and $\Omega_{A / k}^d \cong k(y)[[T_1, \ldots, T_d]] \, dT_1 \wedge \cdots \wedge dT_d$.  Accordingly, we have \Pth{\Refcf{} \cite[\aSec 4, \aEx 3]{Hartshorne:1967-LC}}
    \[
        H^d_{\ideal{m}}(\Omega_{A / k}^d) \cong \Bigl\{ \sum_\alpha \, a_\alpha T^\alpha \, dT_1 \wedge \cdots \wedge dT_d : \alpha = (\alpha_1, \ldots, \alpha_d) \in \bZ_{< 0}^d, \, a_\alpha \in k(y) \Bigr\}
    \]
    \Pth{where the sum is finite}, and the desired morphism \Refeq{\ref{eq-constr-Poin-res}} is the composition of the \Pth{multiple} residue morphism $\sum_{\alpha \in \bZ_{< 0}^d} \, a_\alpha T^\alpha \, dT_1 \wedge \cdots \wedge dT_d \mapsto a_{(-1, \ldots, -1)}$ with the usual trace morphism $k(y) \to k$, which \Pth{by the chain rule} is independent of the choice of coordinates.  \Pth{When $d = 0$, our convention is that $H^d_{\ideal{m}}(\Omega_{A / k}^d) \cong k(y)$ and that the residue morphism reduces to the identity morphism on $k(y)$.}
\end{constr}

\begin{thm}[Serre duality]\label{thm-Serre-duality}
    Let $Y$ be a proper smooth rigid analytic variety over $k$ of pure dimension $d$.  Then there is a unique morphism
    \begin{equation}\label{eq-trace-coh-gen}
        t_\coh: H^d(Y_\an, \Omega_Y^d) \to k,
    \end{equation}
    whose pre-composition with \Refeq{\ref{eq-alg-loc-coh-to-an-coh}}, for any classical point $y$, gives the residue morphism \Refeq{\ref{eq-constr-Poin-res}}.  We call $t_\coh$ the \emph{trace morphism}.  Moreover, by pre-composition with the cup product pairing, it induces the usual Serre duality for coherent cohomology; \ie, a perfect pairing
    \begin{equation}\label{eq-thm-Serre-duality-coh}
        H^i(Y_\an, \cF^\bullet) \times \Ext_{\cO_Y}^{d - i}(\cF^\bullet, \Omega_Y^d) \to k,
    \end{equation}
    for each bounded complex $\cF^\bullet$ of coherent $\cO_Y$-modules and each $i \in \bZ$.  As a special case, we have a perfect pairing
    \begin{equation}\label{eq-thm-Serre-duality-loc-free}
        H^i(Y_\an, \cF^\bullet) \times H^{d - i}(Y_\an, \cF^{\bullet, \dualsign} \otimes_{\cO_Y} \Omega_Y^d) \to k,
    \end{equation}
    for each bounded complex $\cF^\bullet$ of finite locally free $\cO_Y$-modules and each $i \in \bZ$, where $\cF^{\bullet, \dualsign}$ denotes the complex whose $j$-th term is $\dual{(\cF^{-j})}$, for each $j \in \bZ$.  In particular, if $Y$ is geometrically connected, then \Refeq{\ref{eq-trace-coh-gen}} is an isomorphism.
\end{thm}
\begin{proof}
    When $\cF^\bullet$ is concentrated in degree zero, the isomorphism \Refeq{\ref{eq-thm-Serre-duality-coh}} follows from \cite[\aThms 5.1.1 and 5.1.2, \aDef 4.2.4, \aLem 4.2.9, and the explicit descriptions in \aSecs 1.2, 1.3, and 2.1]{Beyer:1997-sdcsr}.  \Pth{An earlier construction of the Serre duality for proper rigid analytic varieties is in \cite{vanderPut:1992-sdras}, but the more explicit descriptions in \cite{Beyer:1997-sdcsr} allow us to more directly relate the trace morphism there to the residue morphism \Refeq{\ref{eq-constr-Poin-res}} here.}  Therefore, by using \cite[\aCh I, \aProp 7.1 (\emph{Lemma on Way-Out Functors})]{Hartshorne:1966-RD} as usual, we also have the isomorphism \Refeq{\ref{eq-thm-Serre-duality-coh}} for each bounded complex $\cF^\bullet$ of coherent $\cO_Y$-modules, which specializes to the isomorphism \Refeq{\ref{eq-thm-Serre-duality-loc-free}} when $\cF^\bullet$ is a bounded complex of finite locally free $\cO_Y$-modules.
\end{proof}

Let us record the following properties of the trace morphism $t_\coh$ for later use.
\begin{lem}\label{lem-property-trace}
    Assume that $Y$ is proper smooth over $k$, of pure dimension $d$.
    \begin{enumerate}
        \item\label{lem-property-trace-exc}  Let $f: Y' \to Y$ be a morphism of proper smooth rigid analytic varieties which induces an isomorphism $f^{-1}(U) \to U$ for some open dense rigid analytic subvariety $U$ of $Y$, then the canonical morphism $H^d(Y_\an, \Omega_Y^d) \to H^d(Y'_\an, \Omega_{Y'}^d)$ induced by the canonical morphism $f^*(\Omega_Y^d) \to \Omega_{Y'}^d$ is an isomorphism compatible with the trace morphisms of $Y$ and $Y'$ as in \Refeq{\ref{eq-trace-coh-gen}}.

        \item\label{lem-property-trace-Gysin}  Let $\imath: Z \Em Y$ be a smooth divisor of $Y$.  Then the canonical morphism $H^{d - 1}(Z_\an, \Omega_Z^{d - 1}) \to H^d(Y_\an, \Omega_Y^d)$ induced by the adjunction exact sequence
            \begin{equation}\label{eq-lem-property-trace-Gysin}
                0 \to \Omega_Y^d \to \Omega_Y^d(Z) \to \imath_*(\Omega_Z^{d - 1}) \to 0
            \end{equation}
            is compatible with the trace morphisms of $Y$ and $Z$ as in \Refeq{\ref{eq-trace-coh-gen}}.
   \end{enumerate}
\end{lem}
\begin{proof}
    The assertion \Refenum{\ref{lem-property-trace-exc}} follows from Theorem \ref{thm-Serre-duality}, Lemma \ref{lem-alg-to-an-loc-coh}, and Construction \ref{constr-Poin-res}, because we can determine the trace morphisms as in \Refeq{\ref{eq-trace-coh-gen}} for $Y$ and $Y'$ by choosing sufficiently many $y \in f^{-1}(U) \Mi U$, and because the residue morphisms as in \Refeq{\ref{eq-constr-Poin-res}} for $Y$ and $Y'$ at $y$ can be canonically identified with each other.

    As for the assertion \Refenum{\ref{lem-property-trace-Gysin}}, for each point $y \in Z \subset Y$, we may choose local coordinates such that, if $(A, \ideal{m})$ denotes the completion of $(\cO_{Y, y}, \ideal{m}_y)$ as in the beginning of this subsection, and if $(B, \ideal{n})$ denotes the corresponding completion of $(\cO_{Z, y}, \ideal{m}_y \cO_{Z, y})$, then we have compatible isomorphisms $A \cong k(y)[[T_1, \ldots, T_d]]$ and $B \cong k(y)[[T_1, \ldots, T_{d - 1}]]$, with maximal ideals $\ideal{m}$ and $\ideal{n}$ generated by $T_1, \ldots, T_d$ and by $T_1, \ldots, T_{d - 1}$, respectively, together with the canonical short exact sequence $0 \to \Omega_{A / k}^d \to \frac{1}{T_d}\Omega_{A / k}^d \to \Omega_{B / k}^{d - 1} \to 0$ induced by \Refeq{\ref{eq-lem-property-trace-Gysin}}, which is given by
    \[
    \begin{split}
        0 & \to k(y)[[T_1, \ldots, T_d]] \, dT_1 \wedge \cdots \wedge dT_d \to k(y)[[T_1, \ldots, T_d]] \, dT_1 \wedge \cdots \wedge \tfrac{dT_d}{T_d} \\
        & \to k(y)[[T_1, \ldots, T_{d - 1}]] \, dT_1 \wedge \cdots \wedge dT_{d - 1} \to 0
    \end{split}
    \]
    in explicit coordinates.  Then the connecting morphism $H_{\ideal{m}}^{d - 1}(\Omega_{B / k}^{d - 1}) \to H_{\ideal{m}}^d(\Omega_{A / k}^d)$ in the associated long exact sequence is \Pth{by an explicit calculation} given by
    \[
    \begin{split}
        & \Bigl\{ \sum_\alpha a_\alpha T^\alpha \, dT_1 \wedge \cdots \wedge dT_{d - 1} : \alpha = (\alpha_1, \ldots, \alpha_{d - 1}) \in \bZ_{< 0}^{d - 1}, \, a_\alpha \in k(y) \Bigr\} \\
        & \to \Bigl\{ \sum_\alpha a_\alpha T^\alpha \, dT_1 \wedge \cdots \wedge dT_d : \alpha = (\alpha_1, \ldots, \alpha_d) \in \bZ_{< 0}^d, \, a_\alpha \in k(y) \Bigr\} : f \mapsto f \wedge \tfrac{dT_d}{T_d},
    \end{split}
    \]
    which is compatible with the trace morphisms, by the construction based on \Pth{multiple} residue morphisms in Construction \ref{constr-Poin-res}.
\end{proof}

By applying the above results to our setting in the beginning of Section \ref{sec-trace}, we obtain a trace morphism
\begin{equation}\label{eq-trace-coh}
    t_\coh: H^d(X_\an, \Omega_X^d) \to k,
\end{equation}
as in \Refeq{\ref{eq-trace-coh-gen}}, which induces by base change from $k$ to $K$ a trace morphism
\begin{equation}\label{eq-trace-coh-K}
    t_{\coh, K}: H^d(X_{K, \an}, \Omega_{X_K}^d) \to K,
\end{equation}
which in turn induces the Serre duality \Pth{as in Theorem \ref{thm-Serre-duality}} for bounded complexes of coherent $\cO_{X_K}$-modules.

\subsection{Poincar\'e duality for de Rham cohomology}\label{sec-trace-dR}

\begin{thm}\label{thm-Higgs-pairing}
    For each $\bZ_p$-local system $\bL$ on $X_\ket$, the composition of the canonical cup product pairing with \Refeq{\ref{eq-trace-coh-K}} induces a perfect pairing
    \begin{equation}\label{eq-thm-Higgs-pairing}
        H_{\Hi, \Sc}^i\bigl(U_{K, \an}, \Hc(\bL)\bigr) \times H_{\Hi, \Snc}^{2d - i}\bigl(U_{K, \an}, \Hc\bigl(\dual{\bL}(d)\bigr)\bigr) \to K.
    \end{equation}
\end{thm}
\begin{proof}
    By \cite[\aThm 3.8(i)]{Liu/Zhu:2017-rrhpl}, $\RH\bigl(\dual{\bL}(d)\bigr) \cong \dual{\bigl(\RH(\bL)(-d)\bigr)}$ as filtered vector bundles on $\cU$.  By the definition of $\RHl(\,\cdot\,)$, we have a canonical morphism $\bigl(\RHl(\bL)(-d)\bigr) \otimes_{\cO_\cX} \RHl\bigl(\dual{\bL}(d)\bigr) \to \RHl(\bQ_p) \cong \cO_\cX$.  Since $\RHl(\bL)(-d)$ and $\RHl\bigl(\dual{\bL}(d)\bigr)$ are filtered vector bundles on $\cX$ extending $\RH(\bL|_U)(-d)$ and $\RH\bigl(\dual{(\bL|_U)}(d)\bigr)$, respectively, we obtain a canonical injective morphism
    \begin{equation}\label{eq-thm-Higgs-pairing-mor}
        \RHl\bigl(\dual{\bL}(d)\bigr) \to \dual{\bigl(\RHl(\bL)(-d)\bigr)}
    \end{equation}
    of vector bundles on $\cX$, which is compatible with the connections with log poles on both sides, whose cokernel is supported on the boundary $D$.  Moreover, the filtration on $\RHl\bigl(\dual{\bL}(d)\bigr)$ is induced by the one on $\RH\bigl(\dual{\bL}(d)\bigr)$ via the canonical injective morphism $\RHl\bigl(\dual{\bL}(d)\bigr) \to \jmath_*\bigl(\RH\bigl(\dual{\bL}(d)\bigr)\bigr)$, where $\jmath: \cU \to \cX$ denotes the canonical open immersion; and the analogous assertions are true for $\RHl(\bL)(-d)$ and $\dual{\bigl(\RHl(\bL)(-d)\bigr)}$.  Therefore, \Refeq{\ref{eq-thm-Higgs-pairing-mor}} is strictly compatible with the filtrations on both sides, and induces a canonical morphism
    \begin{equation}\label{eq-thm-Higgs-pairing-qis-dR}
        \DRl\bigl(\RHl\bigl(\dual{\bL}(d)\bigr)(-D^\Snc)\bigr) \to \DRl\bigl(\dual{\bigl(\RHl(\bL)(-d)\bigr)}(-D^\Snc)\bigr)
    \end{equation}
    between the associated log de Rham complexes over $\cX$, which is also strictly compatible with the filtrations on both sides.

    By comparing the residues of the two sides of \Refeq{\ref{eq-thm-Higgs-pairing-qis-dR}} using \cite[\aThm \logRHthmlogRHgeom(\logRHthmlogRHgeomres) and \aProp \logRHpropconnisomres]{Diao/Lan/Liu/Zhu:lrhrv}, and by the same argument as in the proof of \cite[\aLem 2.10]{Esnault/Viehweg:1992-LVT-B}, we may factor \Refeq{\ref{eq-thm-Higgs-pairing-qis-dR}} as a composition of a series of inclusions $\cE \to \cE'$ of complexes each of whose cokernel is a two-term complex \Pth{in some degrees $a$ and $a + 1$}
    \[
        0 \to \Omega_{D_j}^a\bigl(\log (D - D_j)|_{D_j}\bigr) \otimes_{\cO_{D_j}} \cF \to \Omega_{D_j}^a\bigl(\log (D - D_j)|_{D_j}\bigr) \otimes_{\cO_{D_j}} \cF \to 0,
    \]
    where $\cF$ is the maximal subsheaf of $\bigl(\dual{\bigl(\RHl(\bL)(-d)\bigr)}(-D^\Snc)\bigr)|_{D_j}$, for some $j$, on which the eigenvalues of residues differ from that of $\bigl(\RHl\bigl(\dual{\bL}(d)\bigr)(-D^\Snc)\bigr)|_{D_j}$ and hence belong to $\bQ \cap (-1, 0)$ \Pth{\resp $\bQ \cap (0, 1)$} when $j \in I^\Sc$ \Pth{\resp $j \in I^\Snc$}; and where the morphism between the two terms is induced by $(-1)^a$ times the residue morphism and hence is an isomorphism.  Thus, \Refeq{\ref{eq-thm-Higgs-pairing-qis-dR}} is a \emph{quasi-isomorphism}, which induces by taking cohomology and taking $\gr^0$ a canonical isomorphism
    \begin{equation}\label{eq-thm-Higgs-pairing-isom}
        H_{\Hi, \Snc}^{2d - i}\bigl(U_{K, \an}, \Hc\bigl(\dual{\bL}(d)\bigr)\bigr) \Mi H_{\Hi, \Snc}^{2d - i}\bigl(U_{K, \an}, \dual{\Hc\bigl(\bL(-d)\bigr)}\bigr).
    \end{equation}

    By using the canonical isomorphisms $\Omega_X^{\log, a} \cong \dual{(\Omega_X^{\log, d - a})} \otimes_{\cO_X} \bigl(\Omega_X^d(D)(d)[-d]\bigr)$, for all $0 \leq a \leq d$, induced by $\Omega_X^d \cong \Omega_X^{\log, d}(-D)$ and the exterior algebra structure of $\Omega_X^{\log, \bullet} = \Ex^\bullet \Omega_X^{\log}$, we have a canonical isomorphism $\Hil\bigl(\Hl(\bL)(-D^\Sc)\bigr) \cong \dual{\bigl(\Hil\bigl(\dual{\bigl(\Hl(\bL)(-d)\bigr)}(-D^\Snc)\bigr)\bigr)} \otimes \bigl(\Omega_{X_K}^d(d)[-d]\bigr)$ over $X_K$.  Hence, we obtain the desired perfect pairing \Refeq{\ref{eq-thm-Higgs-pairing}} by combining \Refeq{\ref{eq-thm-Higgs-pairing-isom}} with the duality for bounded complexes of finite locally free $\cO_{X_K}$-modules \Pth{as in Theorem \ref{thm-Serre-duality}}.
\end{proof}

Since $\Omega_X^d \cong \Omega_X^{\log, d}(-D)$ \Pth{which we already used in the above proof} and therefore
\begin{equation}\label{eq-diff-top-deg-coh}
    H_{\dR, \cpt}^{2d}\bigl(U_\an, \cO_U(d)\bigr) \cong H_{\Hdg, \cpt}^{d, d}\bigl(U_\an, \cO_U(d)\bigr) \cong H^d(X_\an, \Omega_X^d)
\end{equation}
\Pth{by Theorem \ref{thm-L-!-coh-comp}, with $\bL = \bZ_p(d)$ and with $\cO_U(d)$ denoting the same underlying $\cO_U$-module with the trivial Hodge filtration shifted by $d$}, we obtain the following:
\begin{thm}\label{thm-trace-dR-Hdg}
    The trace morphism $t_\coh$ as in \Refeq{\ref{eq-trace-coh}} induces compatible trace morphisms
    \begin{equation}\label{eq-thm-trace-dR}
        t_\dR: H_{\dR, \cpt}^{2d}\bigl(U_\an, \cO_U(d)\bigr) \to k
    \end{equation}
    and
    \begin{equation}\label{eq-thm-trace-Hdg}
        t_\Hdg: H_{\Hdg, \cpt}^{d, d}\bigl(U_\an, \cO_U(d)\bigr) \to k.
    \end{equation}
    When $d = 0$ and $X$ is geometrically connected \Pth{\ie, a $k$-point}, the trace morphisms $t_\dR: H_\dR^0(X_\an, \cO_X) \to k$ and $t_\Hdg: H_\Hdg^{0, 0}(X_\an, \cO_X) \to k$ are just the canonical isomorphisms given by $H_\dR^0(X_\an, \cO_X) \cong H_\Hdg^{0, 0}(X_\an, \cO_X) \cong H^0(X_\an, \cO_X) \cong k$.  For each $\bZ_p$-local system $\bL$ on $X_\ket$ such that $\bL|_U$ is \emph{de Rham}, the composition of the cup product pairing with \Refeq{\ref{eq-thm-trace-dR}} induces a perfect pairing
    \begin{equation}\label{eq-thm-trace-dR-pairing}
        H_{\dR, \Sc}^i\bigl(U_\an, \DdR(\bL)\bigr) \times H_{\dR, \Snc}^{2d - i}\bigl(U_\an, \DdR\bigl(\dual{\bL}(d)\bigr)\bigr) \to k,
    \end{equation}
    which is compatible \Pth{by taking graded pieces} with the perfect pairing
    \begin{equation}\label{eq-thm-trace-Hdg-pairing}
        H_{\Hdg, \Sc}^{a, b}\bigl(U_\an, \DdR(\bL)\bigr) \times H_{\Hdg, \Snc}^{d - a, d - b}\bigl(U_\an, \DdR\bigl(\dual{\bL}(d)\bigr)\bigr) \to k
    \end{equation}
    defined by the composition of the cup product pairings with \Refeq{\ref{eq-thm-trace-Hdg}}.
\end{thm}
\begin{proof}
    The first two assertions are clear.  As for the third one, suppose that $\bL|_U$ is de Rham.  Since $K$ is a field extension of $k$, we obtain the desired perfect pairing \Refeq{\ref{eq-thm-trace-Hdg-pairing}}, by Theorem \ref{thm-Higgs-pairing} and comparison isomorphisms as in \Refeq{\ref{eq-lem-L-!-coh-comp-arith-Hdg}}.  Since the formation of cup products is compatible with the formation of the $E_1$ pages of the Hodge--de Rham spectral sequences for $H_{\dR, \Sc}^i\bigl(U_\an, \DdR(\bL)\bigr)$ and $H_{\dR, \Snc}^{2d - i}\bigl(U_\an, \DdR\bigl(\dual{\bL}(d)\bigr)\bigr)$, and since these spectral sequences degenerate on the $E_1$ pages by Theorem \ref{thm-L-!-coh-comp}, we also obtain the desired perfect pairing \Refeq{\ref{eq-thm-trace-dR-pairing}}.
\end{proof}

\begin{lem}\label{lem-trace-dR}
    The formation of $t_\dR$ in Theorem \ref{thm-trace-dR-Hdg} is compatible with restrictions to open rigid analytic subvarieties of the form $U - X_0 = X - D - X_0$ for some closed rigid analytic subvarieties $X_0$ of $X$.  In particular, it is also compatible with any morphism between proper smooth rigid analytic varieties that is an isomorphism over some open dense rigid analytic subvariety \Pth{\eg, any blowup, as in \cite[\aDef 4.1.1]{Conrad:2006-rarg}, at closed rigid analytic subvarieties}.
\end{lem}
\begin{proof}
    By using resolution of singularities \Pth{as in \cite{Bierstone/Milman:1997-cdbml}}, there exists a proper morphism $\pi: X' \to X$ such that $\bigl(\pi^{-1}(D \cup X_0)\bigr)_\red$ \Pth{with its canonical reduced closed subspace structure} is a simple normal crossings divisor, and such that $\pi$ is an isomorphism over $U' = U - X_0$.  Thus, by \Refeq{\ref{eq-diff-top-deg-coh}}, it suffices to note that, by Lemma \ref{lem-property-trace}\Refenum{\ref{lem-property-trace-exc}}, the canonical morphism $H^d(X_\an, \Omega_X^d) \to H^d(X'_\an, \Omega_{X'}^d)$ is compatible with the trace morphisms for coherent cohomology as in \Refeq{\ref{eq-trace-coh}}.
\end{proof}

\subsection{Excision and Gysin isomorphisms}\label{sec-exc-Gysin}

In order to deduce the Poincar\'e duality for \Pth{rational} \'etale cohomology from the Poincar\'e duality for de Rham and Higgs cohomology, we shall establish in this subsection some compatibilities between the comparison isomorphisms and the excision and Gysin isomorphisms defined by complements of smooth divisors.

Recall that $U = X - D$ with $D = \cup_{j \in I} \, D_j$.  Let us begin with the excision isomorphisms between top-degree cohomology.

\begin{lem}\label{lem-exc-isom}
    Let $Z = D_{j_0}$ for some $j_0 \in I$, that $D' := \cup_{j \in I - \{ j_0 \}} \, D_j$, and that $U' := X - D'$, so that $U = U' - W$ for some smooth closed rigid analytic subvariety $W = U' \cap Z$ of $U'$.  Let $\jmath_U: U \to U'$ and $\imath_W: W \to U'$ denote the canonical open immersions and closed immersions \Pth{of underlying adic spaces, without log structures}, respectively.  Then we have the excision short exact sequence
    \begin{equation}\label{eq-lem-exc-isom-ex-seq}
        0 \to \jmath_{U, \et, !}\bigl(\bZ_p(d)\bigr) \to \bZ_p(d) \to \imath_{W, \et, *}\bigl(\bZ_p(d)\bigr) \to 0
    \end{equation}
    over $U'_{K, \et}$, which induces an isomorphism
    \begin{equation}\label{eq-lem-exc-isom}
        H_{\et, \cpt}^{2d}\bigl(U_K, \bZ_p(d)\bigr) \Mi H_{\et, \cpt}^{2d}\bigl(U'_K, \bZ_p(d)\bigr).
    \end{equation}
    By composing such isomorphisms, we have $H_{\et, \cpt}^{2d}\bigl(U_K, \bZ_p(d)\bigr) \Mi H_\et^{2d}\bigl(X_K, \bZ_p(d)\bigr)$.
\end{lem}
\begin{proof}
    This is because $H_{\et, \cpt}^i\bigl(W_K, \bZ_p(d)\bigr) = 0$ for $i > 2\dim(W) = 2d - 2$ in the long exact sequence associated with \Refeq{\ref{eq-lem-exc-isom-ex-seq}}, by Theorem \ref{thm-L-!-prim-comp} \Pth{which implies the analogous vanishing result for $\bZ_p$-local systems, by standard arguments}.
\end{proof}

\begin{prop}\label{prop-exc-dR-comp}
    With the same setting as in Lemma \ref{lem-exc-isom}, the isomorphism \Refeq{\ref{eq-lem-exc-isom}} extends to a commutative diagram
    \begin{equation}\label{eq-prop-exc-dR-comp}
        \xymatrix{ {H_{\et, \cpt}^{2d}\bigl(U_K, \bZ_p(d)\bigr)} \ar[r]^-\sim \ar@{^(->}[d] & {H_{\et, \cpt}^{2d}\bigl(U'_K, \bZ_p(d)\bigr)} \ar@{^(->}[d] \\
        {H_{\et, \cpt}^{2d}\bigl(U_K, \bZ_p(d)\bigr) \otimes_{\bZ_p} \BdR} \ar[r]^-\sim \ar[d] & {H_{\et, \cpt}^{2d}\bigl(U'_K, \bZ_p(d)\bigr) \otimes_{\bZ_p} \BdR} \ar[d] \\
        {H_{\dR, \cpt}^{2d}\bigl(U_\an, \cO_U(d)\bigr) \otimes_k \BdR} \ar[r]^-\sim & {H_{\dR, \cpt}^{2d}\bigl(U'_\an, \cO_{U'}(d)\bigr) \otimes_k \BdR} \\
        {H_{\dR, \cpt}^{2d}\bigl(U_\an, \cO_U(d)\bigr)} \ar[r]^-\sim \ar@{^(->}[u] & {H_{\dR, \cpt}^{2d}\bigl(U'_\an, \cO_{U'}(d)\bigr)} \ar@{^(->}[u] }
    \end{equation}
    Here the fourth row \Pth{at the bottom} is the excision isomorphism induced by the long exact sequence associated with the excision short exact sequence
    \begin{equation}\label{eq-prop-exc-dR-comp-ex-seq-dR}
    \begin{split}
        0 & \to \Omega_X^\bullet(\log D)(-D) \to \Omega_X^\bullet(\log D')(-D') \\
        & \to \imath_{Z, \an, *}\bigl(\Omega_Z^\bullet\bigl(\log (D'|_Z)\bigr)(-D'|_Z)\bigr) \to 0
    \end{split}
    \end{equation}
    for de Rham complexes over $X_\an$, where $\imath_Z: Z \to X$ denotes the canonical closed immersion of adic spaces, and where $D'|_Z := D' \cap Z = \cup_{j \in I - \{ j_0 \}} \, (D_j \cap Z)$, which is an isomorphism because
    \[
        H_{\dR, \cpt}^i\bigl(W_\an, \cO_W(d)\bigr) \cong H^i\bigl(Z_\an, \Omega_Z^\bullet\bigl(\log (D'|_Z)\bigr)(d)\bigr) = 0
    \]
    for $i > 2\dim(W) = 2\dim(Z) = 2d - 2$.  Moreover, this isomorphism at the bottom row is compatible with the trace morphisms, by Lemma \ref{lem-trace-dR}.
\end{prop}
\begin{proof}
    Let $X'$ denote the log adic space with the same underlying space as $X$, but with the log structure induced by the normal crossings divisor $D'$ as in \cite[\aEx \logRHexlogadicspncd]{Diao/Lan/Liu/Zhu:lrhrv}.  Let $Z$ be equipped with the log structure induced by $D'|_Z$, so that we have a canonical closed immersion of log adic spaces $\imath_Z': Z \to X'$.  For simplicity, we shall still write $\imath_{Z, \an}: Z_\an \to X_\an$ instead of $\imath_{Z, \an}': Z_\an \to X'_\an$.  Let $\varepsilon: X \to X'$ denote the canonical morphism.  Let $\jmath_U: U \to X$, $\jmath_U': U \to X'$, $\jmath_{U'}: U' \to X'$, and $\jmath_W: W \to Z$ denote the canonical open immersions.  Then the short exact sequence \Refeq{\ref{eq-lem-exc-isom-ex-seq}} induces \Pth{and is induced by} a short exact sequence
    \begin{equation}\label{eq-prop-exc-dR-comp-ex-seq-ket}
        0 \to \jmath'_{U, \ket, !}\bigl(\bZ_p(d)\bigr) \to \jmath_{U', \ket, !}\bigl(\bZ_p(d)\bigr) \to \imath_{Z, \ket, *}' \, \jmath_{W, \ket, !}\bigl(\bZ_p(d)\bigr) \to 0
    \end{equation}
    over $X'_{K, \ket}$, which induces \Refeq{\ref{eq-lem-exc-isom}} and the first row of \Refeq{\ref{eq-prop-exc-dR-comp}}; and we have
    \begin{equation}\label{eq-prop-exc-dR-comp-ex-seq-ket-ch}
        \jmath'_{U, \ket, !}\bigl(\bZ_p(d)\bigr) \Mi R\varepsilon_{\ket, *} \, \jmath_{U, \ket, !}\bigl(\bZ_p(d)\bigr) \cong \varepsilon_{\ket, *} \, \jmath_{U, \ket, !}\bigl(\bZ_p(d)\bigr),
    \end{equation}
    by the definitions of these sheaves.

    Let $\varpi \in K^{\flat+}$ be such that $\varpi^\sharp = p$.  By Lemma \ref{lem-from-ket-to-proket} and \cite[\aProp \logadicpropproketvsketadj]{Diao/Lan/Liu/Zhu:lasfr}, the morphism
    \[
        \jmath'_{U, \ket, !}\bigl((\bZ / p^m)(d)\bigr) \to \jmath_{U', \ket, !}\bigl((\bZ / p^m)(d)\bigr)
    \]
    \Pth{induced by \Refeq{\ref{eq-prop-exc-dR-comp-ex-seq-ket}}} is the pushforward of the morphism
    \begin{equation}\label{eq-prop-exc-dR-comp-mor-proket}
        \upsilon_{X'}^{-1} \, \jmath'_{U, \ket, !}\bigl((\bZ / p^m)(d)\bigr) \to \upsilon_{X'}^{-1} \, \jmath_{U', \ket, !}\bigl((\bZ / p^m)(d)\bigr)
    \end{equation}
    over $X'_{K, \proket}$, which admits compatible morphisms to a morphism
    \[
    \begin{split}
        & \Bigl(\upsilon_{X'}^{-1} \, \jmath'_{U, \ket, !}\bigl((\bZ / p^m)(d)\bigr)\Bigr) \otimes_{\widehat{\bZ}_p} \bigl(\AAinfX{X'} / (p^m, [\varpi^n])\bigr) \\
        & \to \Bigl(\upsilon_{X'}^{-1} \, \jmath_{U', \ket, !}\bigl((\bZ / p^m)(d)\bigr)\Bigr) \otimes_{\widehat{\bZ}_p} \bigl(\AAinfX{X'} / (p^m, [\varpi^n])\bigr)
    \end{split}
    \]
    over $X'_{K, \proket}$, for each $m \geq 1$ and each $n \geq 1$, and we have
    \[
    \begin{split}
        & \Bigl(\upsilon_{X'}^{-1} \, \jmath'_{U, \ket, !}\bigl((\bZ / p^m)(d)\bigr)\Bigr) \otimes_{\widehat{\bZ}_p} \bigl(\AAinfX{X'} / (p^m, [\varpi^n])\bigr) \\
        & \Mi \Bigl(R\varepsilon_{\proket, *} \, \upsilon_X^{-1} \, \jmath_{U, \ket, !}\bigl((\bZ / p^m)(d)\bigr)\Bigr) \otimes_{\widehat{\bZ}_p} \bigl(\AAinfX{X'} / (p^m, [\varpi^n])\bigr) \\
        & \Mi R\varepsilon_{\proket, *}\Bigl(\upsilon_X^{-1} \, \jmath_{U, \ket, !}\bigl((\bZ / p^m)(d)\bigr) \otimes_{\widehat{\bZ}_p} \bigl(\AAinfX{X} / (p^m, [\varpi^n])\bigr)\Bigr)
    \end{split}
    \]
    over $X'_{K, \proket}$, by induction on $m$ and $n$ based on \cite[\aLem \logadiclemketmorOplusp]{Diao/Lan/Liu/Zhu:lasfr}.  Therefore, by Lemmas \ref{lem-cplx-AAinf} and \ref{lem-BBdRp-cpt}, and by taking derived limits and inverting $p$, we see that the derived limit of \Refeq{\ref{eq-prop-exc-dR-comp-mor-proket}} also admits compatible morphisms to a morphism
    \begin{equation}\label{eq-prop-exc-dR-comp-ex-seq-proket-BBdR}
        R\varepsilon_{\proket, *}\bigl(\widehat{\bZ}_p(d) \otimes_{\widehat{\bZ}_p} \BBdRX{X}^\Sc\bigr) \to \widehat{\bZ}_p(d) \otimes_{\widehat{\bZ}_p} \BBdRX{X'}^\Scalt,
    \end{equation}
    where $\BBdRX{X}^\Sc$ and $\BBdRX{X'}^\Scalt$ are as in Definition \ref{def-BBdRp-cpt}, with $I^\Sc = I$ and $I^\Scalt = I - \{ j_0 \}$, respectively.  Then \Refeq{\ref{eq-prop-exc-dR-comp-ex-seq-proket-BBdR}} induces the second row of \Refeq{\ref{eq-prop-exc-dR-comp}}, by Proposition \ref{prop-L-!-coh-comp-proket}, which is an isomorphism because the first row is.

    Let $\mu'_X: {X_\proket}_{X_K} \to X_\an$ and $\mu'_{X'}: {X'_\proket}_{X'_K} \to X'_\an \cong X_\an$ denote the canonical morphisms of sites, so that $R\mu'_{X, *} \cong R\mu'_{X', *} \, R\varepsilon_{\proket, *}$.  By Proposition \ref{prop-Poin-lem} and the projection formula, by applying $R\mu'_{X', *}$  to \Refeq{\ref{eq-prop-exc-dR-comp-ex-seq-proket-BBdR}}, and by Proposition \ref{prop-RHl-Hl-Ddl-cpt} and Remark \ref{rem-unip-twist}, we obtain the canonical morphism
    \begin{equation}\label{eq-prop-exc-dR-comp-mor-dR-BdR}
        \bigl(\Omega_X^\bullet(\log D)(-D)(d)\bigr) \ho_k \BdR \to \bigl(\Omega_X^\bullet(\log D')(-D')(d)\bigr) \ho_k \BdR
    \end{equation}
    of complexes over $\cX$.  Note that, by construction, this is part of the pullback of \Refeq{\ref{eq-prop-exc-dR-comp-ex-seq-dR}}.  Therefore, this \Refeq{\ref{eq-prop-exc-dR-comp-mor-dR-BdR}} in turn induces the third row of \Refeq{\ref{eq-prop-exc-dR-comp}}, which can be compatibly identified with the second row by the comparison isomorphisms in Theorem \ref{thm-L-!-coh-comp}; and the whole diagram \Refeq{\ref{eq-prop-exc-dR-comp}} is commutative, with the fourth row given by the excision isomorphism for de Rham cohomology, as desired.
\end{proof}

Next, let us consider the Gysin isomorphisms between top-degree cohomology.

\begin{rk}\label{rem-Gysin-seq}
    Suppose that $I = \{ j_0 \}$ is a singleton, so that $D = D_{j_0}$ is an irreducible smooth divisor by assumption.  Suppose moreover that $X_K$ and $D_K$ are connected, in which case $X$ and $D$ are both geometrically connected.  Let $\jmath_U: U \to X$ and $\imath_D: D \to X$ denote the canonical open and closed immersions \Pth{of underlying adic spaces, without log structures}.  Consider the canonical distinguished triangle
    \begin{equation}\label{eq-rem-Gysin-seq-pre}
        \tau_{\leq 0} \, R\jmath_{U, \et, *}\bigl(\bZ_p(d)\bigr) \to R\jmath_{U, \et, *}\bigl(\bZ_p(d)\bigr) \to \tau_{\geq 1} \, R\jmath_{U, \et, *}\bigl(\bZ_p(d)\bigr) \Mapn{+1}
    \end{equation}
    over $X_{K, \et}$.  By canonically identifying the two truncations in \Refeq{\ref{eq-rem-Gysin-seq-pre}} using \cite[\aLem \logadiclemkettoetconst]{Diao/Lan/Liu/Zhu:lasfr}, we obtain a canonical distinguished triangle
    \begin{equation}\label{eq-rem-Gysin-seq}
        \bZ_p(d) \to R\jmath_{U, \et, *}\bigl(\bZ_p(d)\bigr) \to \imath_{D, \et, *}\bigl(\bZ_p(d - 1)\bigr)[-1] \Mapn{+1}.
    \end{equation}
    Note that, because of the proof of \cite[\aLem \logadiclemkettoetconst]{Diao/Lan/Liu/Zhu:lasfr}, this is compatible \Pth{via \cite[\aProp 2.1.4 and \aThm 3.8.1]{Huber:1996-ERA}} with the algebraic construction in \cite[\aSec 4]{Faltings:2002-aee}.  This is also consistent with the results in \cite[\aSec 3.9]{Huber:1996-ERA}.
\end{rk}

\begin{lem}\label{lem-Gysin-isom}
    With the same setting as in Remark \ref{rem-Gysin-seq}, the distinguished triangle \Refeq{\ref{eq-rem-Gysin-seq}} induces the Gysin isomorphism
    \begin{equation}\label{eq-lem-Gysin-isom}
        H_\et^{2d - 2}\bigl(D_K, \bQ_p(d - 1)\bigr) \Mi H_\et^{2d}\bigl(X_K, \bQ_p(d)\bigr).
    \end{equation}
\end{lem}
\begin{proof}
    Since $X$ and $D$ are geometrically connected, we have $H^0(X_\an, \cO_X) \cong k$ and $H^0(D_\an, \cO_D) \cong k$.  Moreover, we have a long exact sequence
    \[
        0 \to H^0\bigl(X_\an, \cO_X(-D)\bigr) \to H^0(X_\an, \cO_X) \to H^0(D_\an, \cO_D) \to \cdots,
    \]
    which forces $H^0\bigl(X_\an, \cO_X(-D)\bigr) = 0$ because $H^0(X_\an, \cO_X) \to H^0(D_\an, \cO_D)$ maps $1$ to $1$ by definition.  This shows that $H_\dR^0(X_\an, \cO_X) \cong k$, $H_\dR^0(D_\an, \cO_D) \cong k$, and $H_{\dR, \cpt}^0(U_\an, \cO_U) = 0$.  By using the perfect Poincar\'e duality pairing \Refeq{\ref{eq-thm-trace-dR-pairing}}, we obtain $H_\dR^{2d}\bigl(X_\an, \cO_X(d)\bigr) \cong k$, $H_\dR^{2d - 2}\bigl(D_\an, \cO_D(d - 1)\bigr) \cong k$, and $H_\dR^{2d}\bigl(U_\an, \cO_U(d)\bigr) = 0$.  By Theorem \ref{thm-L-!-coh-comp}, we obtain $H_\et^{2d}\bigl(X_K, \bQ_p(d)\bigr) \cong \bQ_p$, $H_\et^{2d - 2}\bigl(D_K, \bQ_p(d - 1)\bigr) \cong \bQ_p$, and $H_\et^{2d}\bigl(U_K, \bQ_p(d)\bigr) = 0$.  Now the desired isomorphism \Refeq{\ref{eq-lem-Gysin-isom}} is just a connecting morphism in the long exact sequence associated with \Refeq{\ref{eq-rem-Gysin-seq}} \Pth{and with $p$ inverted}, which is an isomorphism by comparison of dimensions over $\bQ_p$.
\end{proof}

\begin{prop}\label{prop-Gysin-dR-comp}
    With the same setting as in Remark \ref{rem-Gysin-seq} and Lemma \ref{lem-Gysin-isom}, the isomorphism \Refeq{\ref{eq-lem-Gysin-isom}} extends to a commutative diagram
    \begin{equation}\label{eq-prop-Gysin-dR-comp}
        \xymatrix{ {H_\et^{2d - 2}\bigl(D_K, \bQ_p(d - 1)\bigr)} \ar[r]^-\sim \ar@{^(->}[d] & {H_\et^{2d}\bigl(X_K, \bQ_p(d)\bigr)} \ar@{^(->}[d] \\
        {H_\et^{2d - 2}\bigl(D_K, \bQ_p(d - 1)\bigr) \otimes_{\bQ_p} \BdR} \ar[r]^-\sim \ar[d] & {H_\et^{2d}\bigl(X_K, \bQ_p(d)\bigr) \otimes_{\bQ_p} \BdR} \ar[d] \\
        {H_\dR^{2d - 2}\bigl(D_\an, \cO_D(d - 1)\bigr) \otimes_k \BdR} \ar[r]^-\sim & {H_\dR^{2d}\bigl(X_\an, \cO_X(d)\bigr) \otimes_k \BdR} \\
        {H_\dR^{2d - 2}\bigl(D_\an, \cO_D(d - 1)\bigr)} \ar[r]^-\sim \ar@{^(->}[u] & {H_\dR^{2d}\bigl(X_\an, \cO_X(d)\bigr)} \ar@{^(->}[u] }
    \end{equation}
    Here the fourth row \Pth{at the bottom} is the Gysin isomorphism induced by the long exact sequence associated with the adjunction exact sequence
    \begin{equation}\label{eq-prop-Gysin-dR-comp-ex-seq-dR}
        0 \to \Omega_X^\bullet(d) \to \Omega_X^\bullet(\log D)(d) \to \imath_{D, \an, *}\bigl(\Omega_D^\bullet(d - 1)\bigr)[-1] \to 0
    \end{equation}
    for de Rham complexes over $X_\an$, which is an isomorphism, as explained in the proof of Lemma \ref{lem-Gysin-isom}, because $H_\dR^{2d}\bigl(X_\an, \cO_X(d)\bigr) \cong k$, $H_\dR^{2d - 2}\bigl(D_\an, \cO_D(d - 1)\bigr) \cong k$, and $H_{\dR, \cpt}^{2d}\bigl(U_\an, \cO_U(d)\bigr) \cong H^{2d}\bigl(X_\an, \Omega_X^\bullet(\log D)(d)\bigr) = 0$.  Moreover, this isomorphism at the bottom row is compatible with the trace morphisms.
\end{prop}
\begin{proof}
    Let $X^\times$ denote the log adic space with the same underlying space as $X$, but equipped with the trivial log structure.  On the contrary, let $D^\partial := X_{\{ j_0 \}}^\partial$, as in Section \ref{sec-log-str-bd}, and let $D$ be equipped with the trivial log structure.  Let $\jmath_U: U \to X$, $\jmath_U^\times: U \to X^\times$, $\imath_D^\partial: D^\partial \to X$, $\imath_D: D \to X^\times$, $\varepsilon: X \to X^\times$, $\varepsilon_D^\partial: D^\partial \to D$ denote the canonical morphisms of log adic spaces.  Since
    \[
        R\jmath_{U, \ket, *}^\times\bigl(\bZ_p(d)\bigr) \cong R\varepsilon_{\ket, *} \, R\jmath_{U, \ket, *}\bigl(\bZ_p(d)\bigr) \cong R\varepsilon_{\ket, *}\bigl(\bZ_p(d)\bigr)
    \]
    and
    \[
        \imath_{D, \ket, *} \, R\varepsilon_{D, \ket, *}^\partial\bigl(\bZ_p(d - 1)\bigr) \cong R\varepsilon_{\ket, *} \, \imath_{D, \ket, *}^\partial\bigl(\bZ_p(d - 1)\bigr),
    \]
    by \cite[\aLem \logadiclemclimmketmor, \aThm \logadicthmpurity, and \aCor \logadiccorpuritylisse]{Diao/Lan/Liu/Zhu:lasfr}, \Refeq{\ref{eq-rem-Gysin-seq}} induces \Pth{and is induced by} a distinguished triangle
    \begin{equation}\label{eq-prop-Gysin-dR-comp-seq-ket}
        \bZ_p(d) \to R\varepsilon_{\ket, *}\bigl(\bZ_p(d)\bigr) \to \imath_{D, \ket, *}\bigl(\bZ_p(d - 1)\bigr)[-1] \Mapn{+1}
    \end{equation}
    over $X^\times_{K, \ket} \cong X_{K, \et}$, which induces \Refeq{\ref{eq-lem-Gysin-isom}} and the first row of \Refeq{\ref{eq-prop-Gysin-dR-comp}}.  Let $\varpi \in K^{\flat+}$ be such that $\varpi^\sharp = p$.  By \cite[\aProp \logadicpropproketvsketadj]{Diao/Lan/Liu/Zhu:lasfr}, the distinguished triangle
    \[
        (\bZ / p^m)(d) \to R\varepsilon_{\ket, *}\bigl((\bZ / p^m)(d)\bigr) \to
        \imath_{D, \ket, *}\bigl((\bZ / p^m)(d - 1)\bigr)[-1] \Mapn{+1}
    \]
    \Pth{induced by \Refeq{\ref{eq-prop-Gysin-dR-comp-seq-ket}}} is the pushforward of the distinguished triangle
    \begin{equation}\label{eq-prop-Gysin-dR-comp-seq-proket}
    \begin{split}
        & \upsilon_{X^\times}^{-1}\bigl((\bZ / p^m)(d)\bigr) \to \upsilon_{X^\times}^{-1} \, R\varepsilon_{\ket, *}\bigl((\bZ / p^m)(d)\bigr) \\
        & \to
        \upsilon_{X^\times}^{-1} \, \imath_{D, \ket, *}\bigl((\bZ / p^m)(d - 1)\bigr)[-1] \Mapn{+1}
    \end{split}
    \end{equation}
    over $X^\times_{K, \proket}$, which admits a morphism to the distinguished triangle
    \[
    \begin{split}
        & \Bigl(\upsilon_{X^\times}^{-1} \, \bigl((\bZ / p^m)(d)\bigr)\Bigr) \otimes_{\widehat{\bZ}_p} \bigl(\AAinfX{X^\times} / (p^m, [\varpi^n])\bigr) \\
        & \to \Bigl(\upsilon_{X^\times}^{-1} \, R\varepsilon_{\ket, *}\bigl((\bZ / p^m)(d)\bigr)\Bigr) \otimes_{\widehat{\bZ}_p} \bigl(\AAinfX{X^\times} / (p^m, [\varpi^n])\bigr) \\
        & \to \Bigl(\upsilon_{X^\times}^{-1} \, \imath_{D, \ket, *}\bigl((\bZ / p^m)(d - 1)\bigr)[-1]\Bigr) \otimes_{\widehat{\bZ}_p} \bigl(\AAinfX{X^\times} / (p^m, [\varpi^n])\bigr) \Mapn{+1}
    \end{split}
    \]
    over $X^\times_{K, \proket}$, for each $m \geq 1$ and each $n \geq 1$.  Therefore, by induction on $m$ and $n$ based on \cite[\aLems \logadiclemketmorOplusp{} and \logadiclemclimmOplusp]{Diao/Lan/Liu/Zhu:lasfr}, and by taking derived limits and inverting $p$, we see that the derived limit of \Refeq{\ref{eq-prop-Gysin-dR-comp-seq-proket}} also admits a morphism to a distinguished triangle
    \begin{equation}\label{eq-prop-Gysin-dR-comp-seq-proket-BBdR}
    \begin{split}
        & \bigl(\widehat{\bZ}_p(d)\bigr) \otimes_{\widehat{\bZ}_p} \BBdRX{X^\times} \to R\varepsilon_{\proket, *}\Bigl(\bigl(\widehat{\bZ}_p(d)\bigr) \otimes_{\widehat{\bZ}_p} \BBdRX{X}\Bigr) \to \\
        & \imath_{D, \proket, *}\Bigl(\bigl(\widehat{\bZ}_p(d - 1)\bigr)[-1] \otimes_{\widehat{\bZ}_p} \BBdRX{D}\Bigr) \Mapn{+1}
    \end{split}
    \end{equation}
    over $X^\times_{K, \proket}$.  Then \Refeq{\ref{eq-prop-Gysin-dR-comp-seq-proket-BBdR}} induces the second row of \Refeq{\ref{eq-prop-Gysin-dR-comp}}, by Proposition \ref{prop-L-!-coh-comp-proket} \Pth{or \cite[\aThm 8.4]{Scholze:2013-phtra}}, which is an isomorphism because the first row is.

    Let $\mu_X': {X_\proket}_{/X_K} \to X_\an$ and $\mu_{X^\times}': {X^\times_\proket}_{/X^\times_K} \to X^\times_\an \cong X_\an$ denote the canonical morphisms of sites, so that $R\mu_{X, *}' = R\mu_{X^\times, *}' \, R\varepsilon_{\proket, *}$.  By Proposition \ref{prop-Poin-lem} and the projection formula, by applying $R\mu_{X^\times, *}'$ to \Refeq{\ref{eq-prop-Gysin-dR-comp-seq-proket-BBdR}}, and by Proposition \ref{prop-RHl-Hl-Ddl-cpt} and Remark \ref{rem-unip-twist}, we obtain a distinguished triangle
    \begin{equation}\label{eq-prop-Gysin-dR-comp-seq-dR-BdR}
    \begin{split}
        & \bigl(\Omega_X^\bullet(d)\bigr) \ho_k \BdR \to \bigl(\Omega_X^\bullet(\log D)(d)\bigr) \ho_k \BdR \to \\
        & \Bigl(\imath_{D, \an, *}\bigl(\Omega_D^\bullet(d - 1)\bigr)[-1]\Bigr) \ho_k \BdR \Mapn{+1}
    \end{split}
    \end{equation}
    of complexes over $\cX^\times \cong \cX$, and we have
    \[
        \Bigl(\imath_{D, \an, *}\bigl(\Omega_D^\bullet(d - 1)\bigr)[-1]\Bigr) \ho_k \BdR \Mi \imath_{\cD, *}\Bigl(\bigl(\Omega_D^\bullet(d - 1)\bigr) \ho_k \BdR\Bigr)[-1].
    \]
    This \Refeq{\ref{eq-prop-Gysin-dR-comp-seq-dR-BdR}} induces the third row of \Refeq{\ref{eq-prop-Gysin-dR-comp}}, which can be compatibly identified with the second row by the comparison isomorphisms in Theorem \ref{thm-L-!-coh-comp} \Pth{or rather \cite[\aThm 8.4]{Scholze:2013-phtra}}.

    We claim that \Refeq{\ref{eq-prop-Gysin-dR-comp-seq-dR-BdR}} is canonically isomorphic to the pullback of \Refeq{\ref{eq-prop-Gysin-dR-comp-ex-seq-dR}}.  It is clear that the first morphism in \Refeq{\ref{eq-prop-Gysin-dR-comp-seq-dR-BdR}} is the canonical one and hence coincides with the pullback of the first morphism in \Refeq{\ref{eq-prop-Gysin-dR-comp-ex-seq-dR}}.  As for the second morphism, it suffices to show that it induces the pullback to $\cD$ of the canonical morphism $\imath_D^*\bigl(\Omega_X^\bullet(\log D)(d)\bigr) \to \Omega_D^\bullet(d - 1)[-1]$ induced by adjunction.  Once this is known, the third morphism must be zero, and the claim would follow.

    We shall first show this \'etale locally, by adapting the arguments in \cite[\aSec \logRHseccoh]{Diao/Lan/Liu/Zhu:lrhrv}.  Suppose that $X = \Spa(R, R^+)$ is affinoid and admits a strictly \'etale morphism
    \[
        X \to \bE := \bT^{n - 1} \times \bD := \Spa(k\Talg{T_1^{\pm 1}, \ldots, T_{n - 1}^{\pm 1}, T_n}, k^+\Talg{T_1^{\pm 1}, \ldots, T_{n - 1}^{\pm 1}, T_n}),
    \]
    and that the underlying adic space $D$ is the pullback of $\{ T_n = 0 \}$, in which case $D = \Spa(\overline{R}, \overline{R}^+)$ with $\overline{R} := R / (T_n)$.  Recall that $K = \widehat{\AC{k}}$.  Let us take finite extensions $k_m$ of $k$ in $\AC{k}$ such that $k_m$ contains all $m$-th roots of unity in $\AC{k}$, for each $m \geq 1$, and such that $\AC{k} = \cup_m \, k_m$.  For each $m \geq 1$, let
    \[
        \bE_m := \bT^{n - 1}_m \times \bD_m := \Spa(k_m\Talg{T_1^{\pm \frac{1}{m}}, \ldots, T_{n - 1}^{\pm \frac{1}{m}}, T_n^{\frac{1}{m}}}, k_m^+\Talg{T_1^{\pm \frac{1}{m}}, \ldots, T_{n - 1}^{\pm \frac{1}{m}}, T_n^{\frac{1}{m}}}),
    \]
    and let $X_m := X \times_\bE \bE_m$ and $D^\partial_m := D^\partial \times_\bE \bE_m$.  Then $\widetilde{X} := \varprojlim_m X_m \to X_K$ and $\widetilde{D}^\partial := \varprojlim_m D^\partial_m \to D^\partial_K$ are Galois pro-Kummer \'etale covers with Galois group $\Gamma_\geom \cong (\widehat{\bZ}(1))^n$, and we have $\widetilde{D}^\partial \cong \widetilde{X} \times_X D^\partial$.  Similarly, we have a strictly \'etale morphism $D \to \bT^{n - 1}$ \Pth{compatible with the above $D^\partial \to \bE = \bT^{n - 1} \times \bD$}, with Kummer \'etale covers $\bT^{n - 1}_m \to \bT^{n - 1}$ inducing $D_m := X \times_{\bD^{n - 1}} \bD^{n - 1}_m$, and with a Galois pro-Kummer \'etale cover $\widetilde{D} := \varprojlim_m D_m \to D_K$ with Galois group $\overline{\Gamma}_\geom \cong (\widehat{\bZ}(1))^{n - 1}$.  As explained in \cite[\aSec \logadicsectoricchart]{Diao/Lan/Liu/Zhu:lasfr}, $\widetilde{X}$ and $\widetilde{D}$ are log affinoid perfectoid objects in $X_\proket$ and $D_\proket$, respectively.  By \cite[\aLem \logadiclemlogaffperfclimm]{Diao/Lan/Liu/Zhu:lasfr}, $\widetilde{D}^\partial$ is also a log affinoid perfectoid object of $D^\partial_\proket$.  By construction, the induced morphism $\widetilde{D}^\partial \to \widetilde{D}$ is Galois with Galois group
    \[
        \Gamma^\partial := \ker(\Gamma_\geom \to \overline{\Gamma}_\geom) \cong \widehat{\bZ}(1).
    \]
    Therefore, the higher direct images along the canonical morphisms of sites $\nu_X': {X_\proket}_{/X_K} \to X_\et$, $\nu_{D^\partial}': {D^\partial_\proket}_{/D^\partial_K} \to D^\partial_\et \cong D_\et$, $\varepsilon^\partial_{D_K, \proket}: {D^\partial_\proket}_{/D^\partial_K} \to {D_\proket}_{/D_K}$, and $\nu_D': {D_\proket}_{/D_K} \to D_\et$, when computed using the \v{C}ech cohomology of the pro-Kummer \'etale covers $\widetilde{X}^\partial \to X_K$, $\widetilde{D}^\partial \to D^\partial_K$, $\widetilde{D}^\partial \to \widetilde{D}_K$, and $\widetilde{D} \to D_K$, correspond to the group cohomology of $\Gamma_\geom$, $\Gamma_\geom$, $\Gamma^\partial$, and $\overline{\Gamma}_\geom$, respectively.  Thus, as in the proof of Proposition \ref{prop-RHl-Hl-Ddl-cpt}, by the same arguments as in the proofs of \cite[\aThm 2.1(iii)]{Liu/Zhu:2017-rrhpl} and \cite[\aProp \logRHpropLOCl]{Diao/Lan/Liu/Zhu:lrhrv}, we may compute $\imath_D^*\bigl(\Omega_X^\bullet(\log D)(d)\bigr)$ by working with $\nu_{D^\partial}'$ instead of $\nu_X'$.

    Let $\gamma_j \in \Gamma_\geom$ be topological generators such that $\gamma_j \, T_{j'}^{\frac{1}{m}} = \zeta_m^{\delta_{jj'}} T_{j'}^{\frac{1}{m}}$, as in \cite[(\logRHeqdefgammaj)]{Diao/Lan/Liu/Zhu:lrhrv}, for all $j, j' = 1, \ldots, n$, so that $\overline{\Gamma}_\geom$ \Pth{\resp $\Gamma^\partial$} is topologically generated by $\gamma_1, \ldots, \gamma_{n - 1}$ \Pth{\resp $\gamma_n$}.  These depend on some compatible choices of roots of unity, as in \cite[(\logRHeqzeta)]{Diao/Lan/Liu/Zhu:lrhrv}, which are equivalent to the choice of an isomorphism $\widehat{\bZ}(1) \Mi \widehat{\bZ}$, and we will use the same choices to trivialize $\widehat{\bZ}(1)$ and $\bZ_p(1)$ in the following.  Moreover, as in \cite[(\logRHeqchoicet)]{Diao/Lan/Liu/Zhu:lrhrv}, the chosen $\bZ_p(1) \Mi \bZ_p$ canonically defines an element $t \in \BdR$.  By sending the preimage of $1 \in \bZ_p$ to $t \in \BdR$, we obtain a canonical $\Gal(\AC{k} / k)$-equivariant morphism $\bZ_p(1) \to \BdR$ of $\bZ_p$-modules, which is \Pth{by definition} independent of the choice of $\bZ_p(1) \Mi \bZ_p$.

    As usual, given any topological $\Gamma^\partial \rtimes \Gal(\AC{k} / k)$-module $L$, with the above choices, its group cohomology with respect to the subgroup $\Gamma^\partial \cong \widehat{\bZ}(1)$ can be computed by the two-term complex $L \Mapn{\gamma_n - 1} L(-1)$.  Consequently, $\bigl(R\varepsilon^\partial_{D_K, \proket, *}\bigl(\widehat{\bZ}_p(d)\bigr)\bigr)(\widetilde{D})$ can be represented by the complex $\bZ_p(d) \Mapn{\gamma_n - 1} \bZ_p(d - 1)$, where $\gamma_n - 1$ acts by zero and hence the complex just splits \Pth{\Refcf{} \cite[\aLem \logadiclemclimmketmor]{Diao/Lan/Liu/Zhu:lrhrv}}; and the pullback of the second morphism in \Refeq{\ref{eq-prop-Gysin-dR-comp-seq-ket}} corresponds to the canonical morphism
    \begin{equation}\label{eq-prop-Gysin-dR-comp-mor-Z-p-cplx}
        [\bZ_p(d) \Mapn{0} \bZ_p(d - 1)] \to \bZ_p(d - 1)[-1]
    \end{equation}
    given by the identity morphisms on $\bZ_p(d - 1)[-1]$.  By the explicit descriptions in \cite[\aSec \logRHsecOBdlexplicit]{Diao/Lan/Liu/Zhu:lrhrv}, the canonical morphism $\BBdRX{D}(\widetilde{D}) \to \BBdRX{D^\partial}(\widetilde{D}^\partial)$ is an isomorphism, and $\gamma_n - 1$ acts by zero on $\BBdRX{D^\partial}^\partial(\widetilde{D}^\partial)$.  Let us introduce
    \begin{equation}\label{eq-prop-Gysin-dR-comp-def-B}
        \mathbf{B} := \BBdRX{D}(\widetilde{D}) \cong \BBdRX{D^\partial}(\widetilde{D}^\partial),
    \end{equation}
    for simplicity of notation.  Therefore, in a way consistent with \Refeq{\ref{eq-prop-Gysin-dR-comp-mor-Z-p-cplx}} and \Refeq{\ref{eq-prop-Gysin-dR-comp-def-B}}, the pullback of the second morphism of \Refeq{\ref{eq-prop-Gysin-dR-comp-seq-proket-BBdR}} induces the canonical morphism
    \begin{equation}\label{eq-prop-Gysin-dR-comp-mor-BdR-cplx}
        [\mathbf{B}(d) \Mapn{0} \mathbf{B}(d - 1)] \to \mathbf{B}(d - 1)[-1]
    \end{equation}
    given by the identity morphism on $\mathbf{B}(d - 1)[-1]$.  Again, for simplicity of notation, let $\mathbf{\Omega}^\bullet := \Omega_D^\bullet(D)$ and $\mathbf{\Omega}^{\log, \bullet} := \Omega_{D^\partial}^{\log, \bullet}(D)$, where $(D)$ denotes the evaluation on the whole affinoid $D$.  Then $\mathbf{\Omega}^1 \cong \oplus_{j = 1}^{n - 1} (\overline{R} \, dT_j)$ and $\mathbf{\Omega}^{\log, 1} \cong \mathbf{\Omega}^1 \oplus (\overline{R} \, \frac{dT_n}{T_n})$, and we have $\mathbf{\Omega}^{\log, \bullet} \cong \mathbf{\Omega}^\bullet \oplus (\mathbf{\Omega}^\bullet[-1] \wedge \frac{dT_n}{T_n})$.  By Lemma \ref{lem-OBdRp-bd-loc} and Corollary \ref{cor-OBdRp-cplx-bd}, and by \cite[\aCor \logRHcorOBdlplocgr]{Diao/Lan/Liu/Zhu:lrhrv}, the values of the log de Rham complexes for $\OBdlX{D}$ and $\OBdlX{D^\partial}$ on $\widetilde{D}$ and $\widetilde{D}^\partial$, respectively, are given by the complexes $\mathbf{B}\{W_1, \ldots, W_{n - 1}\} \otimes_R \mathbf{\Omega}^\bullet$ and $\mathbf{B}\{W_1, \ldots, W_n\} \otimes_R \mathbf{\Omega}^{\log, \bullet}$, where the differentials are defined by mapping $W_j$ to $t^{-1} \frac{dT_j}{T_j}$, for each $j$, as in \cite[(\logRHeqconnWi)]{Diao/Lan/Liu/Zhu:lrhrv}.  Since $\gamma_n$ acts trivially on $\mathbf{B}\{W_1, \ldots, W_{n - 1}\}$ and $\gamma_n W_n = W_n - 1$ \Pth{\Refcf{} the proof of \cite[\aLem \logRHlemGammageominv]{Diao/Lan/Liu/Zhu:lrhrv}}, we have $H^i(\Gamma^\partial, \mathbf{B}\{W_1, \ldots, W_n\}) = 0$, for all $i > 0$; and the $\Gamma^\partial$-invariants in $\mathbf{B}\{W_1, \ldots, W_n\} \otimes_R \mathbf{\Omega}^{\log, \bullet}$ form the filtered subcomplex
    \[
        \mathbf{B}\{W_1, \ldots, W_{n - 1}\} \otimes_R \mathbf{\Omega}^{\log, \bullet} \cong \mathbf{B}\{W_1, \ldots, W_{n - 1}\} \otimes_R \bigl(\mathbf{\Omega}^\bullet \oplus (\mathbf{\Omega}^\bullet[-1] \wedge \tfrac{dT_n}{T_n})\bigr).
    \]
    By explicit computations as in the proof of \cite[\aCor \logRHcorlogdRcplx]{Diao/Lan/Liu/Zhu:lrhrv}, we have following:
    \begin{enumerate}
        \item The canonical morphism of filtered complexes
            \[
                \mathbf{B}(d) \to (\mathbf{B}\{W_1, \ldots, W_{n - 1}\})(d) \otimes_R \mathbf{\Omega}^\bullet,
            \]
            mapping $\mathbf{B}(d)$ to $\mathbf{B}(d) \otimes 1$ in degree zero via the identity morphism on $\mathbf{B}(d)$, is a filtered quasi-isomorphism.

        \item The canonical morphism of filtered complexes
            \[
            \begin{split}
                [\mathbf{B}(d) \Mapn{0} \mathbf{B}(d - 1)] \to & \bigl((\mathbf{B}\{W_1, \ldots, W_{n - 1}\})(d) \otimes_R \mathbf{\Omega}^\bullet\bigr) \\
                & \quad \oplus \bigl((\mathbf{B}\{W_1, \ldots, W_{n - 1}\})(d - 1) \otimes_R (\mathbf{\Omega}^\bullet[-1] \wedge \tfrac{dT_n}{T_n})\bigr)
            \end{split}
            \]
            mapping the first term $\mathbf{B}(d)$ to $\mathbf{B}(d) \otimes 1$ in degree zero and the second term $\mathbf{B}(d - 1)$ to $\mathbf{B}(d - 1) \otimes \frac{dT_n}{T_n}$ in degree one via the identity morphisms on $\mathbf{B}(d)$ and $\mathbf{B}(d - 1)$, respectively, is a filtered quasi-isomorphism.  \Pth{Note that the filtration on the second term of $[\mathbf{B}(d) \Mapn{0} \mathbf{B}(d - 1)]$ is shifted by one.}

        \item Via the above two quasi-isomorphisms, the morphism \Refeq{\ref{eq-prop-Gysin-dR-comp-mor-BdR-cplx}} is quasi-isomorphic to the morphism
            \[
            \begin{split}
                & \bigl((\mathbf{B}\{W_1, \ldots, W_{n - 1}\})(d) \otimes_R \mathbf{\Omega}^\bullet\bigr) \\
                & \quad \oplus \bigl((\mathbf{B}\{W_1, \ldots, W_{n - 1}\})(d - 1) \otimes_R (\mathbf{\Omega}^\bullet[-1] \wedge \tfrac{dT_n}{T_n})\bigr) \\
                & \to (\mathbf{B}\{W_1, \ldots, W_{n - 1}\})(d - 1) \otimes_R \mathbf{\Omega}^\bullet[-1]
            \end{split}
            \]
            defined by extracting the factor $\frac{dT_n}{T_n}$.
    \end{enumerate}
    By taking $\overline{\Gamma}_\geom$-invariants, which computes the direct images along $\widetilde{D} \to D_K$ \Pth{with vanishing higher direct images, as explained in \cite[\aSec \logRHseccoh]{Diao/Lan/Liu/Zhu:lrhrv}}, and by canonically identifying Tate twists of $\BdR$-modules using the above morphism $\bZ_p(1) \to \BdR$, the last morphism induces the canonical morphism
    \[
        \bigl(\mathbf{\Omega}^{\log, \bullet}(d)\bigr) \ho_k \BdR \to \bigl(\mathbf{\Omega}^\bullet(d - 1)[-1]\bigr) \ho_k \BdR
    \]
    extracting the factor $\frac{dT_n}{T_n}$, which is the same morphism defined by the pullback to $\cD$ of the adjunction morphism $\imath_D^*\bigl(\Omega_X^\bullet(\log D)(d)\bigr) \to \Omega_D^\bullet(d - 1)[-1]$.  \Pth{Note that Tate twists on log de Rham complexes are only shifts of Hodge filtrations.}  Since all the above identifications are canonical, they globalize and the claim follows.

    Thus, the whole diagram \Refeq{\ref{eq-prop-Gysin-dR-comp}} is commutative, with the fourth row given by the Gysin isomorphism for de Rham cohomology, which is compatible with the trace morphisms \Pth{for de Rham cohomology} by Lemma \ref{lem-property-trace}\Refenum{\ref{lem-property-trace-Gysin}}, as desired.
\end{proof}

\subsection{Poincar\'e duality for \'etale cohomology}\label{sec-trace-et}

\begin{thm}\label{thm-trace-et}
    There exists a unique morphism
    \begin{equation}\label{eq-trace-et}
        t_\et: H_{\et, \cpt}^{2d}\bigl(U_K, \bQ_p(d)\bigr) \to \bQ_p,
    \end{equation}
    which we shall call the \emph{trace morphism}, satisfying the following requirements:
    \begin{enumerate}
        \item\label{thm-trace-et-res}  The formation of $t_\et$ is compatible with restrictions to open rigid analytic subvarieties of the form $U - Z = X - D - Z$ for some closed rigid analytic subvarieties $Z$ of $X$.  \Pth{Such open rigid analytic subvarieties are allowed in our setting by resolution of singularities, as in \cite{Bierstone/Milman:1997-cdbml}, by the independence of the choice of compactifications in the definition of cohomology with compact support, based on Lemmas \ref{lem-L-!} and \ref{lem-def-H-c-fin-Z-p}, and on Remark \ref{rem-def-H-c-conv}.}

        \item\label{thm-trace-et-comp}  Suppose that $U_K$ is connected.  Then $t_\et$ and $t_\dR$ are both isomorphisms, and the comparison isomorphism
            \[
                H_{\et, \cpt}^{2d}\bigl(U_K, \bQ_p(d)\bigr) \otimes_{\bQ_p} \BdR \cong H_{\dR, \cpt}^{2d}\bigl(U, \cO_U(d)\bigr) \otimes_k \BdR
            \]
            \Pth{see Theorem \ref{thm-L-!-coh-comp}} maps $\bigl(t_\et^{-1}(1)\bigr) \otimes 1$ to $\bigl(t_\dR^{-1}(1)\bigr) \otimes 1$.  Consequently, the formation of $t_\et$ is compatible with the replacement of $k$ with a finite extension in $\AC{k}$ and with the Gysin isomorphism in the top row of \Refeq{\ref{eq-prop-Gysin-dR-comp}} \Pth{by Proposition \ref{prop-Gysin-dR-comp}, because the formation of $t_\dR$ is compatible with the Gysin isomorphism in the bottom row of \Refeq{\ref{eq-prop-Gysin-dR-comp}}}; and $t_\et$ is the trivial isomorphism $H_\et^0(U_K, \bQ_p) \cong \bQ_p$ when $d = \dim(U) = \dim(X) = 0$.
    \end{enumerate}
    Moreover, such a morphism \Refeq{\ref{eq-trace-et}} satisfies the following properties:
    \begin{enumerate}[resume]
        \item\label{thm-trace-et-pairing}  By pre-composition with the canonical cup product pairing, $t_\et$ induces a perfect pairing
            \begin{equation}\label{eq-thm-trace-et-pairing}
                H_{\et, \Sc}^i(U_K, \bL_{\bQ_p}) \times H_{\et, \Snc}^{2d - i}\bigl(U_K, \dual{\bL}_{\bQ_p}(d)\bigr) \to \bQ_p,
            \end{equation}
            for each $\bZ_p$-local system $\bL$ on $X_\ket$ \Pth{even when $\bL|_U$ is not de Rham}.

        \item\label{thm-trace-et-pairing-dR} When $\bL|_U$ is \emph{de Rham}, we also have a commutative diagram
            \begin{equation}\label{eq-thm-trace-et-pairing-dR-diag}
                \xymatrix{ {H_{\et, \Sc}^i(U_K, \bL_{\bQ_p})} \ar[r]^-\sim \ar@{^(->}[d] & {\Hom_{\bQ_p}\bigl(H_{\et, \Snc}^{2d - i}\bigl(U_K, \dual{\bL}_{\bQ_p}(d)\bigr), \bQ_p\bigr)} \ar@{^(->}[d] \\
                {H_{\et, \Sc}^i(U_K, \bL_{\bQ_p}) \otimes_{\bQ_p} \BdR} \ar[r]^-\sim \ar[d]_-\wr & {\Hom_{\bQ_p}\bigl(H_{\et, \Snc}^{2d - i}\bigl(U_K, \dual{\bL}_{\bQ_p}(d)\bigr), \BdR\bigr)} \ar[d]^-\wr \\
                {H_{\dR, \Sc}^i\bigl(U_\an, \DdR(\bL_{\bQ_p})\bigr) \otimes_k \BdR} \ar[r]^-\sim & {\Hom_k\bigl(H_{\dR, \Snc}^{2d - i}\bigl(U_\an, \DdR\bigl(\dual{\bL}_{\bQ_p}(d)\bigr)\bigr), \BdR\bigr)} \\
                {H_{\dR, \Sc}^i\bigl(U_\an, \DdR(\bL_{\bQ_p})\bigr)} \ar[r]^-\sim \ar@{^(->}[u] & {\Hom_k\bigl(H_{\dR, \Snc}^{2d - i}\bigl(U_\an, \DdR\bigl(\dual{\bL}_{\bQ_p}(d)\bigr)\bigr), k\bigr)} \ar@{^(->}[u] }
            \end{equation}
            in which the top \Pth{\resp bottom} two rows are induced by $t_\et$ \Pth{\resp $t_\dR$}.
    \end{enumerate}
\end{thm}
\begin{proof}
    For our purpose, we may replace $k$ with a finite extension over which the connected components of $X$ are geometrically connected, and replace $X$ with its geometric connected components.  Then we may assume that $X$ is geometrically connected.  \Pth{Then the trace morphism to be constructed will be isomorphisms.}  We may also assume that $X$ contains a $k$-point $\Spa(k, k^+) \cong S \Em X$.  Let us proceed by induction on $d = \dim(U) = \dim(X)$.

    If $d = 0$, then $U_K$ is a single $K$-point, and $H_{\et, \cpt}^0(U_K, \bQ_p) \cong H_\et^0(X_K, \bQ_p)$ has a canonical element given by the identity section, which defines the trace isomorphism $t_\et: H_{\et, \cpt}^0(U_K, \bQ_p) \Mi \bQ_p$.  The same identity section induces the identity section of $H^0(X_{K, \proket}, \BBdR)$, which is also induced by the identity section of $H_\dR^0(X_\an, \cO_X) \cong H^0(X_\an, \cO_X)$.  Hence, $t_\et$ satisfies the requirement \Refenum{\ref{thm-trace-et-comp}}.  It is straightforward that it also satisfies \Refenum{\ref{thm-trace-et-res}}, \Refenum{\ref{thm-trace-et-pairing}}, and \Refenum{\ref{thm-trace-et-pairing-dR}}.

    If $d > 0$, we first construct a trace morphism $t_\et: H_{\et, \cpt}^{2d}\bigl(U_K, \bQ_p(d)\bigr) \to \bQ_p$ satisfying the requirement \Refenum{\ref{thm-trace-et-comp}}.  By Lemma \ref{lem-trace-dR} and Proposition \ref{prop-exc-dR-comp}, we are reduced to the case where $U = X$ and $D = \emptyset$.  Let $Y$ denote the blowup of $X$ along the $k$-point $S$ \Pth{\Refcf{} \cite[\aDef 4.1.1]{Conrad:2006-rarg}}, and let $E$ denote the exceptional divisor.  Since $S$ is a $k$-point, both $Y$ and $E$ are smooth and geometrically connected.  \Pth{Since $Y$ is \'etale locally isomorphic to $\bD_k^n$ for some $n$, this can be seen by an explicit local construction.}  Then we have a commutative diagram of canonical morphisms
    \[
        \xymatrix{ {H_{\et, \cpt}^{2d}\bigl((X - S)_K, \bZ_p(d)\bigr)} \ar[d]_-\wr \ar[r]^-\sim & {H_{\et, \cpt}^{2d}\bigl((Y - E)_K, \bZ_p(d)\bigr)} \ar[d]^-\wr \\
        {H_\et^{2d}\bigl(X_K, \bZ_p(d)\bigr)} \ar[r] & {H_\et^{2d}\bigl(Y_K, \bZ_p(d)\bigr)} }
    \]
    in which the two vertical morphisms are isomorphisms because $H_\et^i\bigl(S_K, \bZ_p(d)\bigr)$ and $H_\et^i\bigl(E_K, \bZ_p(d)\bigr)$ are zero for $i > 2d - 2$, since both $S$ and $E$ are proper smooth of dimensions no greater than $d - 1$, forcing the bottom row in the diagram to be also an isomorphism.  Then we have a commutative diagram of canonical morphisms
    \begin{equation}\label{eq-thm-trace-et-blowup}
        \xymatrix{ {H_\et^{2d}\bigl(X_K, \bQ_p(d)\bigr)} \ar@{^(->}[d] \ar[r]^-\sim & {H_\et^{2d}\bigl(Y_K, \bQ_p(d)\bigr)} \ar@{^(->}[d] \\
        {H_\et^{2d}\bigl(X_K, \bQ_p(d)\bigr) \otimes_{\bQ_p} \BdR} \ar[d]_-\wr \ar[r]^-\sim & {H_\et^{2d}\bigl(Y_K, \bQ_p(d)\bigr) \otimes_{\bQ_p} \BdR} \ar[d]^-\wr \\
        {H_\dR^{2d}\bigl(X_\an, \cO_X(d)\bigr) \otimes_k \BdR} \ar[r]^-\sim & {H_\dR^{2d}\bigl(Y_\an, \cO_Y(d)\bigr) \otimes_k \BdR} \\
        {H_\dR^{2d}\bigl(X_\an, \cO_X(d)\bigr)} \ar@{^(->}[u] \ar[r]^-\sim & {H_\dR^{2d}\bigl(Y_\an, \cO_Y(d)\bigr)} \ar@{^(->}[u] }
    \end{equation}
    in which the bottom row is an isomorphism by Lemma \ref{lem-trace-dR}, and in which the middle square is commutative because both of the middle two rows are induced by the canonical morphism
    \[
        H^{2d}\bigl(X_{K, \proket}, \bigl(\widehat{\bZ}_p(d)\bigr) \otimes_{\widehat{\bZ}_p} \BBdRX{X}\bigr) \to H^{2d}\bigl(Y_{K, \proket}, \bigl(\widehat{\bZ}_p(d)\bigr) \otimes_{\widehat{\bZ}_p} \BBdRX{Y}\bigr).
    \]
    In order to construct $t_\et: H_\et^{2d}\bigl(X_K, \bQ_p(d)\bigr) \Mi \bQ_p$ satisfying the requirement \Refenum{\ref{thm-trace-et-comp}}, it suffices to show that $\bigl(t_\dR^{-1}(1)\bigr) \otimes 1 \in H_\dR^{2d}\bigl(X_\an, \cO_X(d)\bigr) \otimes_k \BdR$ lies in the image of $H_\et^{2d}\bigl(X_K, \bQ_p(d)\bigr)$, so that we can define $t_\et^{-1}(1)$ to be the preimage of $\bigl(t_\dR^{-1}(1)\bigr) \otimes 1$.  \Pth{Note that this does not involve the choice of $S$, and the compatibility with the replacement of $k$ with a finite extension in $\AC{k}$ is clear.}  By using the commutative diagrams \Refeq{\ref{eq-thm-trace-et-blowup}} and \Refeq{\ref{eq-prop-Gysin-dR-comp}}, it suffices to note that, by the induction hypothesis, the analogous assertion holds for $\bigl(t_\dR^{-1}(1)\bigr) \otimes 1 \in H_\dR^{2d - 2}\bigl(E_\an, \cO_E(d - 1)\bigr) \otimes_k \BdR$.

    Such a $t_\et: H_\et^{2d}\bigl(X_K, \bQ_p(d)\bigr) \Mi \bQ_p$ satisfies the requirement \Refenum{\ref{thm-trace-et-res}} because, in the setting of Lemma \ref{lem-trace-dR}, we can choose to blowup at some $k$-point $S$ of $U' \subset U$ \Pth{which exists up to replacing $k$ with a finite extension in $\AC{k}$}, so that we have canonical isomorphisms $H_{\et, \cpt}^{2d}\bigl(U'_K, \bQ_p(d)\bigr) \cong H_{\et, \cpt}^{2d}\bigl(X_K, \bQ_p(d)\bigr) \Mi H_{\et, \cpt}^{2d}\bigl(X'_K, \bQ_p(d)\bigr) \cong H_{\et, \cpt}^{2d}\bigl(U'_K, \bQ_p(d)\bigr)$ because they are all isomorphic to $H_\et^{2d - 2}\bigl(E_K, \bQ_p(d - 1)\bigr)$ via compatible canonical morphisms, and these canonical isomorphisms extend to a commutative diagram \Pth{as in \Refeq{\ref{eq-thm-trace-et-blowup}} and \Refeq{\ref{eq-prop-Gysin-dR-comp}}} involving also their de Rham counterparts and their tensor products with $\BdR$.

    Finally, let us verify the properties \Refenum{\ref{thm-trace-et-pairing}} and \Refenum{\ref{thm-trace-et-pairing-dR}}.  Since $K$ is a field extension of $\bQ_p$, and since the duality pairings are defined by composition with cup product pairings, which are compatible with Higgs comparison isomorphisms as in \Refeq{\ref{eq-thm-L-!-coh-comp-Hi}}, the desired perfect pairing \Refeq{\ref{eq-thm-trace-et-pairing}} for \'etale cohomology follows from the perfect pairing \Refeq{\ref{eq-thm-Higgs-pairing}} for Higgs cohomology.  When $\bL|_U$ is \emph{de Rham}, again since the duality pairings are defined by composition with cup product pairings, and since the de Rham comparison isomorphisms are compatible with the Higgs ones by construction, we have the desired commutative diagram \Refeq{\ref{eq-thm-trace-et-pairing-dR-diag}}, in which the middle square is commutative because both of the middle two rows are induced by the same cup product pairing $H^i(X_{K, \proket}, \widehat{\bL} \otimes_{\widehat{\bZ}_p} \BBdRX{X}^\Sc) \otimes_{\BdR} H^{2d - i}\bigl(X_{K, \proket}, \bigl(\dual{\widehat{\bL}}(d)\bigr) \otimes_{\widehat{\bZ}_p} \BBdRX{X}^\Snc\bigr) \to H^{2d}\bigl(X_{K, \proket}, \bigl(\widehat{\bZ}_p(d)\bigr) \otimes_{\widehat{\bZ}_p} \BBdRX{X}^\cpt\bigr)$, where $\BBdRX{X}^\cpt$ is the analogue of $\BBdRX{X}^\Sc$ when $I^\Sc$ is replaced with $I$.  \Pth{See Definition \ref{def-BBdRp-cpt}.  Note that, since $I = I^\Sc \cup I^\Snc$, the multiplication morphism $\BBdRX{X}^\Sc \otimes_{\widehat{\bZ}_p} \BBdRX{X}^\Snc \to \BBdRX{X}$ factors through $\BBdRX{X}^\cpt$.}
\end{proof}

\subsection{De Rham comparison for generalized interior cohomology}\label{sec-int-coh}

\begin{defn}\label{def-int-coh}
    For any $I^\Scalt \subset I^\Sc \subset I$, we consider the \emph{generalized interior cohomology} \Pth{\Refcf{} Definitions \ref{def-H-c} and \ref{def-dR-Hi-Hdg-coh-cpt}}
    \begin{equation}\label{eq-def-int-coh-et}
        H_{\et, \Sc \to \Scalt}^i(U_K, \bL) := \Image\bigl(H_{\et, \Sc}^i(U_K, \bL) \to H_{\et, \Scalt}^i(U_K, \bL)\bigr),
    \end{equation}
    \begin{equation}\label{eq-def-int-coh-dR}
    \begin{split}
        & H_{\dR, \Sc \to \Scalt}^i\bigl(U_\an, \DdR(\bL)\bigr) \\
        & := \Image\Bigl(H_{\dR, \Sc}^i\bigl(U_\an, \DdR(\bL)\bigr) \to H_{\dR, \Scalt}^i\bigl(U_\an, \DdR(\bL)\bigr)\Bigr),
    \end{split}
    \end{equation}
    and
    \begin{equation}\label{eq-def-int-coh-Hdg}
    \begin{split}
        & H_{\Hdg, \Sc \to \Scalt}^{a, i - a}\bigl(U_\an, \DdR(\bL)\bigr) \\
        & := \Image\Bigl(H_{\Hdg, \Sc}^{a, i - a}\bigl(U_\an, \DdR(\bL)\bigr) \to H_{\Hdg, \Scalt}^{a, i - a}\bigl(U_\an, \DdR(\bL)\bigr)\Bigr),
    \end{split}
    \end{equation}
    for all $i \geq 0$ and $a \in \bZ$.  When $I^\Sc = I$ and $I^\Scalt = \emptyset$, we shall denote the objects with subscripts \Qtn{$\intcoh$} instead of \Qtn{$\Sc \to \Scalt$}, and call them the \emph{interior cohomology}.
\end{defn}

\begin{lem}\label{lem-int-coh-pairing-compat}
    Suppose that $I^\Scalt \subset I^\Sc \subset I$.  The Poincar\'e duality pairings
    \begin{equation}\label{eq-lem-int-coh-pairing-compat-c-nc}
        H_{\et, \Sc}^i(U_K, \bL_{\bQ_p}) \times H_{\et, \Snc}^{2d - i}\bigl(U_K, \dual{\bL}_{\bQ_p}(d)\bigr) \to \bQ_p
    \end{equation}
    and
    \begin{equation}\label{eq-lem-int-coh-pairing-compat-c-alt-nc-alt}
        H_{\et, \Scalt}^i(U_K, \bL_{\bQ_p}) \times H_{\et, \Sncalt}^{2d - i}\bigl(U_K, \dual{\bL}_{\bQ_p}(d)\bigr) \to \bQ_p
    \end{equation}
    induce the same pairing
    \begin{equation}\label{eq-lem-int-coh-pairing-compat-c-nc-alt}
        H_{\et, \Sc}^i(U_K, \bL_{\bQ_p}) \times H_{\et, \Sncalt}^{2d - i}\bigl(U_K, \dual{\bL}_{\bQ_p}(d)\bigr) \to \bQ_p
    \end{equation}
    \Pth{which is defined because $I^\Sc \cup I^\Sncalt = I$ under the condition $I^\Scalt \subset I^\Sc$}.  Consequently, if $x_\Sc \in H_{\et, \Sc}^i(U_K, \bL_{\bQ_p})$ is mapped to $x_\Scalt \in H_{\et, \Scalt}^i(U_K, \bL_{\bQ_p})$, and if $y_\Sncalt \in H_{\et, \Sncalt}^{2d - i}\bigl(U_K, \dual{\bL}_{\bQ_p}(d)\bigr)$ is mapped to $y_\Snc \in H_{\et, \Snc}^{2d - i}\bigl(U_K, \dual{\bL}_{\bQ_p}(d)\bigr)$, then we have
    \[
        \langle x_\Sc, y_\Snc \rangle = \langle x_\Sc, y_\Sncalt \rangle = \langle x_\Scalt, y_\Sncalt \rangle.
    \]
    The analogous assertion for the Poincar\'e duality pairings on the de Rham cohomology of $\DdR(\bL)$ and $\DdR\bigl(\dual{\bL}(d)\bigr)$ is also true.
\end{lem}
\begin{proof}
    This is because the pairings \Refeq{\ref{eq-lem-int-coh-pairing-compat-c-nc}} and \Refeq{\ref{eq-lem-int-coh-pairing-compat-c-alt-nc-alt}} are both compatible with the cup product pairing $H_{\et, \Sc}^i(U_K, \bL_{\bQ_p}) \times H_{\et, \Sncalt}^{2d - i}\bigl(U_K, \dual{\bL}_{\bQ_p}(d)\bigr) \to H_{\et, \cpt}^{2d}\bigl(U_K, \bQ_p(d)\bigr)$ inducing \Refeq{\ref{eq-lem-int-coh-pairing-compat-c-nc-alt}}.  \Pth{The assertion for de Rham cohomology is similar.}
\end{proof}

\begin{prop}\label{prop-int-coh-pairing-perf}
    For any $I^\Scalt \subset I^\Sc \subset I$, the Poincar\'e duality pairing \Refeq{\ref{eq-lem-int-coh-pairing-compat-c-nc}} \Pth{based on \Refeq{\ref{eq-thm-trace-et-pairing}}} induces a canonical prefect pairing
    \begin{equation}\label{eq-prop-int-coh-pairing}
        H_{\et, \Sc \to \Scalt}^i(U_K, \bL_{\bQ_p}) \times H_{\et, \Sncalt \to \Snc}^{2d - i}\bigl(U_K, \dual{\bL}_{\bQ_p}(d)\bigr) \to \bQ_p,
    \end{equation}
    which we also call the Poincar\'e duality pairing, by setting
    \begin{equation}\label{eq-prop-int-coh-pairing-def}
        \langle x, y \rangle = \langle \widetilde{x}, y \rangle
    \end{equation}
    for $x \in H_{\et, \Sc \to \Scalt}^i(U_K, \bL_{\bQ_p})$ and $y \in H_{\et, \Sncalt \to \Snc}^{2d - i}\bigl(U_K, \dual{\bL}_{\bQ_p}(d)\bigr)$, if $x$ is the image of some $\widetilde{x} \in H_{\et, \Sc}^i(U_K, \bL_{\bQ_p})$.  When $I^\Sc = I$ and $I^\Scalt = \emptyset$, in which case $I^\Snc = \emptyset$ and $I^\Sncalt = I$, this defines the Poincar\'e duality pairing
    \begin{equation}\label{eq-prop-int-coh-pairing-spec}
        H_{\et, \intcoh}^i(U_K, \bL_{\bQ_p}) \times H_{\et, \intcoh}^{2d - i}\bigl(U_K, \dual{\bL}_{\bQ_p}(d)\bigr) \to \bQ_p
    \end{equation}
    for interior cohomology.  This pairing \Refeq{\ref{eq-prop-int-coh-pairing}} is well defined.  When $\bL|_U$ is \emph{de Rham}, the analogous assertion for the Poincar\'e duality pairings on the de Rham cohomology of $\DdR(\bL)$ and $\DdR\bigl(\dual{\bL}(d)\bigr)$ is also true.
\end{prop}
\begin{proof}
    To show that the pairing \Refeq{\ref{eq-prop-int-coh-pairing}} is well defined, suppose $x$ is lifted to another element $\widetilde{x}' \in H_{\et, \Sc}^i(U_K, \bL_{\bQ_p})$.  By definition, $y$ is the image of some $\widetilde{y} \in H_{\et, \Sncalt}^{2d - i}\bigl(U_K, \dual{\bL}_{\bQ_p}(d)\bigr)$.  Then we have $\langle \widetilde{x} - \widetilde{x}', y \rangle = \langle 0, \widetilde{y} \rangle = 0$, by Lemma \ref{lem-int-coh-pairing-compat}, showing that we still have $\langle \widetilde{x}, y \rangle = \langle \widetilde{x}', y \rangle$.

    To show that the pairing \Refeq{\ref{eq-prop-int-coh-pairing}} is perfect, let $\{ e_1, \ldots, e_r \}$ be any $\bQ_p$-basis of $H_{\et, \Sc \to \Scalt}^i(U_K, \bL_{\bQ_p})$, which can be extended to some $\bQ_p$-basis $\{ e_1, ..., e_s \}$ of $H_{\et, \Scalt}^i(U_K, \bL_{\bQ_p})$.  Let $\{ \widetilde{f}_1, \ldots, \widetilde{f}_s \}$ denote the dual $\bQ_p$-basis of $H_{\et, \Sncalt}^{2d - i}\bigl(U_K, \dual{\bL}_{\bQ_p}(d)\bigr)$ under the perfect pairing \Refeq{\ref{eq-lem-int-coh-pairing-compat-c-nc}}.  For each $j = 1, \ldots, s$, let $f_j$ denote the image of $\widetilde{f}_j$ in $H_{\et, \Snc}^{2d - i}\bigl(U_K, \dual{\bL}_{\bQ_p}(d)\bigr)$.  For each $j = 1, \ldots, r$, let $\widetilde{e}_j$ denote some element of $H_{\et, \Sc}^i(U_K, \bL_{\bQ_p})$ lifting $e_j$.  For each $j = r + 1, \ldots, s$, if $f_j \neq 0$, then there exists some $\widetilde{e}$ in $H_{\et, \Sc}^i(U_K, \bL_{\bQ_p})$, with image $e$ in $H_{\et, \Sc \to \Scalt}^i(U_K, \bL_{\bQ_p})$, such that $1 = \langle \widetilde{e}, f_j \rangle$, by the perfectness of \Refeq{\ref{eq-lem-int-coh-pairing-compat-c-nc}}.  But this contradicts $\langle \widetilde{e}, f_j \rangle = \langle e, \widetilde{f}_j \rangle = 0$, and hence $f_j = 0$ for all $j > r$.  If $\sum_{j = 1}^r a_j f_j = 0$ in $H_{\et, \Snc}^{2d - i}\bigl(U_K, \dual{\bL}_{\bQ_p}(d)\bigr)$, then $a_{j_0} = \langle e_{j_0}, \sum_{j = 1}^r a_j \widetilde{f}_j \rangle = \langle \widetilde{e}_{j_0}, \sum_{j = 1}^r a_j f_j \rangle = 0$, for all $j_0 = 1, \ldots, r$.  It follows that $\{ f_1, \ldots, f_r \}$ is a $\bQ_p$-basis of $H_{\et, \Sncalt \to \Snc}^{2d - i}\bigl(U_K, \dual{\bL}_{\bQ_p}(d)\bigr)$, which is dual to the $\bQ_p$-basis $\{ e_1, \ldots, e_r \}$ of $H_{\et, \Sc \to \Scalt}^i(U_K, \bL_{\bQ_p})$ under the induced pairing \Refeq{\ref{eq-prop-int-coh-pairing}}, as desired.  \Pth{The assertion for de Rham cohomology is similar.}
\end{proof}

\begin{lem}\label{lem-fil-mor}
    Let $(F_1, \Fil_{F_1}^\bullet)$ and $(F_2, \Fil_{F_2}^\bullet)$ be two filtered vector spaces \Pth{over some fixed base field, which we shall omit}, with a map $F_1 \to F_2$ compatible with filtrations such that $F_3 := \Image(F_1 \to F_2)$ is finite-dimensional.  Suppose that $\dim\bigl(\Image(\gr_{F_1} \to \gr_{F_2})\bigr) = \dim(F_3)$.  Then $\Fil_{F_1}^\bullet$ and $\Fil_{F_2}^\bullet$ are \emph{strictly compatible} in the sense that $\Image(\Fil_{F_1}^\bullet \to F_3)$ and $\Fil_{F_3}^\bullet := \Fil_{F_2}^\bullet \cap F_3$ coincide as filtrations on $F_3$, and we have an induced isomorphism $\Image(\gr_{F_1} \to \gr_{F_2}) \Mi \gr_{F_3}$.
\end{lem}
\begin{proof}
    For each $a \in \bZ$, the map $\gr_{F_1}^a = \Fil_{F_1}^a / \Fil_{F_1}^{a + 1} \to \gr_{F_2}^a = \Fil_{F_2}^a / \Fil_{F_2}^{a + 1}$ factors through $\gr_{F_3}^a = \Fil_{F_3}^a / \Fil_{F_3}^{a + 1}$ with image $\Image(\Fil_{F_1}^a \to F_3) / \Fil_{F_3}^{a + 1}$.  Hence, the assumption that $\dim\bigl(\Image(\gr_{F_1} \to \gr_{F_2})\bigr) = \dim(F_3) = \dim(\gr_{F_3})$ implies the strict compatibility $\Image(\Fil_{F_1}^\bullet \to F_3) = \Fil_{F_3}^\bullet$ and induces $\Image(\gr_{F_1} \to \gr_{F_2}) \Mi \gr_{F_3}$.
\end{proof}

\begin{thm}\label{thm-int-coh-comp}
    When $\bL|_U$ is \emph{de Rham}, the comparison isomorphisms in Theorem \ref{thm-L-!-coh-comp} are compatible with the canonical morphisms induced by any inclusions $I^\Scalt \subset I^\Sc \subset I$, and hence with the comparison isomorphisms in \cite[\aThm \logRHthmlogRHarith(\logRHthmlogRHarithcomp)]{Diao/Lan/Liu/Zhu:lrhrv} \Pth{corresponding to $I^\Scalt = \emptyset$; \Refcf{} the notation in Definition \ref{def-dR-Hi-Hdg-coh-cpt}}, in the sense that we have $\Gal(\AC{k} / k)$-equivariant commutative diagrams
    \begin{equation}\label{eq-thm-int-coh-comp-diag-dR}
        \xymatrix{ {H_{\et, \Sc}^i(U_K, \bL) \otimes_{\bZ_p} \BdR} \ar[r]^-\sim \ar[d] & {H_{\dR, \Sc}^i\bigl(U_\an, \DdR(\bL)\bigr) \otimes_k \BdR} \ar[d] \\
        {H_{\et, \Scalt}^i(U_K, \bL) \otimes_{\bZ_p} \BdR} \ar[r]^-\sim & {H_{\dR, \Scalt}^i\bigl(U_\an, \DdR(\bL)\bigr) \otimes_k \BdR} }
    \end{equation}
    and
    \begin{equation}\label{eq-thm-int-coh-comp-diag-Hdg}
        \xymatrix@C=3ex{ {H_{\et, \Sc}^i(U_K, \bL) \otimes_{\bZ_p} K} \ar[r]^-\sim \ar[d] & {\oplus_{a + b = i} \Bigl(H_{\Hdg, \Sc}^{a, b}\bigl(U_\an, \DdR(\bL)\bigr) \otimes_k K(-a)\Bigr)} \ar[d] \\
        {H_{\et, \Scalt}^i(U_K, \bL) \otimes_{\bZ_p} K} \ar[r]^-\sim & {\oplus_{a + b = i} \Bigl(H_{\Hdg, \Scalt}^{a, b}\bigl(U_\an, \DdR(\bL)\bigr) \otimes_k K(-a)\Bigr)} }
    \end{equation}
    of canonical morphisms, which are compatible with the Hodge--de Rham spectral sequences, for each integer $i \geq 0$.  Hence, we have $\Gal(\AC{k} / k)$-equivariant isomorphisms
    \begin{equation}\label{eq-thm-int-coh-comp-isom-dR}
        H_{\et, \Sc \to \Scalt}^i(U_K, \bL) \otimes_{\bZ_p} \BdR \Mi H_{\dR, \Sc \to \Scalt}^i\bigl(U_\an, \DdR(\bL)\bigr) \otimes_k \BdR
    \end{equation}
    and
    \begin{equation}\label{eq-thm-int-coh-comp-isom-Hdg}
        H_{\et, \Sc \to \Scalt}^i(U_K, \bL) \otimes_{\bZ_p} K \Mi \oplus_{a + b = i} \Bigl(H_{\Hdg, \Sc \to \Scalt}^{a, b}\bigl(U_\an, \DdR(\bL)\bigr) \otimes_k K(-a)\Bigr)
    \end{equation}
    which are compatible with the prefect Poincar\'e duality pairings on both sides.  Moreover, the Hodge filtrations on $H_{\dR, \Sc}^i\bigl(U_\an, \DdR(\bL)\bigr)$ and $H_{\dR, \Scalt}^i\bigl(U_\an, \DdR(\bL)\bigr)$ are \emph{strictly compatible} in the sense \Pth{as in Lemma \ref{lem-fil-mor}} that they induce the same filtration on $H_{\dR, \Sc \to \Scalt}^i\bigl(U_\an, \DdR(\bL)\bigr)$, which we shall still call the \emph{Hodge filtration}, and we have a canonical graded isomorphism
    \begin{equation}\label{eq-thm-int-coh-comp-gr-isom}
        \gr H_{\dR, \Sc \to \Scalt}^i\bigl(U_\an, \DdR(\bL)\bigr) \cong \oplus_{a + b = i} \, H_{\Hdg, \Sc \to \Scalt}^{a, b}\bigl(U_\an, \DdR(\bL)\bigr),
    \end{equation}
    \Pth{matching $\gr^a H_{\dR, \Sc \to \Scalt}^i\bigl(U_\an, \DdR(\bL)\bigr)$ with $H_{\Hdg, \Sc \to \Scalt}^{a, i - a}\bigl(U_\an, \DdR(\bL)\bigr)$}.
\end{thm}
\begin{proof}
    We have the commutative diagram \Refeq{\ref{eq-thm-int-coh-comp-diag-dR}} because, by Proposition \ref{prop-L-!-coh-comp-proket} and the proof of Theorem \ref{thm-L-!-coh-comp}, the morphism in both columns can be identified with the morphism
    \[
        H^i\bigl(X_{K, \proket}, \widehat{\bL} \otimes_{\widehat{\bZ}_p} \BBdR^\Sc\bigr) \to H^i\bigl(X_{K, \proket}, \widehat{\bL} \otimes_{\widehat{\bZ}_p} \BBdR^\Scalt\bigr)
    \]
    induced by the canonical morphism $\BBdR^\Sc \to \BBdR^\Scalt$ \Pth{which exists by the very construction of these sheaves in Definition \ref{def-BBdRp-cpt}}.  Similarly, we have the commutative diagram \Refeq{\ref{eq-thm-int-coh-comp-diag-Hdg}} because, also by Proposition \ref{prop-L-!-coh-comp-proket} and the proof of Theorem \ref{thm-L-!-coh-comp}, the morphism in both columns can be identified with the morphism
    \[
        H^i\bigl(X_{K, \proket}, \widehat{\bL} \otimes_{\widehat{\bZ}_p} \widehat{\cO}_{X_{K, \proket}}^\Sc\bigr) \to H^i\bigl(X_{K, \proket}, \widehat{\bL} \otimes_{\widehat{\bZ}_p} \widehat{\cO}_{X_{K, \proket}}^\Scalt\bigr)
    \]
    induced by the canonical morphism $\widehat{\cO}_{X_{K, \proket}}^\Sc \to \widehat{\cO}_{X_{K, \proket}}^\Scalt$.  The commutative diagrams \Refeq{\ref{eq-thm-int-coh-comp-diag-dR}} and \Refeq{\ref{eq-thm-int-coh-comp-diag-Hdg}} are compatible with the Hodge--de Rham spectral sequences by Proposition \ref{prop-Poin-lem}\Refenum{\ref{prop-Poin-lem-4}}, and the comparison isomorphisms \Refeq{\ref{eq-thm-int-coh-comp-isom-dR}} and \Refeq{\ref{eq-thm-int-coh-comp-isom-Hdg}} thus obtained are compatible with the Poincar\'e duality pairings on generalized interior cohomology because they are induced by comparison isomorphisms respecting the original Poincar\'e duality pairings.  Since the Hodge--de Rham spectral sequences for $H_{\dR, \Sc}^i\bigl(U_\an, \DdR(\bL)\bigr)$ and $H_{\dR, \Scalt}^i\bigl(U_\an, \DdR(\bL)\bigr)$ degenerate on the $E_1$ pages by Theorem \ref{thm-L-!-coh-comp}, and since \Refeq{\ref{eq-thm-int-coh-comp-isom-dR}} and \Refeq{\ref{eq-thm-int-coh-comp-isom-Hdg}} imply that
    \[
    \begin{split}
        & \sum_{a + b = i} \dim_k\bigl(H_{\Hdg, \Sc \to \Scalt}^{a, b}\bigl(U_\an, \DdR(\bL)\bigr)\bigr) \\
        & = \dim_{\bQ_p}\bigl(H_{\et, \Sc \to \Scalt}^i\bigl(U_K, \bL\bigr)\bigr) = \dim_k\bigr(H_{\dR, \Sc \to \Scalt}^i\bigl(U_\an, \DdR(\bL)\bigr)\bigr),
    \end{split}
    \]
    the last assertion of the proposition follows from Lemma \ref{lem-fil-mor}.
\end{proof}

\begin{cor}\label{cor-et-qis}
    Let $I^+_\arith$ be as in Lemma \ref{lem-dR-qis-arith}, and let $I^\Scalt$ be as in \Refeq{\ref{eq-cor-dR-qis-arith-cond}}.  Then we have a canonical isomorphism $H_{\et, \Sc}^i(U_K, \bL_{\bQ_p}) \cong H_{\et, \Scalt}^i(U_K, \bL_{\bQ_p})$, for each $i \geq 0$ and each $a \in \bZ$, which is compatible with \Refeq{\ref{eq-cor-dR-qis-arith}} and \Refeq{\ref{eq-cor-dR-qis-arith-Hdg}} via the comparison isomorphisms as in \Refeq{\ref{eq-thm-L-!-coh-comp-dR}} and \Refeq{\ref{eq-thm-L-!-coh-comp-Hdg}}.
\end{cor}
\begin{proof}
    We may assume that $I^\Scalt = I^\Sc - I^+_\arith \subset I^\Sc$, as in the proof of Corollary \ref{cor-dR-qis-arith}.  Since $\BdR$ is a field extension of $\bQ_p$, and since we have compatible canonical isomorphisms $H_{\et, ?}^i(U_K, \bL) \otimes_{\bZ_p} \BdR \cong H_{\et, ?}^i(U_K, \bL_{\bQ_p}) \otimes_{\bQ_p} \BdR$, for $? = \Sc$ and $\Scalt$, this corollary follows from Theorem \ref{thm-int-coh-comp} and Corollary \ref{cor-dR-qis-arith}.
\end{proof}

\numberwithin{equation}{section}

\section{Comparison theorems for smooth algebraic varieties}\label{sec-comp-alg}

In this section, we let $U$ denote a smooth algebraic variety over a $p$-adic field $k$.  Since $\chr(k) = 0$, by \cite{Nagata:1962-iavcv, Hironaka:1964-rsavz-1, Hironaka:1964-rsavz-2}, there exists a smooth compactification $X$ of $U$ such that the boundary $D = X - U$ \Pth{with its reduced subscheme structure} is a normal crossings divisor, and we may assume that the intersections of the irreducible components of $D$ are all smooth.  We shall denote the analytification of these schemes, viewed as adic spaces over $\Spa(k, \cO_k)$, with superscripts \Qtn{$\an$}, as usual.  Then the analytifications of $U$, $X$, and $D$ satisfy the same setup as in Section \ref{sec-log-str-bd}, and we shall inherit most of the notation from there, the main difference being that objects and morphisms with no superscript \Qtn{$\an$} \Pth{\resp with superscripts \Qtn{$\an$}} are the algebraic \Pth{\resp analytic} ones.  For any $I^\star \subset I$, we shall also consider $D^\Sc := \cup_{j \in I^\star} \, D_j$ and $D^\Snc := \cup_{j \in I - I^\star} \, D_j$ \Pth{with their canonical reduced closed subscheme structures}, and the objects they define.

As in \cite[\aSec \logRHsecDdRalg]{Diao/Lan/Liu/Zhu:lrhrv}, for each $\bZ_p$-local system $\bL$ on $U_\et$, we denote by $\bL^\an$ its analytification on $U^\an_\et$ as usual, and we write $\overline{\bL}^\an := \jmath_{\ket, *}^\an(\bL^\an)\cong R \jmath_{\ket, *}^\an(\bL^\an)$ \Pth{without introducing $\overline{\bL}$}, which is a $\bZ_p$-local system on $X^\an_\ket$.  We also consider $\jmath_{\et, !}^\Sc \, R\jmath_{\Sc, \et, *}(\bL)$ \Pth{\resp $\jmath_{\et, !}^{\Sc, \an} \, R\jmath_{\Sc, \et, *}^\an(\bL^\an)$} on $X_\et$ \Pth{\resp $X^\an_\et$}, and define
\[
\begin{split}
    & H_{\et, \Sc}^i\bigl(U_{\AC{k}}, \bL / p^m\bigr) := H^i\bigl(X_{\AC{k}, \ket}, \jmath_{\ket, !}^\Sc(\overline{\bL} / p^m)\bigr) \\
    & \cong H^i\bigl(X_{\AC{k}, \et}, \jmath_{\et, !}^\Sc \, R\jmath_{\Sc, \et, *}(\bL / p^m)\bigr) \cong H_\cpt^i\bigl(U^\Sc_{\AC{k}, \et}, R\jmath_{\Sc, \et, *}(\bL / p^m)\bigr)
\end{split}
\]
and
\[
    H_{\et, \Sc}^i(U_{\AC{k}}, \bL) := \varprojlim_m H_\Sc^i(U_{\AC{k}, \et}, \bL / p^m)
\]
\Pth{\Refcf{} Definitions \ref{def-H-c-torsion} and \ref{def-H-c}, Lemmas \ref{lem-L-!} and \ref{lem-def-H-c-fin-Z-p}, and Remark \ref{rem-def-H-c-conv}}.

\begin{lem}\label{lem-comp-def-H-c}
    There are canonical isomorphisms
    \begin{equation}\label{eq-lem-comp-def-H-c-sh}
    \begin{split}
        & \bigl(\jmath_{\et, !}^\Sc \, R\jmath_{\Sc, \et, *}(\bL)\bigr)^\an \cong \bigl(\varprojlim_m \, \jmath_{\et, !}^\Sc \, R\jmath_{\Sc, \et, *}(\bL / p^m)\bigr)^\an \\
        & \cong \varprojlim_m \, \jmath_{\et, !}^{\Sc, \an} \, R\jmath_{\Sc, \et, *}^\an(\bL^\an / p^m) \cong \jmath_{\et, !}^{\Sc, \an} \, R\jmath_{\Sc, \et, *}^\an(\bL^\an) \\
        & \cong \varprojlim_m \, R\varepsilon_{\et, *} \, \jmath_{\ket, !}^{\Sc, \an}\big((\overline{\bL}^\an / p^m)|_{U^{\Sc, \an}_\ket}\bigr) \cong R\varepsilon_{\et, *} \, \jmath_{\ket, !}^{\Sc, \an}(\overline{\bL}^\an|_{U^{\Sc, \an}_\ket}).
    \end{split}
    \end{equation}
    For all $i \geq 0$, we have
    \begin{equation}\label{eq-lem-comp-def-H-c}
    \begin{split}
        & H_{\et, \Sc}^i(U_{\AC{k}}, \bL) = \varprojlim_m H_{\et, \Sc}^i(U_{\AC{k}}, \bL / p^m) \\
        & \cong H_{\et, \Sc}^i(U^\an_{\AC{k}}, \bL) \cong \varprojlim_m H_{\et, \Sc}^i(U^\an_{\AC{k}}, \bL^\an / p^m)
    \end{split}
    \end{equation}
    \Pth{see Definition \ref{def-H-c} and Lemma \ref{lem-def-H-c-fin-Z-p}}, which can be identified with
    \[
    \begin{split}
        & H^i\bigl(X_{\AC{k}, \et}, \jmath_{\et, !}^\Sc \, R\jmath_{\Sc, \et, *}(\bL)\big) := \varprojlim_m H^i\bigl(X_{\AC{k}, \et}, \jmath_{\et, !}^\Sc \, R\jmath_{\Sc, \et, *}(\bL / p^m)\bigr) \\
        & \cong H^i\bigl(X^\an_{\AC{k}, \et}, \jmath_{\et, !}^{\Sc, \an} \, R\jmath_{\Sc, \et, *}^\an(\bL^\an)\bigr) \cong \varprojlim_m H^i\bigl(X^\an_{\AC{k}, \et}, \jmath_{\et, !}^{\Sc, \an} \, R\jmath_{\Sc, \et, *}^\an(\bL^\an / p^m)\bigr) \\
        & \cong H^i\bigl(X^\an_{\AC{k}, \ket}, \jmath_{\ket, !}^{\Sc, \an}(\overline{\bL}^\an|_{U^{\Sc, \an}_\ket})\bigr) \cong \varprojlim_m H^i\bigl(X^\an_{\AC{k}, \ket}, \jmath_{\ket, !}^{\Sc, \an}\bigl((\overline{\bL}^\an / p^m)|_{U^{\Sc, \an}_\ket}\bigr)\bigr)
    \end{split}
    \]
    \Pth{see Lemma \ref{lem-L-!}}.  For any $I^\Scalt \subset I^\Sc \subset I$, the isomorphisms in \Refeq{\ref{eq-lem-comp-def-H-c-sh}} and \Refeq{\ref{eq-lem-comp-def-H-c}} are compatible with the canonical morphisms from the analogous objects with subscripts \Qtn{$\cpt$} to those with subscripts \Qtn{$\Sc$}; from those with subscripts \Qtn{$\Sc$} to those with subscripts \Qtn{$\Scalt$}; and from those with subscripts \Qtn{$\Scalt$} to those without any of these subscripts \Pth{\Refcf{} Remark \ref{rem-def-H-c-alt}}.
\end{lem}
\begin{proof}
    These are based on the various definitions and Lemmas \ref{lem-L-!} and \ref{lem-def-H-c-fin-Z-p}, and the compatible isomorphisms $\bigl(\jmath_{\et, !}^\Sc \, R\jmath_{\Sc, \et, *}(\bL / p^m)\bigr)^\an \cong \jmath_{\et, !}^{\Sc, \an} \, R\jmath_{\Sc, \et, *}^\an(\bL^\an / p^m)$ and $H_\cpt^i\bigl(U^\Sc_{\AC{k}, \et}, R\jmath_{\Sc, \et, *}(\bL / p^m)\bigr) \cong H_\cpt^i\bigl(U^{\Sc, \an}_{\AC{k}, \et}, R\jmath_{\Sc, \et, *}^\an(\bL^\an / p^m)\bigr)$, for all $m \geq 1$, by \cite[\aProp 2.1.4, and \aThms 3.8.1 and 5.7.2]{Huber:1996-ERA}.
\end{proof}

\begin{lem}\label{lem-comp-def-H-c-trace}
    The usual algebraic trace morphism $t_\dR: H_{\dR, \cpt}^{2d}\bigl(U, \cO_U(d)\bigr) \to k$ is compatible with the analytic trace morphism $t_\dR^\an: H_{\dR, \cpt}^{2d}\bigl(U^\an_\an, \cO_{U^\an}(d)\bigr) \to k$ defined in Theorem \ref{thm-trace-dR-Hdg} via GAGA \cite{Kopf:1974-efava}.  Similarly, the usual algebraic trace morphism $t_\et: H_{\et, \cpt}^{2d}\bigl(U_{\AC{k}}, \bQ_p(d)\bigr) \to \bQ_p$ is compatible with the analytic trace morphism $t_\et^\an: H_{\et, \cpt}^{2d}\bigl(U^\an_{\AC{k}}, \bQ_p(d)\bigr) \to \bQ_p$ defined in Theorem \ref{thm-trace-et} under the canonical isomorphisms given by Lemma \ref{lem-comp-def-H-c}.
\end{lem}
\begin{proof}
    By Lemma \ref{lem-trace-dR}, Theorem \ref{thm-trace-et}\Refenum{\ref{thm-trace-et-res}}, and the corresponding facts for algebraic trace morphisms, up to replacing $k$ with a finite extension in $\AC{k}$, we may assume that $U_K = X_K$ is connected.  By the compatibility with Gysin isomorphisms in Proposition \ref{prop-Gysin-dR-comp} and Theorem \ref{thm-trace-et}\Refenum{\ref{thm-trace-et-comp}}, by the corresponding facts for algebraic trace morphisms, and by considering smooth divisors defined by blowing up $k$-points \Pth{which exists up to further replacing $k$ with a finite extension in $\AC{k}$} as in the proof of Theorem \ref{thm-trace-et}, we can proceed by induction and reduce to the case where $X_K$ is just a single $K$-point.  In this case, the algebraic and analytic trace isomorphisms for de Rham cohomology are the canonical isomorphisms $H^0(X, \cO_X) \cong k$ and $H^0(X^\an, \cO_{X^\an}) \cong k$, respectively, which are compatible via the canonical morphism $H^0(X, \cO_X) \to H^0(X^\an, \cO_{X^\an})$.  Also, the algebraic and analytic trace isomorphisms for \'etale cohomology are the canonical isomorphisms $H_\et^0(X_K, \bQ_p) \cong \bQ_p$ and $H_\et^0(X_K^\an, \bQ_p) \cong \bQ_p$, respectively, which are compatible via the canonical morphism $H_\et^0(X_K, \bQ_p) \to H_\et^0(X_K^\an, \bQ_p)$, as desired.
\end{proof}

\begin{thm}\label{thm-comp-alg-dR-cpt}
    Suppose that $\bL$ is de Rham.  Then we have canonical $\Gal(\AC{k} / k)$-equivariant filtered isomorphisms
    \begin{equation}\label{eq-comp-alg-dR-cpt}
        H_{\et, \Sc}^i(U_{\AC{k}}, \bL) \otimes_{\bZ_p} \BdR \cong H_{\dR, \Sc}^i\bigl(U, \DdRalg(\bL)\bigr) \otimes_k \BdR
    \end{equation}
    and
    \begin{equation}\label{eq-comp-alg-Hdg-cpt}
        H_{\et, \Sc}^i(U_{\AC{k}}, \bL) \otimes_{\bZ_p} \widehat{\AC{k}} \cong \oplus_{a + b = i} \, \Bigl(H_{\Hdg, \Sc}^i\bigl(U, \DdRalg(\bL)\bigr) \otimes_k \widehat{\AC{k}}(-a) \Bigr),
    \end{equation}
    where
    \begin{equation}\label{eq-def-alg-dR-coh-cpt}
        H_{\dR, \Sc}^i\bigl(U, \DdRalg(\bL)\bigr) := H^i\bigl(X, \DRl\bigl(\bigl(\Ddlalg(\bL)\bigr)(-D^\Sc)\bigr)\bigr)
    \end{equation}
    and
    \begin{equation}\label{eq-def-alg-Hdg-coh-cpt}
        H_{\Hdg, \Sc}^i\bigl(U, \DdRalg(\bL)\bigr) := H^i\bigl(X, \gr \DRl\bigl(\bigl(\Ddlalg(\bL)\bigr)(-D^\Sc)\bigr)\bigr)
    \end{equation}
    are abusively defined as in Definition \ref{def-dR-Hi-Hdg-coh-cpt} \Pth{for the filtered log connection $\Ddlalg(\bL)$ introduced in \cite[\aSec \logRHsecDdRalg]{Diao/Lan/Liu/Zhu:lrhrv}}.  Moreover, the Hodge--de Rham spectral sequence for $H_{\dR, \Sc}^i\bigl(U, \DdRalg(\bL)\bigr)$, as in \Refeq{\ref{eq-def-Hdg-dR-coh-cpt-spec-seq}}, degenerates on the $E_1$ page, and are compatible with the above comparison isomorphisms in the sense that \Refeq{\ref{eq-comp-alg-dR-cpt}} induces \Refeq{\ref{eq-comp-alg-Hdg-cpt}} by taking $\gr^0$.
\end{thm}
\begin{proof}
    These follow from Lemma \ref{lem-comp-def-H-c}, Theorem \ref{thm-L-!-coh-comp}, and GAGA \cite{Kopf:1974-efava}.
\end{proof}

\begin{thm}\label{thm-comp-alg-dR-int}
    The comparison isomorphisms \Refeq{\ref{eq-comp-alg-dR-cpt}} and \Refeq{\ref{eq-comp-alg-Hdg-cpt}} are compatible with the canonical morphisms induced by any inclusions $I^\Scalt \subset I^\Sc \subset I$, and hence with the comparison isomorphisms in \cite[\aThms \logRHthmintromain{} and \logRHthmHTdegencomp]{Diao/Lan/Liu/Zhu:lrhrv} \Pth{corresponding to $I^\Scalt = \emptyset$; \Refcf{} the notation in Definition \ref{def-dR-Hi-Hdg-coh-cpt}}, via the canonical morphisms among them, and also via Poincar\'e duality.  Consequently, we obtain $\Gal(\AC{k} / k)$-equivariant filtered isomorphisms
    \begin{equation}\label{eq-comp-alg-dR-int}
        H_{\et, \Sc \to \Scalt}^i(U_{\AC{k}}, \bL) \otimes_{\bZ_p} \BdR \cong H_{\dR, \Sc \to \Scalt}^i\bigl(U, \DdRalg(\bL)\bigr) \otimes_k \BdR
    \end{equation}
    and
    \begin{equation}\label{eq-comp-alg-Hdg-int}
        H_{\et, \Sc \to \Scalt}^i(U_{\AC{k}}, \bL) \otimes_{\bZ_p} \widehat{\AC{k}} \cong \oplus_{a + b = i} \, \Bigl(H_{\Hdg, \Sc \to \Scalt}^i\bigl(U, \DdRalg(\bL)\bigr) \otimes_k \widehat{\AC{k}}(-a) \Bigr),
    \end{equation}
    where each generalized interior cohomology is defined as the image of the corresponding cohomology with partial compact support along $D^\Sc$ in the corresponding cohomology with partial compact support along $D^\Scalt$, as before, which are compatible with Poincar\'e duality.  Moreover, the canonical morphism $H_{\dR, \Sc}^i\bigl(U, \DdRalg(\bL)\bigr) \to H_{\dR, \Scalt}^i\bigl(U, \DdRalg(\bL)\bigr)$ is strictly compatible with the Hodge filtrations on both sides, which induces a Hodge filtration on $H_{\dR, \Sc \to \Scalt}^i\bigl(U, \DdRalg(\bL)\bigr)$, together with a canonical graded isomorphism
    \[
        \gr H_{\dR, \Sc \to \Scalt}^i\bigl(U, \DdRalg(\bL)\bigr) \cong \oplus_{a + b = i} \, H_{\Hdg, \Sc \to \Scalt}^i\bigl(U, \DdRalg(\bL)\bigr).
    \]
    When $\emptyset = I^\Scalt \subset I^\Sc = I$, we obtain results for the usual interior cohomology, with subscripts \Qtn{$\intcoh$} replacing \Qtn{$\Sc \to \Scalt$} in all of the above.
\end{thm}
\begin{proof}
    These follow from \cite[\aProp 2.1.4, and \aThms 3.8.1 and 5.7.2]{Huber:1996-ERA}; Lemmas \ref{lem-comp-def-H-c} and \ref{lem-comp-def-H-c-trace}; Theorem \ref{thm-int-coh-comp}; and GAGA \cite{Kopf:1974-efava}.
\end{proof}

\numberwithin{equation}{subsection}

\section{Cohomology of Shimura varieties and Hodge--Tate weights}\label{sec-Sh-var}

In this section, we will freely use the notation and definitions of \cite[\aSec \logRHsecShvar]{Diao/Lan/Liu/Zhu:lrhrv}.  Nevertheless, let us mention the following choices:  We shall fix a Shimura datum $(\Grp{G}, \Shdom)$ and a neat open compact subgroup $\levcp \subset \Grp{G}(\bAi)$, which define a Shimura variety $\Model_\levcp = \Sh_\levcp(\Grp{G}, \Shdom)$ over the reflex field $\ReFl = \ReFl(\Grp{G}, \Shdom)$.  We shall also fix the choices of an algebraic closure $\AC{\bQ}_p$ of $\bQ_p$ and an isomorphism $\ACMap: \AC{\bQ}_p \Mi \bC$.

\subsection{Coherent cohomology and dual BGG decompositions}\label{sec-dual-BGG}

The goal of this subsection is to review the so-called \emph{dual BGG complexes} introduced by Faltings in \cite{Faltings:1983-clshs}, and apply them to the de Rham and Hodge cohomology \Pth{with support conditions} of automorphic vector bundles.  \Pth{The abbreviation BGG refers to Bernstein, Gelfand, and Gelfand, because of their seminal work \cite{Bernstein/Gelfand/Gelfand:1975-doagm}.}  In particular, we shall obtain a refined version of the Hodge--Tate decomposition, whose terms are given by the coherent cohomology of automorphic vector bundles.

Let us fix the choice of some $\hc_\hd$ as in \cite[(\logRHeqmuh)]{Diao/Lan/Liu/Zhu:lrhrv} which is induced by some homomorphism $\Gm{\AC{\bQ}} \to \Grp{G}_{\AC{\bQ}}$, which we abusively denote by the same symbols.  Let $\Grp{P}^c$ denote the parabolic subgroup of $\Grp{G}^c_{\AC{\bQ}}$ defined by the choice of $\hc_\hd$ \Pth{\Refcf{} \cite[\aProp \logRHproplocsystinfty]{Diao/Lan/Liu/Zhu:lrhrv}}.  Let $\Grp{M}^c$ denote the Levi subgroup of $\Grp{P}^c$ given by the centralizer of the image of $\hc_\hd$.  As in the case of $\Grp{G}^c$, for any field $F$ over $\AC{\bQ}$, let us denote by $\Rep_F(\Grp{P}^c)$ \Pth{\resp $\Rep_F(\Grp{M}^c)$} the category of finite-dimensional algebraic representations of $\Grp{P}^c$ \Pth{\resp $\Grp{M}^c$} over $F$.  We shall also view the representations of $\Grp{M}^c$ as representations of $\Grp{P}^c$ by pullback via the canonical homomorphism $\Grp{P}^c \to \Grp{M}^c$.

As explained in \cite[\aSec 3]{Harris:1985-avafs-1} \Pth{or \cite[\aSec 2.2]{Lan:2016-vtcac}}, there is a tensor functor assigning to each $\repalt \in \Rep_{\AC{\bQ}}(\Grp{P}^c)$ a vector bundle $\cohSh{\repalt}_\bC$ on $\Model_{\levcp, \bC}$, which is canonically isomorphic to $\dRSh{\rep}_\bC$ when $\repalt_\bC \cong \rep_\bC$ for some $\rep \in \Rep_{\AC{\bQ}}(\Grp{G}^c)$.  We call $\cohSh{\repalt}_\bC$ the \emph{automorphic vector bundle} associated with $\repalt_\bC$.  Moreover, as explained in \cite[\aSec 4]{Harris:1989-ftcls}, this tensor functor canonically extends to a tensor functor assigning to each $\repalt \in \Rep_{\AC{\bQ}}(\Grp{P}^c)$ a vector bundle $\cohSh{\repalt}^\canext_\bC$ on $\Torcpt{\Model}_{\levcp, \bC}$, called the \emph{canonical extension} of $\cohSh{\repalt}$, which is canonically isomorphic to $\dRSh{\rep}^\canext_\bC$ when $\repalt_\bC \cong \rep_\bC$ for some $\rep \in \Rep_{\AC{\bQ}}(\Grp{G}^c)$.  For $\repalt \in \Rep_{\AC{\bQ}}(\Grp{M}^c)$, this $\cohSh{\repalt}^\canext_\bC$ is canonical isomorphic to the canonical extensions defined as in \cite[Main \aThm 3.1]{Mumford:1977-hptnc}.

Consider $\NCD = \Torcpt{\Model}_{\levcp, \bC} - \Model_{\levcp, \bC}$ \Pth{with its reduced subscheme structure}, which is a normal crossings divisor.  We shall also write $\NCD = \cup_{j \in I} \, \NCD_j$, where $\{ \NCD_j \}_{j \in I}$ are the irreducible components of $\NCD$, so that we can also consider $\NCD^\Scalt \subset \NCD^\Sc \subset \NCD$ for any $I^\Scalt \subset I^\Sc \subset I$.  \Pth{The results below will be for the cohomology with compact support along $D^\Sc$ and for the generalized interior cohomology defined by $I^\Scalt \subset I^\Sc \subset I$, which specialize to results for the cohomology with compact support and the interior cohomology when $I^\Sc = I$ and $I^\Scalt = \emptyset$.}

\begin{defn}\label{def-int-coh-coh}
    For each $\repalt \in \Rep_{\AC{\bQ}}(\Grp{P}^c)$, we define the \emph{subcanonical extension}
    \begin{equation}\label{eq-def-subext}
        \cohSh{\repalt}_\bC^\subext := \cohSh{\repalt}_\bC^\canext(-\NCD)
    \end{equation}
    \Pth{as in \cite[\aSec 2]{Harris:1990-afcvs}} and the \emph{interior coherent cohomology}
    \begin{equation}\label{eq-def-int-coh-coh}
    \begin{split}
        & H_\intcoh^i(\Torcpt{\Model}_{\levcp, \bC}, \cohSh{\repalt}_\bC^\canext) := \Image\bigl(H^i(\Torcpt{\Model}_{\levcp, \bC}, \cohSh{\repalt}_\bC^\subext) \to H^i(\Torcpt{\Model}_{\levcp, \bC}, \cohSh{\repalt}_\bC^\canext)\bigr).
    \end{split}
    \end{equation}
    More generally, for any $I^\Scalt \subset I^\Sc \subset I$, by abuse of notation, we define
    \begin{equation}\label{eq-def-coh-coh-gen}
        H_\Sc^i(\Torcpt{\Model}_{\levcp, \bC}, \cohSh{\repalt}_\bC^\canext) := H^i\bigl(\Torcpt{\Model}_{\levcp, \bC}, \cohSh{\repalt}_\bC^\canext(-\NCD^\Sc)\bigr)
    \end{equation}
    and
    \begin{equation}\label{eq-def-int-coh-coh-gen}
    \begin{split}
        & H_{\Sc \to \Scalt}^i(\Torcpt{\Model}_{\levcp, \bC}, \cohSh{\repalt}_\bC^\canext) \\
        & := \Image\bigl(H_\Sc^i(\Torcpt{\Model}_{\levcp, \bC}, \cohSh{\repalt}_\bC^\canext) \to H_\Scalt^i(\Torcpt{\Model}_{\levcp, \bC}, \cohSh{\repalt}_\bC^\canext)\bigr),
    \end{split}
    \end{equation}
    the latter of which gives \Refeq{\ref{eq-def-int-coh-coh}} when $I^\Sc = I$ and $I^\Scalt = \emptyset$.  We similarly define objects with $\bC$ replaced with $\AC{\bQ}_p$ or any field extension of $\ReFl$ over which $\repalt$ has a model, or with $\repalt$ replaced with $\dual{\repalt}$, or both.
\end{defn}

Let us fix the choice of a maximal torus $\Grp{T}^c$ of $\Grp{M}^c$, which is also a maximal torus of $\Grp{G}^c_{\AC{\bQ}}$.  With this choice, let us denote by $\RT_{\Grp{G}^c_{\AC{\bQ}}}$, $\RT_{\Grp{M}^c}$, \etc the roots of $\Grp{G}^c_{\AC{\bQ}}$, $\Grp{M}^c$, \etc, respectively; and by $\WT_{\Grp{G}^c_{\AC{\bQ}}}$, $\WT_{\Grp{M}^c}$, \etc the weights of $\Grp{G}^c_{\AC{\bQ}}$, $\Grp{M}^c$, \etc, respectively.  Then we have naturally $\RT_{\Grp{M}^c} \subset \RT_{\Grp{G}^c_{\AC{\bQ}}}$ and $\WT_{\Grp{G}^c_{\AC{\bQ}}} = \WT_{\Grp{M}^c}$.  Let us denote by $H$ the coweight induced by $\hc_\hd$.  Let us also make compatible choices of positive roots $\RT_{\Grp{G}^c_{\AC{\bQ}}}^+$ and $\RT_{\Grp{M}^c}^+$, and of dominant weights $\WT_{\Grp{G}^c_{\AC{\bQ}}}^+$ and $\WT_{\Grp{M}^c}^+$, so that $\RT_{\Grp{M}^c}^+ \subset \RT_{\Grp{G}^c_{\AC{\bQ}}}^+$ and $\WT_{\Grp{G}^c_{\AC{\bQ}}}^+ \subset \WT_{\Grp{M}^c}^+$.  For an irreducible representation $\rep$ of highest weight $\wt \in \WT_{\Grp{G}^c_{\AC{\bQ}}}^+$, we write $\rep = \rep_\wt$, $\rep_\bC = \rep_{\wt, \bC}$, $\dRSh{\rep}_\bC = \dRSh{\rep}_{\wtalt, \bC}$, \etc.  Similarly, for an irreducible representation $\repalt$ of highest weight $\wtalt \in \WT_\Grp{M}^+$, we write $\repalt = \repalt_\wtalt$, $\repalt_\bC = \repalt_{\wtalt, \bC}$, $\cohSh{\repalt}_\bC = \cohSh{\repalt}_{\wtalt, \bC}$, \etc.  Let $\hsum_{\Grp{G}^c_{\AC{\bQ}}} := \frac{1}{2} \sum_{\wt \in \RT_{\Grp{G}^c_{\AC{\bQ}}}^+} \wt$ and $\hsum_{\Grp{M}^c} := \frac{1}{2} \sum_{\wtalt \in \RT_{\Grp{M}^c}^+} \wtalt$ denote the usual half-sums of positive roots, and let $\hsum^{\Grp{M}^c} := \hsum_{\Grp{G}^c_{\AC{\bQ}}} - \hsum_{\Grp{M}^c}$.  Let $\WG_{\Grp{G}^c_{\AC{\bQ}}}$ and $\WG_{\Grp{M}^c}$ denote the Weyl groups of $\Grp{G}^c_{\AC{\bQ}}$ and $\Grp{M}^c$ with respect to the common maximal torus $\Grp{T}^c$.  Then we can naturally identify $\WG_{\Grp{M}^c}$ as a subgroup of $\WG_{\Grp{G}^c_{\AC{\bQ}}}$.  Given any element $w$ in the above Weyl groups, we shall denote its length by $\wl(w)$.  In addition to the natural action of $\WG_{\Grp{G}^c_{\AC{\bQ}}}$ on $\WT_{\Grp{G}^c_{\AC{\bQ}}}$, there is also the dot action $w \cdot \wt = w( \wt + \hsum_{\Grp{G}^c_{\AC{\bQ}}} ) - \hsum_{\Grp{G}^c_{\AC{\bQ}}}$, for all $w \in \WG_{\Grp{G}^c_{\AC{\bQ}}}$ and $\wt \in \WT_{\Grp{G}^c_{\AC{\bQ}}}$.  Let $\WG^{\Grp{M}^c}$ denote the subset of $\WG_{\Grp{G}^c_{\AC{\bQ}}}$ consisting of elements mapping $\WT_{\Grp{G}^c_{\AC{\bQ}}}^+$ into $\WT_{\Grp{M}^c}^+$, which are the minimal length representatives of $\WG_{\Grp{M}^c} \Lquot \WG_{\Grp{G}^c_{\AC{\bQ}}}$.

As in Definition \ref{def-dR-Hi-Hdg-coh-cpt} and Theorem \ref{thm-comp-alg-dR-cpt}, consider the log de Rham complex,
\[
\begin{split}
    \DRl\bigl(\dRSh{\rep}^\canext_\bC(-\NCD^\Sc)\bigr) & := \bigl(\dRSh{\rep}^\canext_\bC(-\NCD^\Sc) \otimes_{\cO_{\Torcpt{\Model}_{\levcp, \bC}}} \Omega^\bullet_{\Torcpt{\Model}_{\levcp, \bC}}(\log \NCD_\bC), \nabla\bigr) \\
    & \cong \DRl(\dRSh{\rep}^\canext_\bC) \otimes_{\cO_{\Torcpt{\Model}_{\levcp, \bC}}} \cO_{\Torcpt{\Model}_{\levcp, \bC}}(-\NCD^\Sc),
\end{split}
\]
and consider the Hodge cohomology with partial compact support along $\NCD^\Sc$,
\begin{equation}\label{eq-def-aut-bdl-Hodge-coh-cpt}
    H^{a, b}_{\Hdg, \Sc}(\Model_{\levcp, \bC}, \dRSh{\rep}_\bC) := H^{a + b}\bigl(\Torcpt{\Model}_{\levcp, \bC}, \gr^a \DRl\bigl(\dRSh{\rep}^\canext_\bC(-\NCD^\Sc)\bigr)\bigr).
\end{equation}
As in Remark \ref{rem-def-dR-Hi-Hdg-coh-cpt-conv}, when $I^\Sc = \emptyset$, this give the usual Hodge cohomology; and when $I^\Sc = I$, this gives the usual Hodge cohomology with compact support.

While it is difficult to compute hypercohomology in general, the miracle is that $\gr^a \DRl(\dRSh{\rep}^\canext_\bC)$ has a quasi-isomorphic direct summand, called the \emph{graded dual BGG complex}, whose differentials are \emph{zero} and whose terms are direct sums of $\cohSh{\repalt}^\canext_\bC$ for some representations $\repalt$ determined explicitly by $\rep$.  Then the hypercohomology of this graded dual BGG complex is just a direct sum of usual coherent cohomology of $\cohSh{\repalt}^\canext_\bC$ up to degree shifting.  More precisely, we have the following:
\begin{thm}[\emph{dual BGG complexes}; Faltings]\label{thm-dual-BGG}
    There is a canonical filtered quasi-isomorphic direct summand $\BGGl(\dRSh{\rep}^\canext_\bC)$ of $\DRl(\dRSh{\rep}^\canext_\bC)$ \Pth{in the category of complexes of abelian sheaves on $\Torcpt{\Model}_{\levcp, \bC}$ whose terms are coherent sheaves and whose differentials are differential operators} satisfying the following properties:
    \begin{enumerate}
        \item The formation of $\BGGl(\dRSh{\rep}^\canext_\bC)$ is functorial and exact in $\rep_\bC$.

        \item The differentials on $\gr \BGGl(\dRSh{\rep}^\canext_\bC)$ are all zero.

        \item Suppose that $\rep \cong \dual{\rep}_\wt$ for some $\wt \in \WT_{\Grp{G}^c_{\AC{\bQ}}}^+$.  Then, for each $i\geq 0$ and each $a \in \bZ$, the $i$-th term of $\gr^a \BGGl(\dRSh{\rep}^\canext_\bC)$ is given by
            \begin{equation}\label{eq-thm-dual-BGG}
                \gr^a \BGGl^i(\dRSh{\rep}^\canext_\bC) \cong \oplus_{w \in \WG^{\Grp{M}^c}, \; \wl(w) = i, \; (w \cdot \wt)(H) = -a} \; \bigl( \cohSh{\repalt}_{w \cdot \wt, \bC}^\dualsign \bigr)^\canext
            \end{equation}
    \end{enumerate}
\end{thm}
\begin{proof}
    See \cite[\aSecs 3 and 7]{Faltings:1983-clshs}, \cite[\aCh VI, \aSec 5]{Faltings/Chai:1990-DAV}, and \cite[\aThm 5.9]{Lan/Polo:2018-bggab}.  \Pth{Although these references were written in less general settings, the methods of the constructions still generalize to our setting here.}
\end{proof}

\begin{rk}\label{rem-dual-BGG-descent}
    The various automorphic vector bundles $\dRSh{\rep}_{\wt, \bC}$, $\Fil^\bullet(\dRSh{\rep}_{\wt, \bC})$, and $\cohSh{\repalt}_{w \cdot \wt, \bC}$ \Pth{and their canonical extensions} in Theorem \ref{thm-dual-BGG} have models over $\AC{\bQ}$, or even over a finite extension $\ReFl'$ of $\ReFl$ \Pth{depending on $\wt$} over which $\rep_\wt$, $\Fil^\bullet \, \rep_\wt$, and $\repalt_{w \cdot \wt}$ all have models.  \Pth{The cases of $\dRSh{\rep}_{\wt, \bC}$ and its canonical extension follow from \cite[\aCh III, \aThm 5.1, and \aCh V, \aThm 6.2]{Milne:1990-cmsab}; and the cases of $\Fil^\bullet \, \dRSh{\rep}_{\wt, \bC}$, $\cohSh{\repalt}_{w \cdot \wt, \bC}$, and their canonical extensions follow from the same argument as in \cite[\aRem \logRHrempartialflag]{Diao/Lan/Liu/Zhu:lrhrv}, based on the models of the associated partial flag varieties over $\ReFl$.  Note that $\Fil^\bullet(\dRSh{\rep}^\canext_{\wt, \bC})$ did appear in Theorem \ref{thm-dual-BGG}, when we said there is a canonical filtered quasi-isomorphic direct summand $\BGGl(\dRSh{\rep}^\canext_\bC)$ of $\DRl(\dRSh{\rep}^\canext_\bC)$.}  Then the statements of Theorem \ref{thm-dual-BGG} remain true if we replace $\bC$ with $\AC{\bQ}$ or $\ReFl'$, by the same descent argument as in \cite[\aSec 6]{Lan/Polo:2018-bggab}.
\end{rk}

\begin{thm}\label{thm-dual-BGG-cpt-int}
    Suppose that $\rep \cong \dual{\rep}_\wt$ for some $\wt \in \WT_{\Grp{G}^c_{\AC{\bQ}}}^+$.  Then, given any $a, b \in \bZ$ such that $a + b \geq 0$, we have the \emph{dual BGG decomposition}
    \begin{equation}\label{eq-thm-dual-BGG-cpt}
    \begin{split}
        & H^{a, b}_{\Hdg, \Sc}(\Model_{\levcp, \bC}, \dRSh{\rep}_\bC) = H^{a + b}\bigl(\Torcpt{\Model}_{\levcp, \bC}, \gr^a \DRl\bigl(\dRSh{\rep}^\canext_\bC(-\NCD^\Sc)\bigr)\bigr) \\
        & \cong \oplus_{w \in \WG^{\Grp{M}^c}, \; (w \cdot \wt)(H) = -a} \; H_\Sc^{a + b - \wl(w)}\bigl(\Torcpt{\Model}_{\levcp, \bC}, (\cohSh{\repalt}_{w \cdot \wt, \bC}^\dualsign)^\canext\bigr),
    \end{split}
    \end{equation}
    which is compatible with the canonical morphisms induced by any inclusions $I^\Scalt \subset I^\Sc \subset I$ and induces a similar dual BGG decomposition
    \begin{equation}\label{eq-thm-dual-BGG-int}
    \begin{split}
        & H^{a, b}_{\Hdg, \Sc \to \Scalt}(\Model_{\levcp, \bC}, \dRSh{\rep}_\bC) \\
        & \cong \oplus_{w \in \WG^{\Grp{M}^c}, \; (w \cdot \wt)(H) = -a} \; H_{\Sc \to \Scalt}^{a + b - \wl(w)}\bigl(\Torcpt{\Model}_{\levcp, \bC}, (\cohSh{\repalt}_{w \cdot \wt, \bC}^\dualsign)^\canext\bigr).
    \end{split}
    \end{equation}
    Moreover, the Hodge--de Rham spectral sequence for $H^{a, b}_{\Hdg, \Sc}(\Model_{\levcp, \bC}, \dRSh{\rep}_\bC)$ induces the \emph{dual BGG spectral sequence}
    \begin{equation}\label{eq-thm-dual-BGG-cpt-spec-seq}
    \begin{split}
        & E_1^{a, b} = \oplus_{w \in \WG^{\Grp{M}^c}, \, (w \cdot \wt)(H) = -a} \; H_\Sc^{a + b - \wl(w)}\bigl(\Torcpt{\Model}_{\levcp, \bC}, (\cohSh{\repalt}_{w \cdot \wt, \bC}^\dualsign)^\canext\bigr) \\
        & \Rightarrow H^{a + b}_{\dR, \Sc}(\Model_{\levcp, \bC}, \dRSh{\rep}_\bC),
    \end{split}
    \end{equation}
    which degenerates on the $E_1$ page, and induces a \emph{dual BGG decomposition}
    \begin{equation}\label{eq-thm-dual-BGG-spec-decomp-cpt}
        \gr H^i_{\dR, \Sc}(\Model_{\levcp, \bC}, \dRSh{\rep}_\bC)
        \cong \oplus_{w \in \WG^{\Grp{M}^c}} \; H_\Sc^{i - \wl(w)}\bigl(\Torcpt{\Model}_{\levcp, \bC}, (\cohSh{\repalt}_{w \cdot \wt, \bC}^\dualsign)^\canext\bigr),
    \end{equation}
    for each $i \geq 0$, which is \Pth{strictly} compatible with the analogous decomposition for $I^\Scalt \subset I^\Sc \subset I$ and therefore induces a dual BGG decomposition
    \begin{equation}\label{eq-thm-dual-BGG-spec-decomp-int}
    \begin{split}
        & \gr H^i_{\dR, \Sc \to \Scalt}(\Model_{\levcp, \bC}, \dRSh{\rep}_\bC) \\
        & \cong \oplus_{w \in \WG^{\Grp{M}^c}} \; H_{\Sc \to \Scalt}^{i - \wl(w)}\bigl(\Torcpt{\Model}_{\levcp, \bC}, (\cohSh{\repalt}_{w \cdot \wt, \bC}^\dualsign)^\canext\bigr).
    \end{split}
    \end{equation}
\end{thm}
\begin{proof}
    By tensoring the graded quasi-isomorphism between the log de Rham complex $\DRl(\dRSh{\rep}^\canext_\bC)$ and the \Pth{log} dual BGG complex $\BGGl(\dRSh{\rep}^\canext_\bC)$ in Theorem \ref{thm-dual-BGG} with the invertible $\cO_{\Torcpt{\Model}_{\levcp, \bC}}$-ideal $\cO_{\Torcpt{\Model}_{\levcp, \bC}}(-\NCD^\Sc)$, and by taking graded pieces, we obtain a quasi-isomorphism between $\gr^a \DRl\bigl(\dRSh{\rep}^\canext_\bC(-\NCD^\Sc)\bigr)$ and $\bigl(\gr^a \BGGl(\dRSh{\rep}^\canext_\bC)\bigr)(-\NCD^\Sc)$, and the differentials of the last complex are still all zero.  Hence, by comparing the \Pth{algebraic} objects over $\bC$ with their analogues over $\AC{\bQ}_p$ as in \cite[\aRem \logRHremlocsystHodgedegen]{Diao/Lan/Liu/Zhu:lrhrv}, the desired isomorphism \Refeq{\ref{eq-thm-dual-BGG-cpt}} follows from \Refeq{\ref{eq-def-alg-Hdg-coh-cpt}}, and the remaining assertions follow from Theorems \ref{thm-comp-alg-dR-cpt} and \ref{thm-comp-alg-dR-int}.
\end{proof}

\begin{rk}\label{rem-dual-BGG-spec-seq}
    Faltings first introduced the dual BGG spectral sequence associated with the \emph{stupid \Pth{\Qtn{b\^ete}} filtration} in \cite[\aSec 4, \apage 76, and \aSec 7, \aThm 11]{Faltings:1983-clshs}, whose degeneration on the $E_1$ page nevertheless implies \Pth{by comparison of total dimensions} the degeneration of the spectral sequence \Refeq{\ref{eq-thm-dual-BGG-cpt-spec-seq}} associated with the \emph{Hodge filtration}.  The degeneracy on the $E_1$ page was first proved by Faltings himself in the compact case \Pth{see \cite[\aSec 4, \aThm 4]{Faltings:1983-clshs}}, later in the case of Siegel modular varieties by Faltings and Chai \Pth{see \cite[\aCh VI, \aThm 5.5]{Faltings/Chai:1990-DAV}} by reducing to the case of trivial coefficients over some toroidal compactifications of self-fiber products of universal abelian schemes \Pth{and this method can be generalized to the case of all PEL-type and Hodge-type Shimura varieties using \cite{Lan:2012-aatcs}, \cite{Lan:2012-tckf}, and \cite[\aSec 4.5]{Lan/Stroh:2018-ncaes}}, and in general by Harris and Zucker \Pth{see \cite[\aCor 4.2.3]{Harris/Zucker:2001-BCS-3}} using Saito's theory of mixed Hodge modules \Pth{see \cite[\aThm 2.14]{Saito:1990-mhm}}.  Even when $I^\Sc = \emptyset$, our proof of the degeneration of the dual BGG spectral sequence \ref{eq-thm-dual-BGG-cpt-spec-seq}, which can be alternatively based on \cite[\aThms \logRHthmHTdegencomp{} and \logRHthmlocsystcomp]{Diao/Lan/Liu/Zhu:lrhrv}, is a new one.
\end{rk}

\subsection{Hodge--Tate weights}\label{sec-Sh-var-HT-wts}

The goal of this subsection is to describe the Hodge--Tate weights of $H^i_{\et, \Sc}(\Model_{\levcp, \AC{\bQ}_p}, \etSh{\rep}_{\AC{\bQ}_p})$ and $H^i_{\et, \Sc \to \Scalt}(\Model_{\levcp, \AC{\bQ}_p}, \etSh{\rep}_{\AC{\bQ}_p})$ in terms of the dimensions of the dual BGG pieces at the right-hand sides of \Refeq{\ref{eq-thm-dual-BGG-spec-decomp-cpt}} and \Refeq{\ref{eq-thm-dual-BGG-spec-decomp-int}}, respectively.

We first need to provide a definition for the Hodge--Tate weights of the cohomology of \'etale local systems over the infinite extension $\AC{\bQ}_p$ of $\bQ_p$.  Let $\bC_p$ denote the $p$-adic completion of $\AC{\bQ}_p$ as usual.  Let $(\dRSh{\rep}_{\AC{\bQ}_p}, \nabla, \Fil^\bullet)$ denote the pullback of $(\dRSh{\rep}_\bC, \nabla, \Fil^\bullet)$ under $\ACMap^{-1}: \bC \Mi \AC{\bQ}_p$.  As in \cite[\aSec \logRHseclocsystconstr]{Diao/Lan/Liu/Zhu:lrhrv}, let $\Coef$ be a finite extension of $\bQ_p$ in $\AC{\bQ}_p$ such that $\rep_{\AC{\bQ}_p}$ has a model $\rep_\Coef$ over $\Coef$, and let $\etSh{\rep}_\Coef$ be as in \cite[(\logRHeqlocsystetcoefbasech)]{Diao/Lan/Liu/Zhu:lrhrv}.  Let $\BFp$ be a finite extension of the composite of $\ReFl$ and $\Coef$ in $\AC{\bQ}_p$, so that we have $\ReFl \Emn{\can} \AC{\bQ} \Emn{~~\ACMap^{-1}} \AC{\bQ}_p$, and let $\CoefMap: \Coef \otimes_{\bQ_p} \BFp \to \BFp$ be the multiplication homomorphism $a \otimes b \mapsto ab$, as in \cite[(\logRHeqcoefproj)]{Diao/Lan/Liu/Zhu:lrhrv}.  By Theorem \ref{thm-comp-alg-dR-cpt}, $H^i_{\et, \Sc}(\Model_{\levcp, \AC{\bQ}_p}, \etSh{\rep}_\Coef)$ is a de Rham representations of $\Gal(\AC{\bQ}_p / \BFp)$, and we have a canonical $\Gal(\AC{\bQ}_p / \BFp)$-equivariant Hecke-equivariant isomorphism
\begin{equation}\label{eq-Sh-var-arith-log-RH-comp-dR-cpt}
    H^i_{\et, \Sc}(\Model_{\levcp, \AC{\bQ}_p}, \etSh{\rep}_\Coef) \otimes_{\bQ_p} \BdR \cong H^i_{\dR, \Sc}\bigl(\Model_{\levcp, \BFp}, \DdRalg(\etSh{\rep}_\Coef)\bigr) \otimes_\BFp \BdR,
\end{equation}
which is compatible with the filtrations on both sides, whose $0$-th graded piece is a canonical $\Gal(\AC{\bQ}_p / \BFp)$-equivariant Hecke-equivariant isomorphism
\begin{equation}\label{eq-Sh-var-arith-log-RH-comp-HT-cpt}
\begin{split}
    & H^i_{\et, \Sc}(\Model_{\levcp, \AC{\bQ}_p}, \etSh{\rep}_\Coef) \otimes_{\bQ_p} \bC_p \\
    & \cong \oplus_{a + b = i} \, \Bigl(H^{a, b}_{\Hdg, \Sc}\bigl(\Model_{\levcp, \AC{\bQ}_p}, \DdRalg(\etSh{\rep}_\Coef)\bigr) \otimes_k \bC_p(-a)\Bigr).
\end{split}
\end{equation}
By pushing out \Refeq{\ref{eq-Sh-var-arith-log-RH-comp-dR-cpt}} and \Refeq{\ref{eq-Sh-var-arith-log-RH-comp-HT-cpt}} via the homomorphism $\CoefMap$, by \cite[\aThm \logRHthmlocsystcomp]{Diao/Lan/Liu/Zhu:lrhrv}, and by Theorem \ref{thm-comp-alg-dR-int}, we obtain the following:
\begin{thm}\label{thm-Sh-var-comp-dR-HT}
    Suppose that $\rep \in \Rep_{\AC{\bQ}}(\Grp{G}^c)$, and that $\BFp$ is a finite extension of the image of $\ReFl \Emn{\can} \AC{\bQ} \Emn{~~\ACMap^{-1}} \AC{\bQ}_p$ over which $\rep_{\AC{\bQ}_p}$ has a model.  Then there is a canonical $\Gal(\AC{\bQ}_p / \BFp)$-equivariant Hecke-equivariant \emph{de Rham comparison} isomorphism
    \begin{equation}\label{eq-Sh-var-comp-dR-cpt}
        H_{\et, \Sc}^i(\Model_{\levcp, \AC{\bQ}_p}, \etSh{\rep}_{\AC{\bQ}_p}) \otimes_{\AC{\bQ}_p} \BdR \cong H_{\dR, \Sc}^i(\Model_{\levcp, \AC{\bQ}_p}, \dRSh{\rep}_{\AC{\bQ}_p}) \otimes_{\AC{\bQ}_p} \BdR,
    \end{equation}
    compatible with the filtrations on both sides, whose $0$-th graded piece is a canonical $\Gal(\AC{\bQ}_p / \BFp)$-equivariant Hecke-equivariant \emph{Hodge--Tate comparison} isomorphism
    \begin{equation}\label{eq-Sh-var-comp-HT-cpt}
    \begin{split}
        & H_{\et, \Sc}^i(\Model_{\levcp, \AC{\bQ}_p}, \etSh{\rep}_{\AC{\bQ}_p}) \otimes_{\AC{\bQ}_p} \bC_p \\
        & \cong \oplus_{a + b = i} \, \bigl(H^{a, b}_{\Hdg, \Sc}(\Model_{\levcp, \AC{\bQ}_p}, \dRSh{\rep}_{\AC{\bQ}_p}) \otimes_{\AC{\bQ}_p} \bC_p(-a)\bigr).
    \end{split}
    \end{equation}
    Moreover, \Refeq{\ref{eq-Sh-var-comp-dR-cpt}} and \Refeq{\ref{eq-Sh-var-comp-HT-cpt}} are compatible with the canonical morphisms defined by inclusions $I^\Scalt \subset I^\Sc \subset I$, and also with Poincar\'e and Serre duality, which induce a canonical $\Gal(\AC{\bQ}_p / \BFp)$-equivariant Hecke-equivariant \emph{de Rham comparison} isomorphism
    \begin{equation}\label{eq-Sh-var-comp-dR-int}
    \begin{split}
        & H_{\et, \Sc \to \Scalt}^i(\Model_{\levcp, \AC{\bQ}_p}, \etSh{\rep}_{\AC{\bQ}_p}) \otimes_{\AC{\bQ}_p} \BdR \\
        & \cong H_{\dR, \Sc \to \Scalt}^i(\Model_{\levcp, \AC{\bQ}_p}, \dRSh{\rep}_{\AC{\bQ}_p}) \otimes_{\AC{\bQ}_p} \BdR,
    \end{split}
    \end{equation}
    which is compatible with the filtrations on both sides and with Poincar\'e duality, whose $0$-th graded piece is a canonical $\Gal(\AC{\bQ}_p / \BFp)$-equivariant Hecke-equivariant \emph{Hodge--Tate comparison} isomorphism
    \begin{equation}\label{eq-Sh-var-comp-HT-int}
    \begin{split}
        & H_{\et, \Sc \to \Scalt}^i(\Model_{\levcp, \AC{\bQ}_p}, \etSh{\rep}_{\AC{\bQ}_p}) \otimes_{\AC{\bQ}_p} \bC_p \\
        & \cong \oplus_{a + b = i} \, \bigl(H^{a, b}_{\Hdg, \Sc \to \Scalt}(\Model_{\levcp, \AC{\bQ}_p}, \dRSh{\rep}_{\AC{\bQ}_p}) \otimes_{\AC{\bQ}_p} \bC_p(-a)\bigr),
    \end{split}
    \end{equation}
    which is compatible with Poincar\'e and Serre duality.
\end{thm}

\begin{defn}\label{def-Sh-var-HT-wts}
    For $? = \Sc$ or $\Sc \to \Scalt$, we abusively define the multiset of \emph{Hodge--Tate weights} of $H_{\et, ?}^i(\Model_{\levcp, \AC{\bQ}_p}, \etSh{\rep}_{\AC{\bQ}_p})$ to be the multiset of integers in which each $a \in \bZ$ has multiplicity $\dim_{\AC{\bQ}_p}\bigl(H_{\Hdg, ?}^{a, i - a}(\Model_{\levcp, \AC{\bQ}_p}, \dRSh{\rep}_{\AC{\bQ}_p})\bigr)$.  We naturally extend the definition to $\AC{\bQ}_p$-subspaces of $H_{\et, ?}^i(\Model_{\levcp, \AC{\bQ}_p}, \etSh{\rep}_{\AC{\bQ}_p})$ cut out by $\AC{\bQ}_p$-valued Hecke operators by replacing $H_{\Hdg, ?}^{a, i - a}(\Model_{\levcp, \AC{\bQ}_p}, \dRSh{\rep}_{\AC{\bQ}_p})$ with their corresponding $\AC{\bQ}_p$-subspaces cut out by the same $\AC{\bQ}_p$-valued Hecke operators.
\end{defn}

\begin{thm}\label{thm-HT-wts}
    With the same setting as in Theorem \ref{thm-Sh-var-comp-dR-HT}, suppose $\rep \cong \dual{\rep}_\wt$ for some $\wt \in \WT_{\Grp{G}^c_{\AC{\bQ}}}^+$.  For any $\repalt$ in $\Rep_{\AC{\bQ}}(\Grp{M}^c)$, let $\cohSh{\repalt}^\canext_{\AC{\bQ}_p}$ be the pullback of $\cohSh{\repalt}^\canext_\bC$ under $\ACMap^{-1}: \bC \Mi \AC{\bQ}_p$.  Then we have a canonical $\Gal(\AC{\bQ}_p / \BFp)$-equivariant Hecke-equivariant isomorphism
    \begin{equation}\label{eq-Sh-var-HT-decomp-dual-BGG-cpt}
    \begin{split}
        & H_{\et, \Sc}^i(\Model_{\levcp, \AC{\bQ}_p}, \etSh{\rep}_{\AC{\bQ}_p}) \otimes_{\AC{\bQ}_p} \bC_p \\
        & \cong \oplus_{w\in \WG^{\Grp{M}^c}} \; \Bigl(H_\Sc^{i - \wl(w)}\bigl(\Torcpt{\Model}_{\levcp, \AC{\bQ}_p}, (\cohSh{\repalt}_{w \cdot \wt, \AC{\bQ}_p}^\dualsign)^\canext\bigr) \otimes_{\AC{\bQ}_p} \bC_p\bigl((w \cdot \wt)(H)\bigr)\Bigr),
    \end{split}
    \end{equation}
    which is the dual BGG version of the Hodge--Tate decomposition \Refeq{\ref{eq-Sh-var-comp-HT-cpt}}.  This isomorphism \Refeq{\ref{eq-Sh-var-HT-decomp-dual-BGG-cpt}} is compatible with the canonical morphisms defined by any inclusions $I^\Scalt \subset I^\Sc \subset I$, and with Poincar\'e and Serre duality; and induces a canonical $\Gal(\AC{\bQ}_p / \BFp)$-equivariant Hecke-equivariant isomorphism
    \begin{equation}\label{eq-Sh-var-HT-decomp-dual-BGG-int}
    \begin{split}
        & H_{\et, \Sc \to \Scalt}^i(\Model_{\levcp, \AC{\bQ}_p}, \etSh{\rep}_{\AC{\bQ}_p}) \otimes_{\AC{\bQ}_p} \bC_p \\
        & \cong \oplus_{w\in \WG^{\Grp{M}^c}} \; \Bigl(H_{\Sc \to \Scalt}^{i - \wl(w)}\bigl(\Torcpt{\Model}_{\levcp, \AC{\bQ}_p}, (\cohSh{\repalt}_{w \cdot \wt, \AC{\bQ}_p}^\dualsign)^\canext\bigr) \otimes_{\AC{\bQ}_p} \bC_p\bigl((w \cdot \wt)(H)\bigr)\Bigr),
    \end{split}
    \end{equation}
    compatible with Poincar\'e and Serre duality.  The multiset of Hodge--Tate weights of any Hecke-invariant $\AC{\bQ}_p$-subspace of $H_{\et, \Sc}^i(\Model_{\levcp, \AC{\bQ}_p}, \etSh{\rep}_{\AC{\bQ}_p})$ cut out by some $\AC{\bQ}_p$-valued Hecke operator \Pth{as in Definition \ref{def-Sh-var-HT-wts}} contains each $a \in \bZ$ with multiplicity given by the $\bC$-dimension of the corresponding Hecke-invariant $\bC$-subspace of
    \begin{equation}\label{eq-Sh-var-HT-decomp-dual-BGG-wt}
        \oplus_{w\in \WG^{\Grp{M}^c}, \; (w \cdot \wt)(H) = -a} \; H_\Sc^{i - \wl(w)}\bigl(\Torcpt{\Model}_{\levcp, \bC}, (\cohSh{\repalt}_{w \cdot \wt, \bC}^\dualsign)^\canext\bigr)
    \end{equation}
    cut out by the pullback of the same $\AC{\bQ}_p$-valued Hecke operator under $\ACMap: \AC{\bQ}_p \Mi \bC$.  The same holds with $H_{\et, \Sc}^i(\Model_{\levcp, \AC{\bQ}_p}, \etSh{\rep}_{\AC{\bQ}_p})$ \Pth{\resp $H_\Sc^{i - \wl(w)}\bigl(\Torcpt{\Model}_{\levcp, \bC}, (\cohSh{\repalt}_{w \cdot \wt, \bC}^\dualsign)^\canext\bigr)$} replaced with $H_{\et, \Sc \to \Scalt}^i(\Model_{\levcp, \AC{\bQ}_p}, \etSh{\rep}_{\AC{\bQ}_p})$ \Pth{\resp $H_{\Sc \to \Scalt}^{i - \wl(w)}\bigl(\Torcpt{\Model}_{\levcp, \bC}, (\cohSh{\repalt}_{w \cdot \wt, \bC}^\dualsign)^\canext\bigr)$}.
\end{thm}
\begin{proof}
    These follow from Theorems \ref{thm-Sh-var-comp-dR-HT} and \ref{thm-dual-BGG-cpt-int}, and from the fact \Pth{which we have implicitly used several times} that the formation of coherent hypercohomology of qcqs schemes is compatible with arbitrary base field extensions.
\end{proof}

\begin{rk}\label{rem-HT-wts-Siegel}
    All previously known special cases of \Refeq{\ref{eq-Sh-var-HT-decomp-dual-BGG-cpt}} \Pth{see, for example, \cite[\aThm 6.2]{Faltings/Chai:1990-DAV} and \cite[\aSec III.2]{Harris/Taylor:2001-GCS}} were proved using the Hodge--Tate comparison for the cohomology with trivial coefficients of some families of abelian varieties \Pth{and their smooth toroidal compactifications, in the noncompact case}.  The novelty in Theorem \ref{thm-HT-wts} is that we can deal with nontrivial coefficients that are not at all related to families of abelian varieties.
\end{rk}

\begin{rk}\label{rem-HT-wts-compute}
    As in \cite[\aEx 4.6]{Harris:1990-afcvs}, we can often compute the dimension of $H_\intcoh^{i - \wl(w)}\bigl(\Torcpt{\Model}_{\levcp, \bC}, (\cohSh{\repalt}_{w \cdot \wt, \bC}^\dualsign)^\canext\bigr)$ and its Hecke-invariant $\bC$-subspaces cut out by $\bC$-valued Hecke operators in terms of relative Lie algebra cohomology.  Thanks to the recent work \cite{Su:2018-ccsva}, it might also be possible to compute the dimension of $H_\Sc^{i - \wl(w)}\bigl(\Torcpt{\Model}_{\levcp, \bC}, (\cohSh{\repalt}_{w \cdot \wt, \bC}^\dualsign)^\canext\bigr)$ and its Hecke-invariant $\bC$-subspaces cut out by $\bC$-valued Hecke operators in terms of relative Lie algebra cohomology when the image of $\NCD^\Sc$ in the minimal compactification $\Mincpt{\Model}_{\levcp, \bC}$ of $\Model_{\levcp, \bC}$ \Pth{as in \cite{Pink:1989-Ph-D-Thesis}} is stable under the Hecke action of $\Grp{G}(\bAi)$.
\end{rk}

\begin{rk}\label{ex-HT-wts-indep}
    In the special \Pth{but still common} case where $\rep_{\AC{\bQ}_p}$ has a model $\rep_{\bQ_p}$ over $\bQ_p$, we can take $\Coef = \bQ_p$ in the above, and the choice of $\ACMap: \AC{\bQ}_p \Mi \bC$ corresponds to the choice of places $v$ of $\ReFl$ above $p$.  Then $H^i_{\et, ?}(\Model_{\levcp, \AC{\bQ}_p}, \etSh{\rep}_{\bQ_p})$ is a \emph{de Rham} representation of $\Gal(\AC{\bQ}_p / \BFp)$, for $? = \Sc$ or $\Sc \to \Scalt$, and the de Rham comparison isomorphisms \Refeq{\ref{eq-Sh-var-comp-dR-cpt}} and \Refeq{\ref{eq-Sh-var-comp-dR-int}} can be rewritten as
    \[
        H^i_{\et, ?}(\Model_{\levcp, \AC{\bQ}_p}, \etSh{\rep}_{\bQ_p}) \otimes_{\bQ_p} \BdR \cong H^i_{\dR, ?}(\Model_{\levcp, \BFp}, \dRSh{\rep}_\BFp) \otimes_\BFp \BdR
    \]
    \Pth{\Refcf{} \Refeq{\ref{eq-Sh-var-arith-log-RH-comp-dR-cpt}}}.  Moreover, the assertion in Theorem \ref{thm-HT-wts} that the Hodge--Tate weights of $H^i_{\et, ?}(\Model_{\levcp, \AC{\bQ}_p}, \etSh{\rep}_{\AC{\bQ}_p})$ \Pth{as in Definition \ref{def-Sh-var-HT-wts}} depend only on the $\bC$-dimension of \Refeq{\ref{eq-Sh-var-HT-decomp-dual-BGG-wt}}, but not on the choice of $v$, implies that the \Pth{usual} Hodge--Tate weights of $H^i_{\et, ?}(\Model_{\levcp, \AC{\bQ}_p}, \etSh{\rep}_{\bQ_p})$ \Pth{as a representation of $\Gal(\AC{\bQ}_p / \BFp)$} are also independent of $v$.
\end{rk}

\subsection{Some application to intersection cohomology}\label{sec-Sh-var-IH}

In this final subsection, let us discuss an important special case where we can apply our results to the \emph{intersection cohomology} of Shimura varieties with nontrivial coefficients, simply because it coincides with the interior cohomology.

Let us begin with some review of definitions.  Consider the interior cohomology
\begin{equation}\label{eq-def-int-coh-B}
    H_\intcoh^i(\Model_{\levcp, \bC}^\an, \BSh{\rep}_\bC) := \Image\bigl(H_\cpt^i(\Model_{\levcp, \bC}^\an, \BSh{\rep}_\bC) \to H^i(\Model_{\levcp, \bC}^\an, \BSh{\rep}_\bC)\bigr),
\end{equation}
as usual.  By \cite[XI, 4.4, and XVII, 5.3.3 and 5.3.5]{SGA:4} and \cite[\aSec 6]{Beilinson/Bernstein/Deligne/Gabber:2018-FP(2)}, for $? = \emptyset$, $\cpt$, and $\intcoh$, we have canonical Hecke-equivariant isomorphisms
\begin{equation}\label{eq-B-et-comp}
    H_?^i(\Model_{\levcp, \bC}^\an, \BSh{\rep}_\bC) \cong H_{\et, ?}^i(\Model_{\levcp, \AC{\bQ}_p}, \etSh{\rep}_{\AC{\bQ}_p}) \otimes_{\AC{\bQ}_p, \ACMap} \bC
\end{equation}
compatible with each other and with Poincar\'e duality.  Also, by \cite[II, 6]{Deligne:1970-EDR}, \cite[\aSec 2.11 and \aCor 2.12]{Esnault/Viehweg:1992-LVT-B}, and GAGA \cite{Serre:1955-1956-gaga}, for $? = \emptyset$, $\cpt$, and $\intcoh$, we have canonical Hecke-equivariant isomorphisms
\begin{equation}\label{eq-B-dR-comp}
    H_?^i(\Model_{\levcp, \bC}^\an, \BSh{\rep}_\bC) \cong H_{\dR, ?}^i(\Model_{\levcp, \bC}^\an, \dRSh{\rep}_\bC^\an) \cong H_{\dR, ?}^i(\Model_{\levcp, \bC}, \dRSh{\rep}_\bC)
\end{equation}
compatible with each other and with Poincar\'e duality.  By the same argument as in the proof of Theorem \ref{thm-int-coh-comp}, by using the degeneration of Hodge--de Rham spectral sequences on the $E_1$ pages, the Hodge filtrations on $H_{\dR, \cpt}^i(\Model_{\levcp, \bC}, \dRSh{\rep}_\bC)$ and $H_\dR^i(\Model_{\levcp, \bC}, \dRSh{\rep}_\bC)$ are strictly compatible with the canonical morphism between them, and induce the same Hodge filtration on $H_{\dR, \intcoh}^i(\Model_{\levcp, \bC}, \dRSh{\rep}_\bC)$.

Let $\Mincpt{\Model}_\levcp$ denote the minimal compactification of $\Model_\levcp$ over $\ReFl$, as in \cite{Pink:1989-Ph-D-Thesis}, and let $\Mincpt{\Model}_{\levcp, \bC}$ and $\Mincpt{\Model}_{\levcp, \AC{\bQ}_p}$ denote the pullbacks of $\Mincpt{\Model}_\levcp$ to $\bC$ and $\AC{\bQ}_p$, respectively.  Let $\Mincpt{\jmath}: \Model_\levcp \to \Mincpt{\Model}_\levcp$ denote the canonical open immersion, whose pullbacks to $\bC$ and $\AC{\bQ}_p$ we shall denote by the same symbols, for simplicity.  Let $d := \dim(\Model_\levcp)$.  For each $\rep \in \Rep_{\AC{\bQ}}(\Grp{G}^c)$, by abuse of notation, consider the \emph{intersection cohomology}
\begin{equation}\label{eq-IC-B}
    \IH^i\bigl(\Model_{\levcp, \bC}^{\Min, \an}, \BSh{\rep}_\bC\bigr) := H^{i - d}\bigl(\Model_{\levcp, \bC}^{\Min, \an}, \jmath_{!*}^{\Min, \an}(\BSh{\rep}_\bC[d])\bigr)
\end{equation}
and
\begin{equation}\label{eq-IC-et}
    \IH_\et^i\bigl(\Mincpt{\Model}_{\levcp, \AC{\bQ}_p}, \etSh{\rep}_{\AC{\bQ}_p}\bigr) := H^{i - d}\bigl(\Mincpt{\Model}_{\levcp, \AC{\bQ}_p}, \Mincpt{\jmath}_{\et, !*}(\etSh{\rep}_{\AC{\bQ}_p}[d])\bigr).
\end{equation}
By \cite[\aSec 6]{Beilinson/Bernstein/Deligne/Gabber:2018-FP(2)}, we have a canonical Hecke-equivariant isomorphism
\begin{equation}\label{eq-IC-B-et-comp}
    \IH^i\bigl(\Model_{\levcp, \bC}^{\Min, \an}, \BSh{\rep}_\bC\bigr) \cong \IH_\et^i\bigl(\Mincpt{\Model}_{\levcp, \AC{\bQ}_p}, \etSh{\rep}_{\AC{\bQ}_p}\bigr) \otimes_{\AC{\bQ}_p, \ACMap} \bC,
\end{equation}
which is compatible with \Refeq{\ref{eq-B-et-comp}} via canonical morphisms, and with Poincar\'e duality.  By Zucker's conjecture \cite{Zucker:1982-lcwpa}, which has been proved \Pth{independently} by Looijenga \cite{Looijenga:1988-lclsv}; Saper and Stern \cite{Saper/Stern:1990-l2cav}; and Looijenga and Rapoport \cite{Looijenga/Rapoport:1991-wlcbc}, we have a canonical Hecke-equivariant isomorphism
\begin{equation}\label{eq-Zucker}
    \IH^i(\Model_{\levcp, \bC}^{\Min, \an}, \BSh{\rep}_\bC) \cong H_{(2)}^i(\Model_{\levcp, \bC}^\an, \BSh{\rep}_\bC),
\end{equation}
where $H_{(2)}^i\big(\Model_{\levcp, \bC}^\an, \BSh{\rep}_\bC\big)$ denotes the \emph{$L^2$-cohomology}, as in \cite[\aCh XIV, \aSec 3]{Borel/Wallach:2000-CDR(2)}, which is compatible with \Refeq{\ref{eq-B-dR-comp}} via canonical morphisms.

The left-hand side of \Refeq{\ref{eq-Zucker}} is equipped with the Hodge structure given by Saito's theory of Hodge modules \Pth{see \cite{Saito:1988-mhp}}, while the right-hand side of \Refeq{\ref{eq-Zucker}} is equipped with the Hodge structure given by $L^2$ harmonic forms \Pth{which can be refined by a double dual BGG decomposition, as in \cite[\aSec 6]{Faltings:1983-clshs}}.  But it is unclear whether these two Hodge structures are compatible under the isomorphism \Refeq{\ref{eq-Zucker}} \Pth{\Refcf{} \cite[\aConj 5.3]{Harris/Zucker:2001-BCS-3}}.  Nevertheless, the following is known:
\begin{thm}[{Harris and Zucker; see \cite[\aThm 5.4]{Harris/Zucker:2001-BCS-3}}]\label{thm-Harris-Zucker}
    The canonical morphisms from both sides of \Refeq{\ref{eq-Zucker}} to $H_\dR^i(\Model_{\levcp, \bC}, \dRSh{\rep}_\bC)$ are strictly compatible with Hodge filtrations.  In particular, the Hodge filtrations on both sides of \Refeq{\ref{eq-Zucker}} induce the same Hodge structure on their common image in $H_\dR^i(\Model_{\levcp, \bC}, \dRSh{\rep}_\bC)$.
\end{thm}

In general, we have Hecke-equivariant inclusions
\begin{equation}\label{eq-H-cusp-H-int-H-2}
    H_{\Utext{cusp}}^i(\Model_{\levcp, \bC}^\an, \BSh{\rep}_\bC) \subset H_\intcoh^i(\Model_{\levcp, \bC}^\an, \BSh{\rep}_\bC) \subset H_{(2)}^i(\Model_{\levcp, \bC}^\an, \BSh{\rep}_\bC),
\end{equation}
where $H_{\Utext{cusp}}^i(\Model_{\levcp, \bC}^\an, \BSh{\rep}_\bC)$ denote the \emph{cuspidal cohomology} \Pth{see \cite{Borel:1974-srcag, Borel:1981-srcag-2}}, which are compatible with \Refeq{\ref{eq-B-dR-comp}} and \Refeq{\ref{eq-Zucker}} via canonical morphisms.

We have the following useful results:
\begin{thm}[{Schwermer; see \cite[\aCor 2.3]{Schwermer:1994-escag}}]\label{thm-Schwermer}
    Suppose that $\rep \cong \dual{\rep}_\wt$ for some $\wt \in \WT_{\Grp{G}^c_{\AC{\bQ}}}^+$ that is \emph{regular} in the sense that $(\wt, \cort) > 0$ for every simple root $\rt \in \RT_{\Grp{G}^c_{\AC{\bQ}}}^+$.  Then all the containments in \Refeq{\ref{eq-H-cusp-H-int-H-2}} are equalities.
\end{thm}

\begin{thm}[{Li and Schwermer; see \cite[\aCor 5.6]{Li/Schwermer:2004-ecag}}]\label{thm-Li-Schwermer}
    Suppose that $\rep \cong \dual{\rep}_\wt$ for some \emph{regular} $\wt \in \WT_{\Grp{G}^c_{\AC{\bQ}}}^+$.  Then $H_\cpt^i(\Model_{\levcp, \bC}^\an, \BSh{\rep}_\bC) = 0$ for $i > d$; $H^i(\Model_{\levcp, \bC}^\an, \BSh{\rep}_\bC) = 0$ for $i < d$; and $H_\intcoh^i(\Model_{\levcp, \bC}^\an, \BSh{\rep}_\bC) = 0$ for $i \neq d$.
\end{thm}

\begin{cor}\label{cor-IH}
    Suppose that $\rep \cong \dual{\rep}_\wt$ for some \emph{regular} $\wt \in \WT_{\Grp{G}^c_{\AC{\bQ}}}^+$.  Then we have canonical Hecke-equivariant isomorphisms
    \begin{equation}\label{eq-cor-IH-B}
        \IH^i(\Model_{\levcp, \bC}^{\Min, \an}, \BSh{\rep}_\bC) \cong H_\intcoh^i(\Model_{\levcp, \bC}^\an, \BSh{\rep}_\bC) \cong H_{\dR, \intcoh}^i(\Model_{\levcp, \bC}, \dRSh{\rep}_\bC)
    \end{equation}
    compatible with Hodge filtrations and Poincar\'e duality, and also a canonical Hecke-equivariant isomorphism
    \begin{equation}\label{eq-cor-IH-et}
        \IH_\et^i(\Mincpt{\Model}_{\levcp, \bC}, \etSh{\rep}_{\AC{\bQ}_p}) \cong H_{\et, \intcoh}^i(\Model_{\levcp, \bC}, \etSh{\rep}_{\AC{\bQ}_p}).
    \end{equation}
    The cohomology in either \Refeq{\ref{eq-cor-IH-B}} and \Refeq{\ref{eq-cor-IH-et}} can be nonzero only when $i = d$.
\end{cor}
\begin{proof}
    These follow from \Refeq{\ref{eq-B-et-comp}} and \Refeq{\ref{eq-B-dR-comp}}; from Theorems \ref{thm-Harris-Zucker}, \ref{thm-Schwermer}, and \ref{thm-Li-Schwermer}; and from the compatibility of the Poincar\'e duality on intersection cohomology with the usual one.
\end{proof}

\begin{thm}\label{thm-Sh-var-IH}
    Suppose that $\rep \cong \dual{\rep}_\wt$ for some \emph{regular} $\wt \in \WT_{\Grp{G}^c_{\AC{\bQ}}}^+$.  Let $\BFp$ be a finite extension of the image of $\ReFl \Emn{\can} \AC{\bQ} \Emn{~~\ACMap^{-1}} \AC{\bQ}_p$ over which $\rep_{\AC{\bQ}_p}$ has a model.  Then we have a canonical $\Gal(\AC{\bQ}_p / \BFp)$-equivariant Hecke-equivariant isomorphism
    \begin{equation}\label{eq-Sh-var-comp-dR-IH}
        \IH_\et^d(\Mincpt{\Model}_{\levcp, \AC{\bQ}_p}, \etSh{\rep}_{\AC{\bQ}_p}) \otimes_{\AC{\bQ}_p} \BdR \cong H_{\dR, \intcoh}^d(\Model_{\levcp, \AC{\bQ}_p}, \dRSh{\rep}_{\AC{\bQ}_p}) \otimes_{\AC{\bQ}_p} \BdR,
    \end{equation}
    which is compatible with the filtrations on both sides and with Poincar\'e duality, whose $0$-th graded piece can be refined by a canonical $\Gal(\AC{\bQ}_p / \BFp)$-equivariant Hecke-equivariant \emph{dual BGG decomposition}
    \begin{equation}\label{eq-Sh-var-comp-HT-IH}
    \begin{split}
        & \IH_\et^d(\Mincpt{\Model}_{\levcp, \AC{\bQ}_p}, \etSh{\rep}_{\AC{\bQ}_p}) \otimes_{\AC{\bQ}_p} \bC_p \\
        & \cong \oplus_{w\in \WG^{\Grp{M}^c}} \, \Bigl(H_\intcoh^{d - \wl(w)}\bigl(\Torcpt{\Model}_{\levcp, \AC{\bQ}_p}, (\cohSh{\repalt}_{w \cdot \wt, \AC{\bQ}_p}^\dualsign)^\canext\bigr) \otimes_{\AC{\bQ}_p} \bC_p\bigl((w \cdot \wt)(H)\bigr)\Bigr),
    \end{split}
    \end{equation}
    compatible with Poincar\'e and Serre duality.  The multiset of Hodge--Tate weights of any Hecke-invariant $\AC{\bQ}_p$-subspace of $\IH_\et^d(\Mincpt{\Model}_{\levcp, \AC{\bQ}_p}, \etSh{\rep}_{\AC{\bQ}_p})$ cut out by some $\AC{\bQ}_p$-valued Hecke operator \Pth{as in Definition \ref{def-Sh-var-HT-wts}} contains each $a \in \bZ$ with multiplicity given by the $\bC$-dimension of the corresponding Hecke-invariant $\bC$-subspace of
    \begin{equation}\label{eq-Sh-var-comp-IH-HT-wt}
        \oplus_{w\in \WG^{\Grp{M}^c}, \; (w \cdot \wt)(H) = -a} \; H_\intcoh^{d - \wl(w)}\bigl(\Torcpt{\Model}_{\levcp, \bC}, (\cohSh{\repalt}_{w \cdot \wt, \bC}^\dualsign)^\canext\bigr)
    \end{equation}
    cut out by the pullback of the same $\AC{\bQ}_p$-valued Hecke operator under $\ACMap: \AC{\bQ}_p \Mi \bC$.
\end{thm}
\begin{proof}
    These follow from Theorems \ref{thm-Sh-var-comp-dR-HT} and \ref{thm-HT-wts}, and from Corollary \ref{cor-IH}.
\end{proof}

\begin{rk}\label{rem-Sh-var-IH}
    We natural expect the de Rham comparison to work for the intersection cohomology in more generality, which will be an interesting topic for a future project.  But we would like to record the results in Theorem \ref{thm-Sh-var-IH} because regular weights already cover, depending on one's viewpoint, almost all weights.
\end{rk}



\begin{thebibliography}{CDHN21}

\bibitem[AB01]{Andre/Baldassarri:2001-DDA}
Y.~Andr{\'e} and F.~Baldassarri, \emph{De {Rham} cohomology of differential
  modules on algebraic varieties}, Progress in Mathematics, vol. 189,
  Birkh{\"a}user, Basel, Boston, Berlin, 2001.

\bibitem[ABC20]{Andre/Baldassarri/Cailotto:2020-DDA(2)}
Y.~Andr{\'e}, F.~Baldassarri, and M.~Cailotto, \emph{De {Rham} cohomology of
  differential modules on algebraic varieties}, 2nd ed., Progress in
  Mathematics, vol. 189, Birkh{\"a}user, Basel, Boston, Berlin, 2020.

\bibitem[AGV73]{SGA:4}
M.~Artin, A.~Grothendieck, and J.-L. Verdier (eds.), \emph{Th{\'e}orie des
  topos et cohomologie {\'e}tale des sch{\'e}mas \parenthesis{{SGA} 4}},
  Lecture Notes in Mathematics, vol. 269, 270, 305, Springer-Verlag, Berlin,
  Heidelberg, New York, 1972, 1972, 1973.

\bibitem[BBDG18]{Beilinson/Bernstein/Deligne/Gabber:2018-FP(2)}
A.~Beilinson, J.~Bernstein, P.~Deligne, and O.~Gabber, \emph{Faisceaux
  pervers}, 2nd ed., Ast{\'e}risque, vol. 100, Soci{\'e}t{\'e} Math{\'e}matique
  de France, Paris, 2018.

\bibitem[Bey97]{Beyer:1997-sdcsr}
P.~Beyer, \emph{On {Serre}-{Duality} for coherent sheaves on rigid-analytic
  spaces}, Manuscripta Math. \textbf{93} (1997), no.~2, 219--245.

\bibitem[BGG75]{Bernstein/Gelfand/Gelfand:1975-doagm}
I.~N. Bernstein, I.~M. Gelfand, and S.~I. Gelfand, \emph{Differential operators
  on the base affine space and a study of {$\mathfrak{g}$}-modules}, Lie Groups
  and Their Representations (I.~M. Gelfand, ed.), Adam Hilger Ltd., London,
  1975, pp.~21--64.

\bibitem[BM97]{Bierstone/Milman:1997-cdbml}
E.~Bierstone and P.~D. Milman, \emph{Canonical desingularization in
  characteristic zero by blowing up the maximum strata of a local invariant},
  Invent. Math. \textbf{128} (1997), no.~2, 207--302.

\bibitem[Bor74]{Borel:1974-srcag}
A.~Borel, \emph{Stable real cohomology of arithmetic groups}, Ann. Sci. Ecole
  Norm. Sup. (4) \textbf{7} (1974), 235--272.

\bibitem[Bor81]{Borel:1981-srcag-2}
\bysame, \emph{Stable real cohomology of arithmetic groups, {II}}, Manifolds
  and {Lie} groups. {Papers} in honor of {Yoz\^{o}} {Matsushima} (J.~Hano,
  A.~Morimoto, S.~Murakami, K.~Okamoto, and H.~Ozeki, eds.), Progress in
  Mathematics, vol.~14, Birkh{\"a}user, Boston, 1981, pp.~21--55.

\bibitem[BS98]{Brodmann/Sharp:1998-LC}
M.~P. Bordmann and R.~Y. Sharp, \emph{Local cohomology: An algebraic
  introduction with geometric applications}, Cambridge Studies in Advanced
  Mathematics, vol.~60, Cambridge University Press, Cambridge, New York, 1998.

\bibitem[BW00]{Borel/Wallach:2000-CDR(2)}
A.~Borel and N.~Wallach, \emph{Continuous cohomology, discrete subgroups, and
  representations of reductive groups}, 2nd ed., Mathematical Surveys and
  Monographs, vol.~67, American Mathematical Society, Providence, Rhode Island,
  2000.

\bibitem[CDHN21]{Colmez/Dospinescu/Hauseux/Niziol:2021-pecpd}
P.~Colmez, G.~Dospinescu, J.~Hauseux, and W.~Nizio{\l}, \emph{{$p$}-adic
  {\'e}tale cohomology of period domains}, Math. Ann. \textbf{381} (2021),
  105--180.

\bibitem[Con99]{Conrad:1999-icrs}
B.~Conrad, \emph{Irreducible components of rigid spaces}, Ann. Inst. Fourier.
  Grenoble \textbf{49} (1999), no.~2, 473--541.

\bibitem[Con06]{Conrad:2006-rarg}
\bysame, \emph{Relative ampleness in rigid geometry}, Ann. Inst. Fourier.
  Grenoble \textbf{56} (2006), no.~4, 1049--1126.

\bibitem[Del70]{Deligne:1970-EDR}
P.~Deligne, \emph{Equations diff{\'e}rentielles {\`a} points singuliers
  r{\'e}guliers}, Lecture Notes in Mathematics, vol. 163, Springer-Verlag,
  Berlin, Heidelberg, New York, 1970.

\bibitem[DI87]{Deligne/Illusie:1987-rdcdr}
P.~Deligne and L.~Illusie, \emph{Rel{\`e}vements modulo {$p^2$} et
  d{\'e}compositions du complex de {de Rham}}, Invent. Math. \textbf{89}
  (1987), 247--270.

\bibitem[DLLZa]{Diao/Lan/Liu/Zhu:lasfr}
H.~Diao, K.-W. Lan, R.~Liu, and X.~Zhu, \emph{Logarithmic adic spaces: some
  foundational results}, arXiv:1912.09836.

\bibitem[DLLZb]{Diao/Lan/Liu/Zhu:lrhrv}
\bysame, \emph{Logarithmic {Riemann}--{Hilbert} correspondences for rigid
  varieties}, J. Amer. Math. Soc., to appear.

\bibitem[EV92]{Esnault/Viehweg:1992-LVT-B}
H.~Esnault and E.~Viehweg, \emph{Lectures on vanishing theorems}, DMV Seminar,
  vol.~20, Birkh{\"a}user Verlag, Basel, 1992.

\bibitem[Fal83]{Faltings:1983-clshs}
G.~Faltings, \emph{On the cohomology of locally symmetric hermitian spaces},
  S{\'e}minaire d'Alg{\`e}bre {Paul} {Dubreil} et {Marie}--{Paule} {Malliavin}
  (M.-P. Malliavin, ed.), Lecture Notes in Mathematics, vol. 1029,
  Springer-Verlag, Berlin, Heidelberg, New York, 1983, pp.~55--98.

\bibitem[Fal89]{Faltings:1989-ccpgr}
\bysame, \emph{Crystalline cohomoloy and {$p$}-adic {Galois}-representations},
  Algebraic Analysis, Geometry, and Number Theory (J.-I. Igusa, ed.), The Johns
  Hopkins University Press, Baltimore, 1989, pp.~25--80.

\bibitem[Fal02]{Faltings:2002-aee}
\bysame, \emph{Almost \'etale extensions}, Cohomologies {$p$}-adiques et
  applications arithm{\'e}tiques {(II)} (P.~Berthelot, J.-M. Fontaine,
  L.~Illusie, K.~Kato, and M.~Rapoport, eds.), Ast{\'e}risque, no. 279,
  Soci{\'e}t{\'e} Math{\'e}matique de France, Paris, 2002, pp.~185--270.

\bibitem[FC90]{Faltings/Chai:1990-DAV}
G.~Faltings and C.-L. Chai, \emph{Degeneration of abelian varieties},
  Ergebnisse der Mathematik und ihrer Grenzgebiete, 3. Folge, vol.~22,
  Springer-Verlag, Berlin, Heidelberg, New York, 1990.

\bibitem[Har66]{Hartshorne:1966-RD}
R.~Hartshorne, \emph{Residues and duality}, Lecture Notes in Mathematics,
  vol.~20, Springer-Verlag, Berlin, Heidelberg, New York, 1966.

\bibitem[Har67]{Hartshorne:1967-LC}
\bysame, \emph{Local cohomology}, Lecture Notes in Mathematics, vol.~41,
  Springer-Verlag, Berlin, Heidelberg, New York, 1967, a seminar given by
  A.~Grothendieck, Harvard University, Fall, 1961.

\bibitem[Har85]{Harris:1985-avafs-1}
M.~Harris, \emph{Arithmetic vector bundles and automorphic forms on {Shimura}
  varieties. {I}}, Invent. Math. \textbf{82} (1985), 151--189.

\bibitem[Har89]{Harris:1989-ftcls}
\bysame, \emph{Functorial properties of toroidal compactifications of locally
  symmetric varieties}, Proc. London Math. Soc. (3) \textbf{59} (1989), 1--22.

\bibitem[Har90]{Harris:1990-afcvs}
\bysame, \emph{Automorphic forms and the cohomology of vector bundles on
  {Shimura} varieties}, Automorphic Forms, {Shimura} Varieties, and
  {$L$}-Functions. {Volume} {II} (L.~Clozel and J.~S. Milne, eds.),
  Perspectives in Mathematics, vol.~11, Academic Press Inc., Boston, 1990,
  pp.~41--91.

\bibitem[Hir64a]{Hironaka:1964-rsavz-1}
H.~Hironaka, \emph{Resolution of singularities of an algebraic variety over a
  field of characteristic zero: {I}}, Ann. Math. (2) \textbf{79} (1964), no.~1,
  109--203.

\bibitem[Hir64b]{Hironaka:1964-rsavz-2}
\bysame, \emph{Resolution of singularities of an algebraic variety over a field
  of characteristic zero: {II}}, Ann. Math. (2) \textbf{79} (1964), no.~2,
  205--326.

\bibitem[HLTT16]{Harris/Lan/Taylor/Thorne:2016-rccsv}
M.~Harris, K.-W. Lan, R.~Taylor, and J.~Thorne, \emph{On the rigid cohomology
  of certain {Shimura} varieties}, Res. Math. Sci. \textbf{3} (2016), article
  no.~37, 308 pp.

\bibitem[HT01]{Harris/Taylor:2001-GCS}
M.~Harris and R.~Taylor, \emph{The geometry and cohomology of some simple
  {Shimura} varieties}, Annals of Mathematics Studies, vol. 151, Princeton
  University Press, Princeton, 2001.

\bibitem[Hub96]{Huber:1996-ERA}
R.~Huber, \emph{{\'E}tale cohomology of rigid analytic varieties and adic
  spaces}, Aspects of Mathematics, vol. E30, Friedr. Vieweg \& Sohn,
  Braunschweig/Wiesbaden, 1996.

\bibitem[Hub98]{Huber:1998-ctlac}
\bysame, \emph{A comparison theorem for {$\ell$}-adic cohomology}, Compositio
  Math. \textbf{112} (1998), no.~2, 217--235.

\bibitem[HZ01]{Harris/Zucker:2001-BCS-3}
M.~Harris and S.~Zucker, \emph{Boundary cohomology of {Shimura} varieties
  {III}. {Coherent} cohomology on higher-rank boundary strata and applications
  to {Hodge} theory}, M{\'e}moires de la Soci{\'e}t{\'e} Math{\'e}matique de
  France. Nouvelle S{\'e}rie, vol.~85, Soci{\'e}t{\'e} Math{\'e}matique de
  France, Paris, 2001.

\bibitem[Jan88]{Jannsen:1988-cec}
U.~Jannsen, \emph{Continuous {\'e}tale cohomology}, Math. Ann. \textbf{280}
  (1988), 207--245.

\bibitem[K{\"o}p74]{Kopf:1974-efava}
U.~K{\"o}pf, \emph{{\"U}ber eigentliche {Familien} algebraischer
  {Variet{\"a}ten} {\"u}ber affinoiden {R{\"a}umen}}, Schr. Math. Inst. Univ.
  M{\"u}nster (2) \textbf{7} (1974), iv+72.

\bibitem[KS90]{Kashiwara/Shapira:1990-SM}
M.~Kashiwara and P.~Schapira, \emph{Sheaves on manifolds}, Grundlehren der
  mathematischen Wissenschaften, vol. 292, Springer-Verlag, Berlin, Heidelberg,
  New York, 1990, with a short history \emph{Les debuts de la theorie des
  faisceaux} by Christian Houzel.

\bibitem[Lan12a]{Lan:2012-aatcs}
K.-W. Lan, \emph{Comparison between analytic and algebraic constructions of
  toroidal compactifications of {PEL}-type {Shimura} varieties}, J. Reine
  Angew. Math. \textbf{664} (2012), 163--228.

\bibitem[Lan12b]{Lan:2012-tckf}
\bysame, \emph{Toroidal compactifications of {PEL}-type {Kuga} families},
  Algebra Number Theory \textbf{6} (2012), no.~5, 885--966.

\bibitem[Lan16]{Lan:2016-vtcac}
\bysame, \emph{Vanishing theorems for coherent automorphic cohomology}, Res.
  Math. Sci. \textbf{3} (2016), article no.~39, 43 pp.

\bibitem[Loo88]{Looijenga:1988-lclsv}
E.~Looijenga, \emph{{$L^2$}-cohomology of locally symmetric varieties},
  Compositio Math. \textbf{67} (1988), 3--20.

\bibitem[LP18]{Lan/Polo:2018-bggab}
K.-W. Lan and P.~Polo, \emph{Dual {BGG} complexes for automorphic bundles},
  Math. Res. Lett. \textbf{25} (2018), no.~1, 85--141.

\bibitem[LR91]{Looijenga/Rapoport:1991-wlcbc}
E.~Looijenga and M.~Rapoport, \emph{Weights in the local cohomology of a
  {Baily}--{Borel} compactification}, Complex Geometry and {Lie} Theory (J.~A.
  Carlson, C.~H. Clemens, and D.~R. Morrison, eds.), Proceedings of Symposia in
  Pure Mathematics, vol.~53, Proceedings of a Symposium on Complex Geometry and
  Lie Theory held at Sundance, Utah, May 26--30, 1989, Springer-Verlag, Berlin,
  Heidelberg, New York, 1991, pp.~223--260.

\bibitem[LS04]{Li/Schwermer:2004-ecag}
J.-S. Li and J.~Schwermer, \emph{On the {Eisenstein} cohomology of arithmetic
  groups}, Duke Math. J. \textbf{123} (2004), no.~1, 141--169.

\bibitem[LS18a]{Lan/Stroh:2018-csisv}
K.-W. Lan and B.~Stroh, \emph{Compactifications of subschemes of integral
  models of {Shimura} varieties}, Forum Math. Sigma \textbf{6} (2018), e18,
  105.

\bibitem[LS18b]{Lan/Stroh:2018-ncaes}
\bysame, \emph{Nearby cycles of automorphic \'etale sheaves}, Compos. Math.
  \textbf{154} (2018), no.~1, 80--119.

\bibitem[LZ17]{Liu/Zhu:2017-rrhpl}
R.~Liu and X.~Zhu, \emph{Rigidity and a {Riemann}--{Hilbert} correspondence for
  {$p$}-adic local systems}, Invent. Math. \textbf{207} (2017), 291--343.

\bibitem[Mil90]{Milne:1990-cmsab}
J.~S. Milne, \emph{Canonical models of (mixed) {Shimura} varieties and
  automorphic vector bundles}, Automorphic Forms, {Shimura} Varieties, and
  {$L$}-Functions. {Volume} {I} (L.~Clozel and J.~S. Milne, eds.), Perspectives
  in Mathematics, vol.~10, Academic Press Inc., Boston, 1990, pp.~283--414.

\bibitem[Mum77]{Mumford:1977-hptnc}
D.~Mumford, \emph{Hirzebruch's proportionality theorem in the non-compact
  case}, Invent. Math. \textbf{42} (1977), 239--272.

\bibitem[Nag62]{Nagata:1962-iavcv}
M.~Nagata, \emph{Imbedding of an abstract variety in a complete variety}, J.
  Math. Kyoto Univ. \textbf{2} (1962), 1--10.

\bibitem[Pin89]{Pink:1989-Ph-D-Thesis}
R.~Pink, \emph{Arithmetic compactification of mixed {Shimura} varieties}, Ph.D.
  thesis, Rheinischen Friedrich-Wilhelms-Universit{\"a}t, Bonn, 1989.

\bibitem[Sai88]{Saito:1988-mhp}
M.~Saito, \emph{Modules de {Hodge} polarisables}, Publ. Res. Inst. Math. Sci.
  \textbf{24} (1988), no.~6, 849--995.

\bibitem[Sai90]{Saito:1990-mhm}
\bysame, \emph{Mixed {Hodge} modules}, Publ. Res. Inst. Math. Sci. \textbf{26}
  (1990), no.~2, 221--333.

\bibitem[Sch94]{Schwermer:1994-escag}
J.~Schwermer, \emph{Eisenstein series and cohomology of arithmetic groups:
  {The} generic case}, Invent. Math. \textbf{116} (1994), 481--511.

\bibitem[Sch12]{Scholze:2012-ps}
P.~Scholze, \emph{Perfectoid spaces}, Publ. Math. Inst. Hautes {\'E}tud. Sci.
  \textbf{116} (2012), 245--313.

\bibitem[Sch13]{Scholze:2013-phtra}
\bysame, \emph{{$p$}-adic {Hodge} theory for rigid-analytic varieties}, Forum
  Math. Pi \textbf{1} (2013), e1, 77.

\bibitem[Ser56]{Serre:1955-1956-gaga}
J.-P. Serre, \emph{Geom{\'e}trie alg{\'e}brique et g{\'e}om{\'e}trie
  analytique}, Ann. Inst. Fourier. Grenoble \textbf{6} (1955--1956), 1--42.

\bibitem[SS90]{Saper/Stern:1990-l2cav}
L.~Saper and M.~Stern, \emph{{$L_2$}-cohomology of arithmetic varieties}, Ann.
  Math. (2) \textbf{132} (1990), no.~1, 1--69.

\bibitem[Su18]{Su:2018-ccsva}
J.~Su, \emph{Coherent cohomology of {Shimura} varieties and automorphic forms},
  preprint, 2018.

\bibitem[SW20]{Scholze/Weinstein:2020-BLG}
P.~Scholze and J.~Weinstein, \emph{Berkeley lectures on {$p$}-adic geometry},
  Annals of Mathematics Studies, vol. 207, Princeton University Press,
  Princeton, 2020.

\bibitem[vdP92]{vanderPut:1992-sdras}
M.~van~der Put, \emph{Serre duality for rigid analytic spaces}, Indag. Math.
  (N.S) \textbf{3} (1992), no.~2, 219--235.

\bibitem[Zuc82]{Zucker:1982-lcwpa}
S.~Zucker, \emph{{$L_2$} cohomology of warpped products and arithmetic groups},
  Invent. Math. \textbf{70} (1982), no.~2, 169--218.

\end{thebibliography}

\providecommand{\bysame}{\leavevmode\hbox to3em{\hrulefill}\thinspace}
\providecommand{\MR}{\relax\ifhmode\unskip\space\fi MR }
\providecommand{\MRhref}[2]{%
  \href{http://www.ams.org/mathscinet-getitem?mr=#1}{#2}
}
\providecommand{\href}[2]{#2}

\end{document}